\documentclass[11pt]{amsart}

\usepackage[a4paper,margin=1in]{geometry}
\usepackage{amsmath,amssymb,amsfonts,amsthm,amsrefs,bbm}
\usepackage{mathrsfs}
\usepackage{hyperref}
\hypersetup{colorlinks=true,linkcolor=blue,citecolor=blue,urlcolor=blue}
\usepackage{enumitem}
\usepackage{mathtools}
\mathtoolsset{showonlyrefs}

\newtheorem{theorem}{Theorem}[section]
\newtheorem{lemma}[theorem]{Lemma}
\newtheorem{proposition}[theorem]{Proposition}
\newtheorem{corollary}[theorem]{Corollary}
\theoremstyle{definition}
\newtheorem{definition}[theorem]{Definition}
\newtheorem{notation}[theorem]{Notation}
\newtheorem{observation}[theorem]{Observation}
\theoremstyle{remark}
\newtheorem{remark}[theorem]{Remark}

\newcommand{\R}{\mathbb R}
\newcommand{\N}{\mathbb N}
\newcommand{\Q}{\mathbb Q}

\newcommand{\C}{\mathbb C}

\newcommand{\PP}{\mathbb P}
\newcommand{\EE}{\mathbb E}
\newcommand{\1}{\mathbbm{1}}
\newcommand{\id}{\mathrm{id}}

\newcommand{\ind}{\mathrm{ind}}
\newcommand{\ip}[2]{\langle #1,#2\rangle}
\newcommand{\ev}{\operatorname{ev}}
\newcommand{\Law}{\mathrm{Law}}

\newcommand{\diam}{\mathrm{diam}}

\title[Factorizing random sets and type~III Arveson systems]
{Factorizing random sets and type~III Arveson systems}
\author{Remus Floricel}
\date{\today}
\address{University of Regina, Department of Mathematics and Statistics, Regina, SK, Canada}
\email{Remus.Floricel@uregina.ca} 
\thanks {This research was partially funded by a Discovery Grant from NSERC }
\begin{document}

\maketitle
\begin{abstract}
We develop a representative-level framework for the Liebscher–Tsirelson
random-set construction of Arveson systems from stationary
factorizing measure types.

We introduce the notion of a measurable factorizing family of probability
measures on hyperspaces of closed subsets of time intervals and prove that
every such family canonically generates an Arveson system.
Within this framework we obtain a purely measure-theoretic characterization
of spatiality: positive normalized units correspond exactly to dominated
families of measures that factorize strictly.

We then present a general mechanism for constructing type~III Arveson systems via infinite products of measurable factorizing families.
Starting from a type~II$_0$ seed satisfying a quantitative
Hellinger-smallness condition, we form a marked infinite product indexed by
$[0,1]\times\mathbb N$ and show, using Kakutani’s criterion, that the
resulting product system admits no units.
This yields a robust construction principle for type~III random-set systems.

As an application we analyze zero sets of Brownian motion.
After anchor-adapted localization and Palm-type uniformization, the
Brownian seed satisfies the required overlap estimates, and the associated
infinite-product construction produces explicit examples of type~III
random-set systems, as anticipated in the work of
Tsirelson and Liebscher.
\end{abstract}

{\small\tableofcontents}

\section{Introduction}

Product systems of Hilbert spaces, also known as Arveson systems, were introduced by W. Arveson in \cite{Arveson89} as a complete invariant for $E_0$-semigroups of endomorphisms of $\mathcal{B}(H)$ \cite{Powers}. In his foundational work, Arveson established a type classification of product systems—types I, II, and III—paralleling the Murray–von Neumann classification of $W^*$-factors \cite{Dixmier}. This classification remains central to the structure theory of noncommutative dynamics and the analysis of $E_0$-semigroups \cite{arveson-continuous}.

Following Arveson’s work, a major problem emerged: to construct and understand Arveson systems beyond the completely spatial (type I) case. A decisive breakthrough was achieved by B. Tsirelson \cite{Tsirelson-Vershik, Tsirelson-slightly-coloured-unitless, Tsirelson-spec, Tsirelson00, Tsirelson2004, Tsirelson2004a}, who revealed the deep probabilistic nature of Arveson systems. In particular, he showed that Arveson systems can be constructed from random closed sets and related stochastic structures, demonstrating that Arveson systems may encode subtle measure-theoretic structure. A landmark example is his construction of a type~II$_0$ Arveson system arising from the random zero set of Brownian motion \cite{Tsirelson00}. This work also established that Arveson systems need not arise from Fock space constructions (see also \cite{Tsirelson-Vershik}) and that probabilistic input can generate genuinely new types.

Tsirelson’s construction was subsequently developed and systematized by V. Liebscher \cite{Liebscher03}, who introduced the notion of stationary factorizing measure types on the hyperspace of closed subsets of $[0,1]\times L$ and established a structural framework for random-set Arveson systems. These works firmly established the measure-class viewpoint as a powerful tool in the construction and classification of Arveson systems.

Despite this progress, the construction of type~III product systems from stationary factorizing measure types has remained unresolved. Liebscher suggested in \cite[Note 4.4]{Liebscher03} that type III examples might arise via infinite tensor product constructions. The Liebscher-Tsirelson random-set approach is, however,  intrinsically formulated at the level of measure classes—that is, equivalence classes under mutual absolute continuity. While this suffices to define algebraic Arveson systems, it obscures two fundamental issues, which we address in this paper:

\begin{itemize}
\item[(1)] how to construct measurable Arveson product systems directly from factorizing data;
\item[(2)] how to perform infinite-product constructions in a way that is stable under changes of representative.
\end{itemize}

The second issue is particularly delicate: infinite tensor products are notoriously unstable under changes of representatives, and measure-class equivalence does not interact well with naive infinite-product constructions.

The obstacle for constructing Arveson systems of type~III from random sets is therefore conceptual rather than merely technical: one must develop a framework that allows one to work coherently with measurable representatives while retaining sufficient structure to perform factorization and controlled infinite-product constructions.

In this paper, we develop a systematic theory of representatives of stationary factorizing measure types with the introduction of measurable factorizing  families of probability measures. Our approach identifies canonical structural features that are stable under measure-class equivalence and provides tools for performing controlled infinite-product constructions at the representative level. Using this framework, we establish the existence of type~III product systems arising purely from measurable factorizing families of probability measures, and clarify the mechanism by which type~III behaviour emerges.

\subsection{Main results and organization.}

The paper is organized as follows. Section~2 develops the hyperspace framework and  introduces  factorizing families of measures
$\boldsymbol\mu=\{\mu_t\}_{t>0}$ on the hyperspaces $\mathscr C_t$ of closed subsets of
$[0,t]\times L$ (Definition~\ref{def:fact-meas}).  In addition to factorization up to
equivalence, we impose: (a) a ``no deterministic time-slice'' condition to remove boundary
ambiguities in concatenation, and (b) explicit measurability hypotheses that provide a usable Borel structure on the resulting Hilbert field.
The first main construction shows that every measurable factorizing family $\boldsymbol\mu$ canonically
generates a measurable Arveson product system $\mathbb{E}^{\boldsymbol\mu}$ (Theorem~\ref{th:random-system}), henceforth called a random-set system. We further show that this representative-level construction is compatible with the
Liebscher--Tsirelson measure-class construction: the algebraic system arising from
$\boldsymbol\mu$ is canonically isomorphic to the Liebscher--Tsirelson system associated to
the measure type $[\mu_1]$ (Proposition~\ref{prop:liebscher-equals-mu}). 
In particular, our construction provides a natural way to equip the Liebscher–Tsirelson algebraic Arveson system $\mathbb{E}^{\mathcal M}$, associated with a stationary factorizing measure type $\mathcal M$, with a measurable structure.

Section~3 analyzes units and the type of the random-set system $\mathbb{E}^{\boldsymbol\mu}$.
Our main results identify the intrinsic measure-theoretic content of units. We prove that $\mathbb{E}^{\boldsymbol\mu}$ is spatial if and only if there exists a dominated family of measures that factorizes exactly, rather than merely up to equivalence (Theorem~\ref{prop:spatial-iff-exact-factor-meas}). Moreover, positive normalized units are in bijection with such dominated exactly factorizing families (Corollary~\ref{cor:unit-transfer}). This correspondence provides a natural bridge between the Hilbert-space structure of the product system and the underlying probabilistic structure, and it serves as the main tool in our type~III construction.

We illustrate the effectiveness of this framework in Subsection~3.2, where we show that the Poisson random-set system $\mathbb{E}^{\boldsymbol\mu^{(P,L)}}$, associated with fully $L$-spread Poisson random closed sets, is of type~$\mathrm{I}$ with Arveson index 
 $
\ind\bigl(\mathbb E^{\boldsymbol\mu^{(P,L)}}\bigr)=\dim L^2(L,\eta) $ (Theorem~\ref{th:Poisson-random}).

Section~4 constructs a representative-level infinite-product mechanism which is stable and
detects non-spatiality (i.e. type~III) by a Kakutani/Hellinger criterion. We start from a seed family $\boldsymbol\nu$ over $[0,1]\times\{\ast\}$ whose random-set
system $\mathbb E^{\boldsymbol\nu}$ is of type~$\mathrm{II}_0$.  Under a quantitative
short-time Hellinger smallness assumption (Definition~\ref{ass:hellinger}), we dilate the
seed along a sequence $\{a_n\}$ with $\sum_n a_n=\infty$ and $\sum_n a_n^2<\infty$ and form the
marked infinite product over $L=\mathbb N$: $\mu_t^{(\nu)}=\bigotimes_{n\ge1} \nu_t^{(n)}$ on $\mathscr C_t^{\{\mathbb N\}}$ (Definition~\ref{def:mu-t}).
We show that this produces a measurable factorizing family
(Proposition~\ref{th:meas-seed}).  We then prove a type~III criterion:
if the seed unit has a linear ``vacuum overlap deficit'' at short times, then the resulting marked product system admits no units
(Theorem~\ref{th:typeIII-infiniteproduct}).  The proof combines:
(i) the exact-factorization characterization of units,
(ii) type~$\mathrm{II}_0$ uniqueness of the positive unit in finite marginals,
and (iii) Kakutani/Hellinger singularity for infinite product measures.

To make the type~III mechanism applicable, we give general procedures for modifying a
Palm-uniformizable factorizing family within its measure class so that it becomes a
``small'' seed and simultaneously satisfies the linear overlap condition.
This consists of an anchor-adapted logarithmic localization
(Definition~\ref{def:log-local-anchor}), an exact Palm uniformization of the anchor
(Lemma~\ref{lem:palm-uniformization}), and a deterministic normalization of the vacuum mass
(Theorem~\ref{prop:palm-loc-vacuum}).  We also provide usable sufficient conditions for
Hellinger smallness in terms of one-block stability estimates
(Proposition~\ref{lem:hellinger-small-criterion-uncond} and
Corollary~\ref{lem:hellinger-small-criterion-cond}). As a concrete application, we verify the required overlap estimates for
Brownian zero sets after anchor-adapted localization and Palm-type
uniformization, obtaining a type~III random-set product system
(Theorem~\ref{th:brownian-typeIII}).

\subsection{Background on Arveson systems}

We briefly review below the principal components of the theory of Arveson systems that are required for this work. For a comprehensive treatment, we refer the reader to~\cite{arveson-continuous}. We do not, however, provide a separate exposition of the Liebscher–Tsirelson theory of stationary factorizing measure types (see~\cite{Liebscher03}), as we choose instead to incorporate it directly into the main development of the paper.

An Arveson system is a measurable field
$\mathbb{E} =\bigsqcup_{t>0}\{t\}\times E_t$
of infinite-dimensional, separable Hilbert spaces $(E_t, \langle \cdot, \cdot \rangle_{E_t})$,
equipped with the structure of a standard Borel space.
In addition, the field is endowed with a ``multiplication'' consisting of a measurable field
of unitary operators
$U_{s,t}: E_s\otimes E_t\to E_{s+t}$, $s,t>0$,
which satisfy the associativity law
\begin{equation}\label{eq:meas-pentagon}
U_{r,\, s+t}\bigl(\mathbf{1}_{E_r} \otimes U_{s,t}\bigr)=U_{r+s,\, t}\bigl(U_{r,s}\otimes \mathbf{1}_{E_{t}}\bigr),
\end{equation}
for all $r,s,t>0$.
Finally, the field $\mathbb E$ is assumed to be trivializable, meaning that it is isomorphic,
as a measurable field of Hilbert spaces, to the trivial field $(0,\infty)\times H_{\mathrm{triv}}$
for some separable Hilbert space $H_{\mathrm{triv}}$.

An Arveson system will be denoted succinctly by $\mathbb{E}=\{E_t, U_{s,t}\}_{t>0}$.
If the measurability and trivializability requirements are omitted, and only the algebraic data
$\{E_t,U_{s,t}\}$ satisfying the associativity relation are retained, then $\mathbb E$ will be
referred to as an \emph{algebraic} Arveson system.

A unit $u = \{u_t\}_{t > 0}$ of $\mathbb{E}=\{E_t,U_{s,t}\}_{s,t>0}$ is a non-zero Borel section
$(0, \infty) \ni t \mapsto u_t \in E_t$
that satisfies the multiplicativity condition
\[
u_{s + t} = u_s\cdot u_t:=U_{s,t}(u_s\otimes u_t), \quad s,\, t > 0.
\]
The set of all units of $\mathbb E$ is denoted by $\mathrm{Unit}(\mathbb E)$.
Two units $u=\{u_t\}_{t > 0}$ and $v=\{v_t\}_{t > 0}$ are said to be equivalent if there exists
\(\lambda \in \mathbb{C}\) such that $u_t = e^{\lambda t} v_t$ for all $t > 0$.
The set $\mathrm{Unit}(\mathbb E)$ carries the Arveson covariance kernel
(\cite[Prop.\ 3.6.2]{arveson-continuous}): for $u,v\in\mathrm{Unit}(\mathbb E)$ there exists a unique
$c_{\mathbb E}(u,v)\in\C$ such that
\begin{equation}\label{eq:Arveson-cov}
\langle u_t,v_t\rangle=\exp\!\bigl(t\,c_{\mathbb E}(u,v)\bigr)\quad (t>0).
\end{equation}
The kernel $c_{\mathbb E}$ is conditionally positive definite \cite[Prop.\ 2.5.2]{arveson-continuous},
hence defines a Hilbert space $H(\mathbb E)$. The Arveson index of $\mathbb E$ is defined as
$\operatorname{ind}(\mathbb E):=\dim H(\mathbb E)$ if $\mathrm{Unit}(\mathbb E)\neq\emptyset$.

If $\mathrm{Unit}(\mathbb E) \neq \emptyset$, the Arveson system \(\mathbb E\) is said to be spatial.
Moreover, if each Hilbert space \(E_t\) is the closed linear span of finite products
$u^{(1)}_{t_1}\cdot u^{(2)}_{t_2}\cdot \dots \cdot u^{(k)}_{t_k}$, where $u^{(i)}$ are units of $\mathbb E$
and $t_1 + \dots + t_k = t$, then $\mathbb E$ is said to be of type $\mathrm{I}$.
If $\mathbb E$ is spatial but not of type $\mathrm{I}$, it is referred to as a type $\mathrm{II}$ Arveson system.
Finally, if $\mathrm{Unit}(\mathbb E) = \emptyset$, $\mathbb E$ is said to be of type $\mathrm{III}$.

We note that $\mathbb{E}$ is of type $\mathrm{II}_0$, i.e.\ type $\mathrm{II}$ with Arveson index $0$,
if and only if it has unit-line uniqueness: there exists a normalized unit $u=\{u_t\}_{t > 0}$
such that every normalized unit $v$ of $\mathbb{E}$ is equivalent to $u$.
In particular, the span of units in each fiber is the line $\C u_t$.

\subsection{Notation and Conventions.} 

All Hilbert spaces are assumed to be separable. Unless explicitly stated
otherwise, all measures are probability measures.

Two measures $\mu$ and $\nu$ on the same measurable space are said to be
equivalent, written $\mu \sim \nu$, if they are mutually absolutely
continuous. A measure type is an equivalence class $\mathcal M=[\mu]$
under this relation. The Lebesgue measure is denoted by $\mathrm{Leb}$.

Families $\{\mu_t\}_{t>0}$ of measures $\mu_t$ on measurable spaces
$(\mathscr C_t,\Sigma_t)$ will be denoted by $\boldsymbol\mu=\{\mu_t\}_{t>0}$.
For two such families $\boldsymbol\mu=\{\mu_t\}_{t>0}$ and
$\boldsymbol\nu=\{\nu_t\}_{t>0}$ we write $\boldsymbol\nu \ll \boldsymbol\mu$
if $\nu_t \ll \mu_t$ for all $t>0$, and $\boldsymbol\nu \sim \boldsymbol\mu$
if $\nu_t \sim \mu_t$ for all $t>0$.

When referring to Liebscher’s work on stationary factorizing measure types,
we use the arXiv version \cite{Liebscher03}. The notation and numbering in
that version differ slightly from those in the published article, and our
references follow the arXiv labeling.

All other notation is standard unless stated otherwise, and additional
notation will be introduced where appropriate in the course of the paper.

\section{Random-set systems}\label{sec:rigidity}

In this section we give a representative-level formulation of
Liebscher's stationary factorizing measure types.
We introduce factorizing families of probability measures and explain
how they realize Liebscher’s measure types at the level of concrete
representatives.
We then prove that every such family canonically generates an algebraic
Arveson system, and that appropriate measurability assumptions ensure
that the resulting system carries the natural measurable structure of an
Arveson system.

\subsection{Closed-set spaces and factorizing families of measures} 

 Fix a locally compact  second countable Hausdorff space $L$.
For every $t>0$, let $\mathscr C_t$ denote the hyperspace of closed subsets of $[0,t]\times L$,
equipped with its standard Borel $\sigma$--field $\Sigma_t$ (e.g.\ the Borel $\sigma$--field of the
Fell (hit-or-miss) topology on the hyperspace of closed subsets).
In \cite{Liebscher03}, the same hyperspace is denoted by
$\mathfrak F_{[0,t]\times L}$; throughout the present paper we use the simpler notation
$\mathscr C_t$.

When a clear dependence on $L$ needs to be emphasized, we shall use the notation $\mathscr C_t^{\{L\}}$ instead of $\mathscr C_t$. If $L=\{\ast\}$ is a singleton, then the hyperspace $\mathscr C_t^{\{L\}}$ will be denoted by $\mathscr C_t^{\{\ast\}}$.  

A standard subbasis for the Fell topology on $\mathscr C_t$ is given by the collection of all hit sets $H(O)=\{F\in\mathscr C_t:\ F\cap O\neq\varnothing\},$ where $O\subset [0,t]\times L$ is open, and miss sets
$M(K)=\{F\in\mathscr C_t:\ F\cap K=\varnothing\}$, where $K\subset [0,t]\times L$ is compact.

For later use, we record a countable generating family for $\Sigma_t$.
Fix a countable base $\mathcal G$ of open sets for $L$, and let $\mathcal I_t^{\Q}$ be the countable basis
of open intervals in $[0,t]$ of the three forms
\[
\mathcal I_t^{\Q}:=
\bigl\{(a,b):a,b\in\Q,\ 0<a<b<t\bigr\}
\ \cup\
\bigl\{[0,b):b\in\Q,\ 0<b\le t\bigr\}
\ \cup\
\bigl\{(a,t]:a\in\Q,\ 0\le a<t\bigr\}.
\]
Define the countable family of open rectangles
\begin{equation}\label{eq:open-rectangles}
\mathcal U_t^{\Q}:=\{\, I\times G:\ I\in\mathcal I_t^{\Q},\ G\in\mathcal G\,\}.
\end{equation}
For $U\in\mathcal U_t^{\Q}$ define the void event $V_t(U):=\{Z\in\mathscr C_t:\ Z\cap U=\varnothing\}\in \Sigma_t.$

The following lemma identifies a convenient countable generating class for the hyperspace $\sigma$--field.
\begin{lemma}\label{lem:void-generate}
The collection $\{V_t(U): U\in\mathcal U_t^{\Q}\}$ generates $\Sigma_t$ for every $t>0$.
\end{lemma}

\begin{proof}
Let $\mathcal S_t$ be the $\sigma$--field generated by $\{V_t(U):U\in\mathcal U_t^{\Q}\}$.
We claim $\mathcal S_t=\Sigma_t$. Indeed, for any open $U\subset [0,t]\times L$, 
the void event $V_t(U)=\mathscr C_t\setminus H(U)$ is Fell-closed, in particular Borel. Thus $\mathcal S_t\subset\Sigma_t$. For the converse inclusion, it suffices to show that the hit sets $H(O)$ and miss sets $M(K)$ lie in $\mathcal S_t$.

Fix an open set $O\subset [0,t]\times L$ and write $O=\bigcup_{U\in\mathcal U(O)} U$, where $\mathcal U(O):=\{U\in\mathcal U_t^{\Q}:U\subset O\}$. Then
$H(O)=\bigcup_{U\in\mathcal U(O)} H(U)=\bigcup_{U\in\mathcal U(O)}\left( \mathscr C_t\setminus V_t(U)\right)\in\mathcal S_t$ because $H(U)\in\mathcal S_t$, for every $U\in\mathcal U_t^{\Q}$, and $\mathcal U(O)$ is countable.

Fix a compact set $K\subset  [0,t]\times L$. We show that
\begin{equation}\label{eq:miss-compact-as-union}
M(K)=\bigcup\Bigl\{\ \bigcap_{j=1}^n V_t(U_j)\ :\ n\in\N,\ U_1,\dots,U_n\in\mathcal U_t^{\Q},\
K\subset\bigcup_{j=1}^n U_j\ \Bigr\}.\end{equation}

Let $F\in M(K)$, i.e.\ $F$ is closed and $F\cap K=\varnothing$.
Then $([0,t]\times L)\setminus F$ is an open set containing $K$.
For each $x\in K$, choose $U_x\in\mathcal U_t^{\Q}$ such that $x\in U_x \subset ([0,t]\times L)\setminus F.$ The family $\{U_x\}_{x\in K}$ is an open cover of the compact $K$, and let $K\subset U_{x_1}\cup\cdots\cup U_{x_n}$ be a finite subcover.
Since each $U_{x_j}\subset ([0,t]\times L)\setminus F$, we have $F\cap U_{x_j}=\varnothing$ for all $j$,
i.e.\ $F\in \bigcap_{j=1}^n V_t(U_{x_j})$.
Thus $F$ belongs to the right-hand side of \eqref{eq:miss-compact-as-union}.

Conversely, suppose $F\in \bigcap_{j=1}^n V_t(U_j)$ for some $U_1,\dots,U_n\in\mathcal U_t^{\Q}$
with $K\subset\bigcup_{j=1}^n U_j$.
Then $F\cap U_j=\varnothing$ for all $j$, hence $F$ is disjoint from $\bigcup_{j=1}^n U_j$,
and therefore $F\cap K=\varnothing$. Thus $F\in M(K)$.\end{proof}

\begin{notation}For $t>0$, let $\sigma_t:\mathscr C_t\to\mathscr C_1$ be the scaling homeomorphism acting on the time coordinate:
\begin{equation}\label{eq:scale-hom}
\sigma_t(Z):=\bigl\{(t^{-1}r,\ell): (r,\ell)\in Z\bigr\}.
\end{equation}
For $s,t>0$, we consider the concatenation map $ \oplus_{s,t}:\mathscr C_s\times \mathscr C_t\to \mathscr C_{s+t}$,
\[
\oplus_{s,t}(Z_1,Z_2)= Z_1\cup (s+Z_2),
\] where $s+Z_2:=\bigl\{(s+r,\ell): (r,\ell)\in Z_2\bigr\}\subset [s,s+t]\times L.$
\end{notation}
We now introduce the main object of study of this paper.
\begin{definition}\label{def:fact-meas}
A family $\boldsymbol\mu:=\{\mu_t\}_{t>0}$ of measures $\mu_t$ on $(\mathscr C_t,\Sigma_t)$ is called a
factorizing family over $[0,1]\times L$ if the following hold:
\begin{enumerate}[label=\textup{(\roman*)}]
\item (\textit{Factorization}.) $ \mu_{s+t}\sim(\mu_s\otimes \mu_t)\circ \oplus_{s,t}^{-1},$  for all $s,t>0$.
\item (\textit{Absence of deterministic time-slices}.) 
  $\mu_t\bigl(\{Z\in\mathscr C_t:\ Z\cap(\{r\}\times L)\neq\varnothing\}\bigr)=0,$ for all $t>0$ and all $r\in[0,t]$.
\end{enumerate}
A factorizing family $\boldsymbol\mu:=\{\mu_t\}_{t>0}$ is said to be measurable if it satisfies the following two additional conditions:
\begin{enumerate}
\item[(iii)] (\textit{Non-degeneracy}.) $\mu_t$ is not supported on finitely many atoms, for every $t>0$.
\item[(iv)] (\textit{Measurability}.) 
\begin{enumerate}
\item[(iva)] There exists a fixed countable ring $\mathcal R\subset\Sigma_1$ generating $\Sigma_1$ such that
$t\mapsto \widetilde\mu_t(A)$ is a Borel function on $(0,\infty)$ for every $A\in\mathcal R$, where $\widetilde\mu_t:=(\sigma_t)_*\mu_t$.

\item[(ivb)] There exists a Borel map
$\Delta:(0,\infty)^2\times\mathscr C_1\to(0,\infty)$ such that, for every $s,t>0$,
\[
  \Delta\bigl(s,t,\sigma_{s+t}(Z)\bigr)=\Delta_{s,t}(Z)
  \qquad\text{for $\mu_{s+t}$--a.e.\ $Z\in\mathscr C_{s+t}$},
\]
where $ \Delta_{s,t}$ is the Radon--Nikodym derivative
\[
  \Delta_{s,t}
  :=
  \frac{d\bigl((\mu_s\otimes \mu_t)\circ \oplus_{s,t}^{-1}\bigr)}{d\mu_{s+t}}
  \quad(\mu_{s+t}\text{--a.e.}).
\]
\end{enumerate}
\end{enumerate}
\end{definition}

\begin{observation}\label{obs:aux-conds}
Condition~(ii) says that, under $\mu_t$, the canonical random closed set
$Z\subset[0,t]\times L$ almost surely avoids every deterministic time-slice $\{r\}\times L$.
This removes boundary-slice ambiguities at concatenation times.
For Liebscher's stationary factorizing measure types, the analogous property is expected to hold
outside obvious degenerate cases; see \cite[Note~4.2]{Liebscher03}.

Condition~(iii) ensures the infinite dimensionality of the separable Hilbert spaces $L^2(\mathscr C_t,\mu_t)$.
Together with this, the measurability requirements in (iv) provide a convenient sufficient
hypothesis ensuring that the construction yields the measurable structure on an Arveson system.
\end{observation}

\subsection{Relation to Liebscher's stationary factorizing measure types}

The notion of a factorizing family of measures is equivalent (up to the choice of representatives)
to Liebscher's notion of a stationary factorizing measure type 
(Definition~4.1 in \cite{Liebscher03}).
We describe this equivalence in the present normalization. To this end, we first introduce some useful conceptual notation.
\begin{notation}
For $0\le s<t\le 1$, define the restriction--translate map $R_{s,t}:\mathscr C_1\to \mathscr C_{t-s}$ by
\begin{equation}\label{eq:res-tran}
R_{s,t}(Z):=\bigl((Z\cap([s,t]\times L))-(s,0)\bigr),
\end{equation}
where $(s,0)$ denotes translation by $s$ in the time coordinate (and no change in $L$): $A-(s,0):=\{(r-s,\ell):(r,\ell)\in A\}.$
For $u\in\R$, define the \emph{periodic shift} (rotation) on $\mathscr C_1$ by
\begin{equation}\label{eq:shift}
S_u(Z):=\bigl((u+Z)\cup(u-1+Z)\bigr)\cap([0,1]\times L).
\end{equation}
Note $S_{u+n}=S_u$ for all $n\in\mathbb Z$, so it suffices to consider $u\in[0,1]$. 

The next lemma shows that interval restrictions of a factorizing family recover the expected member of the family up to measure equivalence
\end{notation}
\begin{lemma}\label{lem:restriction-type}
Let $\boldsymbol\mu=\{\mu_t\}_{t>0}$ be a factorizing family of measures over $[0,1]\times L$. Then for every $0\le s<t\le 1$,
\[
(R_{s,t})_*\mu_1 \ \sim\ \mu_{t-s}.
\]
Equivalently, the restriction-translation measure type $\mathcal M_{s,t}:=(R_{s,t})_*[\mu_1]$ equals $[\mu_{t-s}]$.
\end{lemma}

\begin{proof}
Fix $0\le s<t\le 1$. First, assume $t<1$.
Iterating factorization (i) of Definition~\ref {def:fact-meas} gives the threefold factorization
\begin{equation}\label{eq:triple}
\mu_1 \ \sim\ (\mu_s\otimes \mu_{t-s}\otimes \mu_{1-t})\circ \oplus_{s,\,t-s,\,1-t}^{-1},
\end{equation}
where $\oplus_{s,\,t-s,\,1-t}(Z_1,Z_2,Z_3):=Z_1\cup(s+Z_2)\cup(t+Z_3)\in\mathscr C_1.$
Let $Z=\oplus_{s,\,t-s,\,1-t}(Z_1,Z_2,Z_3)$. Then $Z\cap([s,t]\times L)
= (s+Z_2) \cup \bigl(Z_1\cap([s,t]\times L)\bigr) \cup \bigl((t+Z_3)\cap([s,t]\times L)\bigr).$
The last two intersections can only contribute boundary slices at time $s$ or $t$.
By Definition~\ref{def:fact-meas}(ii), under $\mu_s\otimes\mu_{t-s}\otimes\mu_{1-t}$ we have
$Z_1\cap(\{s\}\times L)=\varnothing$ and $Z_2\cap(\{0\}\times L)=\varnothing$ and $Z_3\cap(\{0\}\times L)=\varnothing$ almost surely,
hence boundary-slice ambiguities occur only on a null set. Therefore for
$(\mu_s\otimes\mu_{t-s}\otimes\mu_{1-t})$--a.e.\ $(Z_1,Z_2,Z_3)$,
\[
R_{s,t}(Z)=\bigl(Z\cap([s,t]\times L)\bigr)-(s,0) = Z_2.
\]
Thus the pushforward of the RHS of \eqref{eq:triple} under $R_{s,t}$ is exactly $\mu_{t-s}$,
because it is just the projection to the middle coordinate.
Since pushforward preserves mutual absolute continuity, applying $R_{s,t}$ to \eqref{eq:triple} yields
$(R_{s,t})_*\mu_1\sim\mu_{t-s}$.

Next, suppose $t=1$. Factorization (i) gives $\mu_1\sim (\mu_s\otimes\mu_{1-s})\circ \oplus_{s,\,1-s}^{-1}$.
By (ii), for $(\mu_s\otimes\mu_{1-s})$--a.e.\ $(Z_1,Z_2)$ we have
$Z_1\cap(\{s\}\times L)=\varnothing$ and $Z_2\cap(\{0\}\times L)=\varnothing$, hence
$R_{s,1}(Z_1\oplus_{s,1-s}Z_2)=Z_2$ a.e.  Pushing forward yields $(R_{s,1})_*\mu_1\sim \mu_{1-s}$.
\end{proof}

The following proposition shows that every factorizing family determines a stationary factorizing measure type in the sense of Liebscher
\cite[Definition~4.1]{Liebscher03}.
\begin{proposition}\label{prop:mu-to-liebscher}
Let $\boldsymbol\mu=\{\mu_t\}_{t>0}$ be a factorizing family of measures, and set $\mathcal M:=[\mu_1]$
on $\mathscr C_1$. Then $\mathcal M$ is a (non-degenerate) stationary factorizing measure type
over $[0,1]\times L$, namely:
\begin{enumerate}
\item[(F)] 
\emph{(Factorization on subintervals.)}
$\mathcal M_{r,t} = \mathcal M_{r,s} * \mathcal M_{s,t},$ for all $0\le r<s<t\le 1$,
where $\mathcal M_{a,b}:=(R_{a,b})_*\mathcal M$ and $*$ is the convolution induced by union/concatenation.
\item[(S)] \emph{(Stationarity.)}
$\mathcal M^{+u}:=(S_u)_*\mathcal M=\mathcal M,$ for every $u\in\R$.
\item[(N)] \emph{(Non-degeneracy.)} $\mathcal M$ is not supported on finitely many atoms.
\end{enumerate}
Moreover $\mathcal M_{s,t}=[\mu_{t-s}]$ for all $0\le s<t\le 1$.
\end{proposition}

\begin{proof}
Fix $0\le r<s<t\le 1$.
By Lemma~\ref{lem:restriction-type}, $\mathcal M_{r,t}=[\mu_{t-r}],$
$\mathcal M_{r,s}=[\mu_{s-r}],$ and $\mathcal M_{s,t}=[\mu_{t-s}].$
Under the natural identification of time intervals by translation in the first coordinate,
the union map on $[r,t]\times L$ corresponds to the concatenation map
$\oplus_{s-r,\,t-s}$ on $[0,t-r]\times L$. Therefore $\mathcal M_{r,s}*\mathcal M_{s,t}
=
\bigl[(\mu_{s-r}\otimes \mu_{t-s})\circ \oplus_{s-r,\,t-s}^{-1}\bigr].$
Using Definition~\ref {def:fact-meas}(i) with $(s-r)+(t-s)=t-r$ gives
$\mu_{t-r}\sim (\mu_{s-r}\otimes \mu_{t-s})\circ \oplus_{s-r,\,t-s}^{-1},$
hence $\mathcal M_{r,t}=\mathcal M_{r,s}*\mathcal M_{s,t}$.

To check stationarity, it is enough to prove $(S_u)_*\mu_1\sim \mu_1$ for $u\in[0,1]$.
Fix $u\in[0,1]$ and set $a:=1-u$. Define the coordinate flip
\[
\mathrm{swap}:\mathscr C_a\times\mathscr C_u\to \mathscr C_u\times\mathscr C_a,
\qquad \mathrm{swap}(Z_1,Z_2):=(Z_2,Z_1).
\]
A direct check of \eqref{eq:shift} shows that
$S_u\bigl(Z_1\oplus_{a,u} Z_2\bigr) = Z_2\oplus_{u,a} Z_1$
whenever $Z_1\cap(\{a\}\times L)=\varnothing$ and $Z_2\cap(\{0\}\times L)=\varnothing$.
By Definition~\ref {def:fact-meas}(ii), the exceptional set has $(\mu_a\otimes\mu_u)$--measure $0$.
Hence the maps $S_u\circ\oplus_{a,u}$ and $\oplus_{u,a}\circ\mathrm{swap}$ agree
$(\mu_a\otimes\mu_u)$--a.e., and therefore their pushforwards of $\mu_a\otimes\mu_u$ coincide.
Pushing forward $\mu_1 \ \sim\ (\mu_a\otimes \mu_u)\circ \oplus_{a,u}^{-1}$ by $S_u$, one has
\[
(S_u)_*\mu_1
\ \sim\
(\mu_a\otimes\mu_u)\circ (S_u\circ \oplus_{a,u})^{-1}
=
(\mu_u\otimes\mu_a)\circ \oplus_{u,a}^{-1}\sim\ \mu_1,
\]
hence $\mathcal M^{+u}=\mathcal M$.

Non-degeneracy follows from Definition~\ref{def:fact-meas}(iii) applied at $t=1$.
\end{proof}

Liebscher's \cite[Definition~4.1]{Liebscher03} starts from a single measure type $\mathcal M$ on $\mathscr C_1$ and defines
its restriction types $\mathcal M_{s,t}$ for subintervals. The next proposition gives the converse passage from Liebscher's measure-type framework back to a factorizing family of measures, provided one chooses a representative with no deterministic time-slices.


\begin{proposition}\label{prop:liebscher-to-mu}
Let $\mathcal M$ be a stationary factorizing measure type on $\mathscr C_1$ (in Liebscher's sense).
Assume that $\mathcal M$ admits a representative $\mu\in\mathcal M$ satisfying the
no deterministic time-slice property
\[
\mu\bigl(\{Z\in\mathscr C_1:\ Z\cap(\{r\}\times L)\neq\varnothing\}\bigr)=0,\qquad r\in[0,1].
\]
Then $\mathcal M$ yields a factorizing family of measures $\boldsymbol\mu=\{\mu_t\}_{t>0}$
whose associated interval measure types are exactly $\{\mathcal M_{s,t}\}_{s,t>0}$.
\end{proposition}

\begin{proof}
Pick such a representative probability measure $\mu\in \mathcal M$ and define measures $\mu_t$ on $\mathscr C_t$ for $t\in(0,1]$ by restriction/translation, i.e.\ $\mu_t:=(R_{0,t})_*\mu$.
By applying Liebscher factorization (F) with $r=0$ and $s,t$ such that $s+t\le 1$, we have
$\mathcal M_{0,s+t} = \mathcal M_{0,s} * \mathcal M_{s,s+t}.$
By stationarity (S), shifting by $s$ maps the strip $[0,t]\times L$ to $[s,s+t]\times L$ and preserves
the measure type, so $\mathcal M_{s,s+t}$ identifies with $\mathcal M_{0,t}$ after translating back to start at $0$.
Thus $\mathcal M_{0,s+t}=\mathcal M_{0,s}*\mathcal M_{0,t}$, which is exactly 
$[\mu_{s+t}]=\bigl[(\mu_s\otimes\mu_t)\circ\oplus_{s,t}^{-1}\bigr],$
hence the claimed mutual absolute continuity.
The assumed no-deterministic-slice property for $\mu$ implies Definition~\ref{def:fact-meas}(ii)
for all $\mu_t$ with $t\in(0,1]$.

To extend from $(0,1]$ to all $t>0$, write $t=n+\alpha$ with $n\in\N$ and $\alpha\in (0,1]$, and define $\mu_t$ by iterated concatenation:
\[
\mu_t:=(\mu_1^{(\otimes n)}\otimes\mu_\alpha )\circ  \oplus_{1,\cdots,1,\alpha}^{-1}.
\]
Associativity of concatenation gives a well-defined measure class, and the resulting family is factorizing for all $s,t>0$.
\end{proof}

\begin{observation}[Failure of the no deterministic time-slice property for non-singleton $L$]
When $L$ is a singleton, Liebscher's Corollary~4.2(i) shows that for every stationary factorizing
measure type on $[0,1]$ different from $\delta_{[0,1]}$ (the measure type concentrated on the full
set $[0,1]$), and every representative $\mu$ in that measure type, one has
$\mu(\{Z:\ t\in Z\})=0$ for all $t\in[0,1]$. In particular, in the unmarked case the hypothesis
in Proposition~\ref{prop:liebscher-to-mu} is automatic once the degenerate case is excluded.
See \cite[Cor.~4.2]{Liebscher03}.

For non-singleton $L$ the corresponding ``no deterministic time-slice'' property may fail. Assume $L$ is a non-singleton compact metric space and fix $\ell_0\in L$.
Choose a Borel probability measure $\eta$ on $L\setminus\{\ell_0\}$ and fix $\lambda>0$.
Let $Y$ be a Poisson point process on $(0,1)\times(L\setminus\{\ell_0\})$ with intensity
$\nu:=\lambda\,\mathrm{Leb}\!\restriction_{(0,1)}\otimes\eta$, and let $Y\subset(0,1)\times(L\setminus\{\ell_0\})$
denote its (finite) set of atoms. Define the random closed set
\[
Z:=([0,1]\times\{\ell_0\})\cup Y \ \subset\ [0,1]\times L,
\]
and let $\mu$ be the law of $Z$ on $\mathscr C_1$, with measure type $\mathcal M:=[\mu]$. Then for every $r\in[0,1]$ one has $(r,\ell_0)\in Z$ almost surely, hence
\[
\mu\bigl(\{Z\in\mathscr C_1:\ Z\cap(\{r\}\times L)\neq\varnothing\}\bigr)=1,
\qquad r\in[0,1].
\]
In particular, the hypothesis of Proposition~\ref{prop:liebscher-to-mu} requiring a representative
with no deterministic time-slices fails for $\mathcal M$ (and in fact fails for every $\mu'\sim\mu$).

On the other hand, $\mathcal M$ is stationary and factorizing in Liebscher's sense:
$S_u$ leaves $[0,1]\times\{\ell_0\}$ invariant, and it preserves the Poisson law since
$(S_u)_*\nu=\nu$; moreover, restrictions of $Y$ to disjoint time intervals are independent and
concatenate by union, so factorization on subintervals holds. Finally, $\mu$ is non-degenerate
(not supported on finitely many atoms) because the Poisson component has infinitely many possible
configurations with positive probability.
\end{observation}

\subsection{From factorizing families to algebraic Arveson systems}

Given a factorizing family of measures $\boldsymbol\mu:=\{\mu_t\}_{t>0}$, consider the Hilbert spaces
\[
  E_t^{\boldsymbol\mu}:=L^2(\mathscr C_t,\mu_t)\qquad(t>0)
\]
and structure unitaries $U_{s,t}:E_s^{\boldsymbol\mu}\otimes E_t^{\boldsymbol\mu}\to E_{s+t}^{\boldsymbol\mu}$ defined initially on simple tensors by
\begin{equation}\label{eq:Ust}
  (U_{s,t}(f\otimes g))(Z)
  \;:=\;\Delta_{s,t}(Z)^{1/2}\,f(Z\cap([0,s]\times L))\,g\bigl((Z\cap([s,s+t]\times L))-(s,0)\bigr).
\end{equation}

\begin{proposition}\label{prop:U-unitary}
$\mathbb{E}^{\boldsymbol\mu}:=\{  E_t^{\boldsymbol\mu}, U_{s,t}\}_{t>0}$ is an algebraic Arveson system.
\end{proposition}

\begin{proof}
First, we show that each $U_{s,t}$ is an isometry on simple tensors. Identify
$E_s^{\boldsymbol\mu}\otimes E_t^{\boldsymbol\mu}$ with $L^2(\mathscr C_s\times \mathscr C_t,\mu_s\otimes\mu_t)$ and set
$\nu_{s,t}:=(\mu_s\otimes\mu_t)\circ\oplus_{s,t}^{-1}$. By definition,
$d\nu_{s,t}=\Delta_{s,t}\,d\mu_{s+t}$.
For a simple tensor $f\otimes g$ we then have
\[
\|U_{s,t}(f\otimes g)\|_{L^2(\mu_{s+t})}^2
= \int_{\mathscr C_{s+t}} |f(Z\cap([0,s]\times L))|^2\,\bigl|g((Z\cap([s,s+t]\times L))-(s,0))\bigr|^2\,d\nu_{s,t}(Z).
\]
Changing variables under the pushforward $\nu_{s,t}$ gives
\[\|U_{s,t}(f\otimes g)\|_{L^2(\mu_{s+t})}^2
= \int_{\mathscr C_s\times\mathscr C_t}
|f((Z_1\oplus_{s,t}Z_2)\cap([0,s]\times L))|^2
\bigl|g(((Z_1\oplus_{s,t}Z_2)\cap([s,s+t]\times L))-(s,0))\bigr|^2\,d(\mu_s\otimes\mu_t).\]
By Definition~\ref{def:fact-meas}(ii) (absence of deterministic time-slices), $\mu_t(\{Z_2: Z_2\cap(\{0\}\times L)\neq\varnothing\})=0$
and $\mu_s(\{Z_1: Z_1\cap(\{s\}\times L)\neq\varnothing\})=0$, hence for
$(\mu_s\otimes\mu_t)$--a.e.\ $(Z_1,Z_2)$ we have
\[
(Z_1\oplus_{s,t}Z_2)\cap([0,s]\times L)=Z_1,\qquad ((Z_1\oplus_{s,t}Z_2)\cap([s,s+t]\times L))-(s,0)=Z_2.
\]
Therefore $\|U_{s,t}(f\otimes g)\|^2=\|f\|^2\|g\|^2$, so $U_{s,t}$ extends to an isometry.

To show that $U_{s,t}$ is unitary, define $V_{s,t}:E_{s+t}^{\boldsymbol\mu}\to E_s^{\boldsymbol\mu}\otimes E_t^{\boldsymbol\mu}$ on $h\in L^2(\mathscr C_{s+t},\mu_{s+t})$ by
\[
(V_{s,t}h)(Z_1,Z_2)
:=\Delta_{s,t}(Z_1\oplus_{s,t}Z_2)^{-1/2}\,h(Z_1\oplus_{s,t}Z_2),
\qquad (Z_1,Z_2)\in\mathscr C_s\times\mathscr C_t.
\]
The same change-of-variables computation shows
$\|V_{s,t}h\|_{L^2(\mu_s\otimes\mu_t)}=\|h\|_{L^2(\mu_{s+t})}$, hence $V_{s,t}$ is an isometry.
Moreover, for simple tensors $f\otimes g$ one checks that
$V_{s,t}U_{s,t}(f\otimes g)=f\otimes g$ and $U_{s,t}V_{s,t}h=h$ $\mu_{s+t}$--a.e.
Thus $V_{s,t}=U_{s,t}^*$ and $U_{s,t}$ is unitary.

Next, we show that the family  $\{U_{s,t}\}_{s,t>0}$ satisfies the associativity law
\eqref{eq:meas-pentagon}. Fix $r,s,t>0$, and define the three-fold concatenation map
$\oplus_{r,s,t}:\mathscr C_r\times\mathscr C_s\times\mathscr C_t\to\mathscr C_{r+s+t}$ by
\[
\oplus_{r,s,t}(Z_1,Z_2,Z_3)=Z_1\cup(r+Z_2)\cup(r+s+Z_3).
\]
Let $\nu_{r,s,t}:=(\mu_r\otimes\mu_s\otimes\mu_t)\circ \oplus_{r,s,t}^{-1}$ and define
\[
\Delta_{r,s,t}:=\frac{d\nu_{r,s,t}}{d\mu_{r+s+t}}\qquad(\mu_{r+s+t}\text{-a.e.}).
\]
Also define the restriction/shift maps (well-defined $\mu_{r+s+t}$--a.e.\ by the absence of deterministic time-slices):
\[
R_{r}^{(s+t)}(Z):=\bigl((Z\cap([r,r+s+t]\times L))-(r,0)\bigr)\in\mathscr C_{s+t},
\quad
R_{r+s}^{(t)}(Z):=\bigl((Z\cap([r+s,r+s+t]\times L))-(r+s,0)\bigr)\in\mathscr C_t,
\]
and similarly $R_{0}^{(r+s)}(Z):=Z\cap([0,r+s]\times L)\in\mathscr C_{r+s}$. We claim that $\Delta_{r,s,t}$ factorizes in two ways:
\begin{equation}\label{eq:Delta-3factor}
\Delta_{r,s,t}(Z)
=
\Delta_{r,s+t}(Z)\,\Delta_{s,t}\bigl(R_r^{(s+t)}(Z)\bigr)
=
\Delta_{r+s,t}(Z)\,\Delta_{r,s}\bigl(R_0^{(r+s)}(Z)\bigr)
\qquad(\mu_{r+s+t}\text{-a.e. }Z).
\end{equation}

For this, define the factorized measures $\nu_{s,t}:=(\mu_s\otimes\mu_t)\circ\oplus_{s,t}^{-1}$ on $\mathscr C_{s+t}$, $\nu_{r,s+t}:=(\mu_r\otimes\mu_{s+t})\circ\oplus_{r,s+t}^{-1}$ on $\mathscr C_{r+s+t}$,
and $\nu_{r,s,t}:=(\mu_r\otimes\mu_s\otimes\mu_t)\circ\oplus_{r,s,t}^{-1}$ on $\mathscr C_{r+s+t}.$
By definition,
\[
\Delta_{s,t}:=\frac{d\nu_{s,t}}{d\mu_{s+t}}\quad(\mu_{s+t}\text{-a.e.}),
\quad
\Delta_{r,s+t}:=\frac{d\nu_{r,s+t}}{d\mu_{r+s+t}}\quad(\mu_{r+s+t}\text{-a.e.}),
\quad
\Delta_{r,s,t}:=\frac{d\nu_{r,s,t}}{d\mu_{r+s+t}}\quad(\mu_{r+s+t}\text{-a.e.}).
\]

Let $F:\mathscr C_{r+s+t}\to[0,\infty)$ be any nonnegative measurable function. Then
\begin{align*}
\int_{\mathscr C_{r+s+t}} F(Z)\,d\nu_{r,s,t}(Z)
&=\int_{\mathscr C_r\times\mathscr C_s\times\mathscr C_t}
F\bigl(\oplus_{r,s,t}(Z_1,Z_2,Z_3)\bigr)\,d\mu_r(Z_1)\,d\mu_s(Z_2)\,d\mu_t(Z_3)\\
&=\int_{\mathscr C_r\times\mathscr C_s\times\mathscr C_t}
F\bigl(Z_1\oplus_{r,s+t}(Z_2\oplus_{s,t}Z_3)\bigr)\,d\mu_r(Z_1)\,d\mu_s(Z_2)\,d\mu_t(Z_3)\\
&=\int_{\mathscr C_r\times\mathscr C_{s+t}}
F\bigl(Z_1\oplus_{r,s+t}W\bigr)\,d\mu_r(Z_1)\,d\nu_{s,t}(W)\\
&=\int_{\mathscr C_r\times\mathscr C_{s+t}}
F\bigl(Z_1\oplus_{r,s+t}W\bigr)\,\Delta_{s,t}(W)\,d\mu_r(Z_1)\,d\mu_{s+t}(W).
\end{align*}
Now push forward the measure $\mu_r\otimes\mu_{s+t}$ under $\oplus_{r,s+t}$.
By definition, this pushforward is $\nu_{r,s+t}$, and on the full-measure set where the boundary
time-slice $\{r\}\times L$ is not hit (so no boundary ambiguity occurs), we have $R_r^{(s+t)}\bigl(Z_1\oplus_{r,s+t}W\bigr)=W.$
Therefore the previous integral becomes
\begin{align*}
\int_{\mathscr C_{r+s+t}} F(Z)\,d\nu_{r,s,t}(Z)
&=\int_{\mathscr C_{r+s+t}} F(Z)\,\Delta_{s,t}\bigl(R_r^{(s+t)}(Z)\bigr)\,d\nu_{r,s+t}(Z).
\end{align*}
This shows $\nu_{r,s,t}\ll \nu_{r,s+t}$ with Radon--Nikodym derivative $d\nu_{r,s,t}/d\nu_{r,s+t}=\Delta_{s,t}\circ R_r^{(s+t)}$ $(\nu_{r,s+t}\text{-a.e.}).$
Since $d\nu_{r,s+t}=\Delta_{r,s+t}\,d\mu_{r+s+t}$, the Radon--Nikodym chain rule gives
\[
\frac{d\nu_{r,s,t}}{d\mu_{r+s+t}}
=
\frac{d\nu_{r,s,t}}{d\nu_{r,s+t}}\cdot\frac{d\nu_{r,s+t}}{d\mu_{r+s+t}}
=
\bigl(\Delta_{s,t}\circ R_r^{(s+t)}\bigr)\cdot \Delta_{r,s+t},
\]
which is the first factorization in \eqref{eq:Delta-3factor}. The second factorization is proved in the same way, bracketing as $(r+s)+t$ instead of $r+(s+t)$:
define $\nu_{r+s,t}:=(\mu_{r+s}\otimes\mu_t)\circ\oplus_{r+s,t}^{-1}$ and argue as before 
to obtain $d\nu_{r,s,t}/d\mu_{r+s+t}
=
\bigl(\Delta_{r,s}\circ R_0^{(r+s)}\bigr)\cdot \Delta_{r+s,t}$ $(\mu_{r+s+t}\text{-a.e.})$,
which is exactly the second equality in \eqref{eq:Delta-3factor}.

Now check the associativity law \eqref{eq:meas-pentagon} on elementary tensors $f\otimes g\otimes h$ with
$f\in E_r^{\boldsymbol\mu}$, $g\in E_s^{\boldsymbol\mu}$, $h\in E_t^{\boldsymbol\mu}$, evaluated at $Z\in\mathscr C_{r+s+t}$.
Using \eqref{eq:Ust} twice and (ii) to ignore boundary slices, we obtain
\begin{align*}
\bigl(U_{r,s+t}(\id\otimes U_{s,t})(f\otimes g\otimes h)\bigr)(Z)
&=\Delta_{r,s+t}(Z)^{1/2}\,f(Z\cap([0,r]\times L))\,
\bigl(U_{s,t}(g\otimes h)\bigr)(R_r^{(s+t)}(Z))\\
&=\Delta_{r,s+t}(Z)^{1/2}\,\Delta_{s,t}(R_r^{(s+t)}(Z))^{1/2}\,
f(\cdots)\,g(\cdots)\,h(\cdots)\\
&=\Delta_{r,s,t}(Z)^{1/2}\,f(\cdots)\,g(\cdots)\,h(\cdots),
\end{align*}
where the last step uses \eqref{eq:Delta-3factor}.
Similarly,
\begin{align*}
\bigl(U_{r+s,t}(U_{r,s}\otimes\id)(f\otimes g\otimes h)\bigr)(Z)
&=\Delta_{r+s,t}(Z)^{1/2}\,
\bigl(U_{r,s}(f\otimes g)\bigr)(R_0^{(r+s)}(Z))\,
h(R_{r+s}^{(t)}(Z))\\
&=\Delta_{r+s,t}(Z)^{1/2}\,\Delta_{r,s}(R_0^{(r+s)}(Z))^{1/2}\,
f(\cdots)\,g(\cdots)\,h(\cdots)\\
&=\Delta_{r,s,t}(Z)^{1/2}\,f(\cdots)\,g(\cdots)\,h(\cdots),
\end{align*}
again by \eqref{eq:Delta-3factor}. Hence the two sides agree $\mu_{r+s+t}$--a.e.\ on a dense set of
tensors, so \eqref{eq:meas-pentagon} holds.
\end{proof}

\subsection{Compatibility with Liebscher--Tsirelson construction}

The next proposition identifies the algebraic product system constructed from a factorizing family with the Liebscher--Tsirelson product system attached to the associated stationary factorizing measure type.

\begin{proposition}\label{prop:liebscher-equals-mu}
Let $\boldsymbol\mu=\{\mu_t\}_{t>0}$ be a factorizing family over $[0,1]\times L$, and let
$\mathcal M:=[\mu_1]$ be the associated stationary factorizing measure type.
Let $\mathbb E^{\boldsymbol\mu}=\{E_t^{\boldsymbol\mu},U_{s,t}\}_{t>0}$ be the algebraic Arveson system associated with $\boldsymbol\mu$. Then $\mathbb E^{\boldsymbol\mu}$ is isomorphic to the Liebscher--Tsirelson algebraic Arveson system $\mathbb E^{\mathcal M}$ associated with $\mathcal M$.
\end{proposition}

\begin{proof}
We recall briefly the Liebscher--Tsirelson construction (cf.\ \cite{Tsirelson00, Tsirelson2004, Liebscher03}) in a form adapted to our notation.

Let $(X,\Sigma)$ be standard Borel and let $\mathcal M$ be a measure type on $X$.
For $\mu,\mu'\in\mathcal M$ define the change-of-measure unitary
$U_{\mu,\mu'}:L^2(X,\mu)\to L^2(X,\mu')$ by
\[
  (U_{\mu,\mu'}f)(x)
  :=
  \Bigl(\frac{d\mu}{d\mu'}(x)\Bigr)^{1/2} f(x)
  \qquad(\mu'\text{-a.e. }x),
\]
so that $\int |U_{\mu,\mu'}f|^2\,d\mu'=\int |f|^2\,d\mu$. Then $U_{\mu,\mu''}=U_{\mu',\mu''}\,U_{\mu,\mu'}$ by the Radon--Nikodym chain rule.
We also consider the Hilbert space
\[
  L^2(\mathcal M)
  :=
  \Bigl\{\,\psi=(\psi_\mu)_{\mu\in\mathcal M}:
  \psi_\mu \in L^2(X,\mu),\
  \psi_{\mu'}=U_{\mu,\mu'}\psi_\mu\ \forall\,\mu,\mu'\in\mathcal M
  \Bigr\},
\]
with inner product $\langle \psi,\phi\rangle:=\int_X \psi_\mu\overline{\phi_\mu}\,d\mu$.
For any $\mu_0\in\mathcal M$, the evaluation operator $\ev_{\mu_0}:L^2(\mathcal M)\to L^2(X,\mu_0)$,
$\ev_{\mu_0}(\psi):=\psi_{\mu_0}$, is unitary.

Let now $\boldsymbol\mu=\{\mu_t\}_{t>0}$ be a factorizing family and $\mathcal M=[\mu_1]$.
Then $\mathcal M$ has interval restriction types
\[
  \mathcal M_{0,t}=[\mu_t]\qquad(0<t\le 1).
\]
By \cite[Prop.~4.1]{Liebscher03}, the fibers of the Liebscher--Tsirelson algebraic system $\mathbb E^{\mathcal M}$ are $E_t^{\mathcal M}:=L^2(\mathcal M_{0,t})$ for $t\in(0,1]$.
We then consider the unitary $ \theta_t := \ev_{\mu_t}:E_t^{\mathcal M}\longrightarrow E_t^{\boldsymbol\mu}$.

For $t>1$, Liebscher extends the multiplication from $[0,1]$ to all $t>0$ via a standard extension procedure (cf.\ \cite[Prop.~3.1]{Liebscher03}).
We extend $\theta_t$ to all $t>0$ by multiplicativity (define $\theta_{n+\alpha}$ using $\theta_1^{\otimes n}\otimes \theta_\alpha$ and the structure maps).
It suffices to check the product relation for $s,t\in(0,1]$ with $s+t\le 1$.

Fix such $s,t$ and let $f\in L^2(\mathscr C_s,\mu_s)$ and $g\in L^2(\mathscr C_t,\mu_t)$.
Form $\xi_s:=\theta_s^{-1}(f)\in E_s^{\mathcal M}$ and $\xi_t:=\theta_t^{-1}(g)\in E_t^{\mathcal M}$.
By the explicit multiplication formula in \cite[Prop.~4.1]{Liebscher03},
the product $V_{s,t}(\xi_s\otimes \xi_t)\in E_{s+t}^{\mathcal M}$ is determined by its
component with respect to the factorized representative measure
$\nu_{s,t}:=(\mu_s\otimes\mu_t)\circ\oplus_{s,t}^{-1}\in [\mu_{s+t}]$:
for $\nu_{s,t}$--a.e.\ $Z\in\mathscr C_{s+t}$,
\[
  \bigl(V_{s,t}(\xi_s\otimes \xi_t)\bigr)_{\nu_{s,t}}(Z)
  =
  f(Z\cap([0,s]\times L))\;
  g\bigl((Z\cap([s,s+t]\times L))-(s,0)\bigr).
\]
Boundary-slice ambiguities occur only on a null set thanks to Definition~\ref{def:fact-meas}(ii).

Now apply $\theta_{s+t}=\ev_{\mu_{s+t}}$, i.e.\ take the $\mu_{s+t}$--component.
Since $\nu_{s,t}\sim\mu_{s+t}$, the consistency relation gives
\[
  \bigl(V_{s,t}(\xi_s\otimes \xi_t)\bigr)_{\mu_{s+t}}
  =
  \Bigl(\frac{d\nu_{s,t}}{d\mu_{s+t}}\Bigr)^{1/2}
  \bigl(V_{s,t}(\xi_s\otimes \xi_t)\bigr)_{\nu_{s,t}}.
\]
By definition, $\Delta_{s,t}:=\frac{d\nu_{s,t}}{d\mu_{s+t}}$, hence
\[
  \theta_{s+t}V_{s,t}(\xi_s\otimes\xi_t)
  =
  \Delta_{s,t}^{1/2}\,f(Z\cap([0,s]\times L))\,g\bigl((Z\cap([s,s+t]\times L))-(s,0)\bigr),
\]
which is exactly $U_{s,t}(f\otimes g)$ by~\eqref{eq:Ust}. Therefore,
\[
  \theta_{s+t}\,V_{s,t}
  =
  U_{s,t}\,(\theta_s\otimes \theta_t)
  \qquad\text{for }s,t\in(0,1],\ s+t\le 1,
\]
on a dense subspace, hence everywhere. By the extension of the product system from $[0,1]$ to all
$t>0$, this intertwining identity holds for all $s,t>0$.
\end{proof}

\subsection{Measurability: from measurable families to Arveson systems}

The following lemma gives the basic parameter-dependent integration principle needed to establish measurability of the resulting Hilbert field.
\begin{lemma}\label{lem:meas-integral}
Let $(X,\Sigma)$ be standard Borel, let $T$ be standard Borel, and let
$t\mapsto \nu_t$ be a family of finite measures on $(X,\Sigma)$ such that
$t\mapsto \nu_t(A)$ is Borel for every $A$ in a fixed countable ring $\mathcal R$ generating $\Sigma$.
If $F:T\times X\to[0,\infty)$ is Borel, then $t\mapsto \int_X F(t,x)\,d\nu_t(x)$ is Borel.
\end{lemma}

\begin{proof}
Since $T$ is standard Borel, fix a countable $\pi$--system $\mathcal P_T$ that generates the Borel $\sigma$-field $\operatorname{Bor}(T)$.
Enlarging $\mathcal R$ if necessary (still countable), we may assume $X\in\mathcal R$.

For $E\subset T\times X$ and $t\in T$, write
$E_t:=\{x\in X:\ (t,x)\in E\},$ and define $\mathcal D
:=
\bigl\{E\in \operatorname{Bor}(T)\otimes\Sigma:\ t\mapsto \nu_t(E_t)\ \text{is Borel on }T\bigr\}.$
Then $\mathcal D$ is a Dynkin system 
because each $\nu_t$ is a finite measure and $t\mapsto \nu_t(X)$ is Borel. Moreover, for every $B\in\mathcal P_T$ and $A\in\mathcal R$ we have
$(B\times A)_t
=
\begin{cases}
A,& t\in B\\
\varnothing,& t\notin B
\end{cases}$
hence $\nu_t\bigl((B\times A)_t\bigr)=\1_B(t)\,\nu_t(A).$
 Thus $B\times A\in\mathcal D$ for all $B\in\mathcal P_T$, $A\in\mathcal R$.
Since the family $\{B\times A:\ B\in\mathcal P_T,\ A\in\mathcal R\}$ is a $\pi$--system and $\sigma(\mathcal P_T)\otimes\sigma(\mathcal R)=\operatorname{Bor}(T)\otimes\Sigma,$
the $\pi$--$\lambda$ theorem yields $\operatorname{Bor}(T)\otimes\Sigma\subset\mathcal D$.
Consequently, for every $E\in\operatorname{Bor}(T)\otimes\Sigma$, the map $t\mapsto \nu_t(E_t)$ is Borel.

Now let $F:T\times X\to[0,\infty)$ be a Borel function. Choose nonnegative simple functions
$F_n=\sum_{k=1}^{m_n} a_{n,k}\,\1_{E_{n,k}},$ where $a_{n,k}\ge0$ and $E_{n,k}\in\operatorname{Bor}(T)\otimes\Sigma$,
with $F_n\uparrow F$ pointwise. Then for each $n$,
\[
t\longmapsto \int_X F_n(t,x)\,d\nu_t(x)
=
\sum_{k=1}^{m_n} a_{n,k}\,\nu_t\bigl((E_{n,k})_t\bigr)
\]
is Borel, because each $t\mapsto \nu_t((E_{n,k})_t)$ is Borel.
By monotone convergence,
\[
\int_X F_n(t,x)\,d\nu_t(x)\ \uparrow\ \int_X F(t,x)\,d\nu_t(x)\qquad(t\in T),
\]
so $t\mapsto \int_X F(t,x)\,d\nu_t(x)$ is the pointwise limit of Borel functions and therefore Borel.
\end{proof}
The next theorem shows that the additional measurability assumptions on a factorizing family are exactly what is needed to promote the algebraic construction to an Arveson system.
\begin{theorem}\label{th:random-system}
Let $\boldsymbol\mu=\{\mu_t\}_{t>0}$ be a measurable factorizing family over $[0,1]\times L$.
Then $\mathbb E^{\boldsymbol\mu}=\{E_t^{\boldsymbol\mu},U_{s,t}\}_{t>0}$ is an Arveson system.
\end{theorem}

\begin{proof}
Since $\mathbb E^{\boldsymbol\mu}$ is an algebraic Arveson system by Proposition \ref{prop:U-unitary}, it suffices to check the measurability requirements.

Let $\mathscr C:=\mathscr C_1$ and $\Sigma:=\Sigma_1$.
Put $\widetilde\mu_t:=(\sigma_t)_*\mu_t$, where $\sigma_t:\mathscr C_t\to\mathscr C$ is the scaling homeomorphism \eqref{eq:scale-hom}. The unitary
\begin{equation}\label{eq:scaling-unity}
(S_t f)(W):=f(\sigma_t^{-1}(W))\qquad\bigl(f\in L^2(\mathscr C_t,\mu_t),\ W\in\mathscr C\bigr)
\end{equation}
identifies $E_t^{\boldsymbol\mu}$ with $\widetilde E_t^{\boldsymbol\mu}:=L^2(\mathscr C,\widetilde\mu_t)$.
Thus it suffices to build a measurable Hilbert field from $\{\widetilde E_t^{\boldsymbol\mu}\}_{t>0}$.

Let $\mathcal R=\{A_n\}_{n\ge1}$ be the fixed countable ring from Definition~\ref{def:fact-meas}\textup{(iva)} and define sections
$\xi_n(t)\in\widetilde E_t^{\boldsymbol\mu}$ by
$\xi_n(t):=\1_{A_n}$.
Then the linear span of $\{\xi_n(t)\}_n$ is dense in $\widetilde E_t^{\boldsymbol\mu}$ for each $t$.
Moreover
\begin{equation}\label{eq:meas1}
t\mapsto\ip{\xi_m(t)}{\xi_n(t)}_{\widetilde E_t^{\boldsymbol\mu}}=\widetilde\mu_t(A_m\cap A_n)
\end{equation}
is Borel by measurability \textup{(iva)}. Hence $\{\xi_n\}_{n\geq 1}$ is a measurable fundamental sequence.

Define the $\sigma$--field on $\widetilde{\mathbb E}^{\boldsymbol\mu}:=\bigsqcup_{t>0}\{t\}\times\widetilde E_t^{\boldsymbol\mu}$
to be the smallest one making all coefficient maps
\[
(t,\xi)\longmapsto \ip{\xi}{\xi_n(t)}_{\widetilde E_t^{\boldsymbol\mu}}\in\C,\qquad n\in\N,
\]
Borel. This makes $\widetilde{\mathbb E}^{\boldsymbol\mu}$ a standard Borel space and turns $t\mapsto \widetilde E_t^{\boldsymbol\mu}$ into a measurable field.
Applying a pointwise Gram--Schmidt procedure to $\{\xi_n\}_{n\geq 1}$ yields a measurable orthonormal basis
$\{e_k(t)\}_{k\ge1}$ of $\widetilde E_t^{\boldsymbol\mu}$ for each $t$ (see \cite[II.1.2]{Dixmier}).

Let $H_{\mathrm{triv}}:=\ell^2(\N)$ with its standard basis $\{\varepsilon_k\}_{k\geq 1}$, and define $\Phi: \widetilde{\mathbb E}^{\boldsymbol\mu}\to (0,\infty)\times H_{\mathrm{triv}}$ by
\[
\Phi(t,\xi):=\Bigl(t,\ \bigl(\ip{\xi}{e_k(t)}_{\widetilde E_t}\bigr)_{k\ge1}\Bigr).
\]
Then $\Phi$ is a Borel bijection whose fiber maps are unitary, giving a trivialization.
Transporting the Borel structure back along $S_t$ yields the corresponding standard Borel structure on the original disjoint union $\mathbb E^{\boldsymbol\mu}$.

Finally, we show that the map $(s,t,\xi,\eta)\mapsto U_{s,t}(\xi\otimes\eta)$ is Borel.
Define the conjugated multiplication
\[
\widetilde U_{s,t}
:=
S_{s+t}\,U_{s,t}\,(S_s^{-1}\otimes S_t^{-1})
:\widetilde E_s^{\boldsymbol\mu}\otimes \widetilde E_t^{\boldsymbol\mu}\longrightarrow \widetilde E_{s+t}^{\boldsymbol\mu}.
\]
It suffices to show that $(s,t,\xi,\eta)\mapsto \widetilde U_{s,t}(\xi\otimes\eta)$ is Borel.

Fix $s,t>0$ and set $u:=\frac{s}{s+t}\in(0,1)$.
For $W\in\mathscr C$ define the two ``restriction--rescale'' maps
\[
\rho_u(W):=\sigma_u\bigl(W\cap([0,u]\times L)\bigr)\in\mathscr C,
\qquad
\tau_u(W):=\sigma_{1-u}\bigl((W\cap([u,1]\times L))-(u,0)\bigr)\in\mathscr C.
\]
A direct substitution into \eqref{eq:Ust} shows that for $\xi\in\widetilde E_s^{\boldsymbol\mu}$, $\eta\in\widetilde E_t^{\boldsymbol\mu}$
and $W\in\mathscr C$,
\begin{equation}\label{eq:Utilde-formula}
  \bigl(\widetilde U_{s,t}(\xi\otimes\eta)\bigr)(W)
  =
  \Delta(s,t,W)^{1/2}\,\xi(\rho_u(W))\,\eta(\tau_u(W))
  \qquad \widetilde\mu_{s+t}\text{-a.e. }W,
\end{equation}
where $\Delta:(0,\infty)^2\times\mathscr C\to(0,\infty)$ is the jointly Borel version from
Definition~\ref{def:fact-meas}\textup{(ivb)}.

We show that the maps $(u,W)\mapsto \rho_u(W)$ and $(u,W)\mapsto\tau_u(W)$ are Borel.
By Lemma~\ref{lem:void-generate} at $t=1$, the $\sigma$--field on $\mathscr C$ is generated by the void events
$V_1(U)$ with $U\in\mathcal U_1^{\Q}$. Hence, it suffices to show that
$ \rho_u^{-1}(V_1(U))$ and $\tau_u^{-1}(V_1(U))$ are Borel for every $U\in\mathcal U_1^{\Q}$.

Fix $U_0=I_0\times G_0\in\mathcal U_1^{\Q}$.
For $u\in(0,1)$ define the affine images of the interval $I_0\subset[0,1]$ by
\[
uI_0:=\{ur:r\in I_0\}\subset[0,u],
\qquad
u+(1-u)I_0:=\{u+(1-u)r:r\in I_0\}\subset[u,1].
\]
A direct unwinding of the definitions gives, for $W\in\mathscr C$,
\[
\rho_u(W)\in V_1(U_0)
\iff
W\cap\bigl((uI_0)\times G_0\bigr)=\varnothing,
\qquad
\tau_u(W)\in V_1(U_0)
\iff
W\cap\bigl((u+(1-u)I_0)\times G_0\bigr)=\varnothing.
\]

We treat $\rho_u$; the argument for $\tau_u$ is identical.
Consider the hit event
\[
B_{U_0}
:=
\bigl\{(u,W)\in(0,1)\times\mathscr C:\ W\cap\bigl((uI_0)\times G_0\bigr)\neq\varnothing\bigr\}.
\]
Since $\mathcal U_1^{\Q}$ is a countable basis of open rectangles in $[0,1]\times L$, we have
\[
W\cap\bigl((uI_0)\times G_0\bigr)\neq\varnothing
\iff
\exists\,U\in\mathcal U_1^{\Q}\ \text{with}\ U\subset (uI_0)\times G_0\ \text{and}\ W\in H(U),
\]
where $H(U)=\{W\in\mathscr C:\ W\cap U\neq\varnothing\}$.
Therefore
\begin{equation}\label{eq:BU0-union}
B_{U_0}
=
\bigcup_{U\in\mathcal U_1^{\Q}}
\Bigl(\{u\in(0,1):\ U\subset (uI_0)\times G_0\}\times H(U)\Bigr).
\end{equation}
Each $H(U)$ is Borel in $\mathscr C$.
For fixed $U=J\times G\in\mathcal U_1^{\Q}$, the set
\[
\{u\in(0,1):\ U\subset (uI_0)\times G_0\}
\]
is Borel (indeed an interval, possibly empty): it is empty unless $G\subset G_0$, and when $G\subset G_0$
the inclusion $J\subset uI_0$ is equivalent to finitely many linear inequalities in $u$
because $I_0$ and $J$ are among the rational intervals in $\mathcal I_1^{\Q}$.
Hence each term in the union \eqref{eq:BU0-union} is Borel in $(0,1)\times\mathscr C$, and since the union is countable,
$B_{U_0}$ is Borel.

Consequently, its complement
\[
(0,1)\times\mathscr C\setminus B_{U_0}
=
\{(u,W):\ W\cap((uI_0)\times G_0)=\varnothing\}
=
\{(u,W):\ \rho_u(W)\in V_1(U_0)\}
=
\rho^{-1}(V_1(U_0))
\]
is Borel. Since this holds for all $U_0\in\mathcal U_1^{\Q}$, the map $(u,W)\mapsto\rho_u(W)$ is Borel.
The same argument (replacing $uI_0$ by $u+(1-u)I_0$) shows that $(u,W)\mapsto\tau_u(W)$ is Borel.

Consequently, composing with the Borel map $(s,t)\mapsto s/(s+t)$ yields that
\begin{equation}\label{eq:Borel-rr}
(s,t,W)\longmapsto \rho_{s/(s+t)}(W),
\qquad
(s,t,W)\longmapsto \tau_{s/(s+t)}(W)
\end{equation}
are Borel on $(0,\infty)^2\times\mathscr C$.

Fix $i,j,k\in\N$ and define a function $F_{i,j,k}:(0,\infty)^2\times\mathscr C\to\C$ by
\[
F_{i,j,k}(s,t,W)
:=
\Delta(s,t,W)^{1/2}\,
\1_{A_i}\!\bigl(\rho_{s/(s+t)}(W)\bigr)\,
\1_{A_j}\!\bigl(\tau_{s/(s+t)}(W)\bigr)\,
\overline{\1_{A_k}(W)}.
\]
Then $F_{i,j,k}$ is Borel by \textup{(ivb)} and \eqref{eq:Borel-rr}.
Moreover, Definition~\ref{def:fact-meas}\textup{(iva)} says that $t\mapsto \widetilde\mu_t(A)$ is Borel for every $A\in\mathcal R$.
Since $(s,t)\mapsto s+t$ is continuous, $(s,t)\mapsto \widetilde\mu_{s+t}(A)$ is Borel for every $A\in\mathcal R$.
Therefore the family of measures $\nu_{s,t}:=\widetilde\mu_{s+t}$ on $(\mathscr C,\Sigma)$ satisfies the hypotheses of Lemma~\ref{lem:meas-integral} with parameter space $T=(0,\infty)^2$ and base space $X=\mathscr C$.

Using \eqref{eq:Utilde-formula} with $\xi=\xi_i(s)=\1_{A_i}$ and $\eta=\xi_j(t)=\1_{A_j}$, we obtain
\[
\Bigl\langle \widetilde U_{s,t}\bigl(\xi_i(s)\otimes\xi_j(t)\bigr),\,\xi_k(s+t)\Bigr\rangle
=
\int_{\mathscr C} F_{i,j,k}(s,t,W)\,d\widetilde\mu_{s+t}(W).
\]
By Lemma~\ref{lem:meas-integral}, the right-hand side is a Borel function of $(s,t)$.
Therefore all matrix coefficients
\[
(s,t)\longmapsto
\Bigl\langle \widetilde U_{s,t}\bigl(\xi_i(s)\otimes\xi_j(t)\bigr),\,\xi_k(s+t)\Bigr\rangle
\]
are Borel, hence $(s,t,\xi,\eta)\mapsto \widetilde U_{s,t}(\xi\otimes\eta)$ is Borel by the standard matrix-coefficient criterion for measurability of a field of bounded operators
between measurable Hilbert fields \cite[II.2.1, Prop.1, p. 179]{Dixmier}. This concludes the proof.
\end{proof}
The following corollary records the corresponding conclusion in Liebscher's measure-type language.
\begin{corollary}\label{cor:psmeas}
Let $\boldsymbol\mu=\{\mu_t\}_{t>0}$ be a measurable factorizing family, and let
$\mathcal M:=[\mu_1]$ be the associated stationary factorizing measure type. Then $\mathbb E^{\mathcal M}$ is an Arveson product system.
\end{corollary}

\begin{definition}
The Arveson system $\mathbb E^{\boldsymbol\mu}=\{E_t^{\boldsymbol\mu},U_{s,t}\}_{t>0}$ will be referred to as the random-set system associated with the factorizing measurable family $\boldsymbol\mu=\{\mu_t\}_{t>0}$ (over $[0,1]\times L$).
\end{definition}

\section{Units and Type}

In this section we study the unit structure and type classification of random-set product systems associated with measurable factorizing families.
We give a measure-theoretic description of units in terms of dominated families of measures that factorize exactly. This viewpoint yields practical criteria for spatiality and for identifying type~I systems, which we illustrate with the fully $L$-spread Poisson random-set construction in Subsection~3.2.

 Unless otherwise stated, we keep the notation and standing assumptions of the previous section; in particular, we continue to work with a fixed locally compact, second countable Hausdorff space $L$.

\subsection{Spatiality}

The next lemma shows that, under mild measurability assumptions on two families of measures, one can choose Radon--Nikodym derivatives jointly in the parameter and the base point in a Borel way.
\begin{lemma}
\label{lem:measurable-RN}
Let $(X,\Sigma)$ be standard Borel and let $T$ be standard Borel.
Let $\mathcal A\subset\Sigma$ be a countable algebra that generates $\Sigma$.
Let $t\mapsto \mu_t$ and $t\mapsto \nu_t$ be families of finite measures on $(X,\Sigma)$ such that
\begin{enumerate}[label=\textup{(\roman*)}]
\item $t\mapsto \mu_t(A)$ and $t\mapsto \nu_t(A)$ are Borel for every $A\in\mathcal A$;
\item $\nu_t\ll \mu_t$ for every $t\in T$.
\end{enumerate}
Then there exists a Borel function $h:T\times X\to[0,\infty)$ such that for every $t\in T$,
\[
h(t,\cdot)=\frac{d\nu_t}{d\mu_t}\qquad \mu_t\text{-a.e. on }X.
\]
\end{lemma}

\begin{proof}
Enumerate $\mathcal A=\{A_1,A_2,\dots\}$ and let $\mathcal A_n$ be the finite algebra generated by
$A_1,\dots,A_n$. Let $\{C_{n,1},\dots,C_{n,m_n}\}$ be the atoms of $\mathcal A_n$, so that
$X=\bigsqcup_{k=1}^{m_n} C_{n,k}$ and each $C_{n,k}\in\mathcal A_n\subset\Sigma$.

We define the function $h_n:T\times X\to[0,\infty)$ by
\[
h_n(t,x):=\sum_{k=1}^{m_n}
\1_{C_{n,k}}(x)\,
\frac{\nu_t(C_{n,k})}{\mu_t(C_{n,k})}\,\1_{\{\mu_t(C_{n,k})>0\}}(t).
\]
Since $t\mapsto \mu_t(C_{n,k})$ and $t\mapsto \nu_t(C_{n,k})$ are Borel (each $C_{n,k}\in\mathcal A_n\subset\mathcal A$), the coefficient
$t\mapsto \frac{\nu_t(C_{n,k})}{\mu_t(C_{n,k})}\,\1_{\{\mu_t(C_{n,k})>0\}}$
is Borel; hence $h_n$ is Borel.

Fix $t\in T$. Let $g_t:=\frac{d\nu_t}{d\mu_t}$ be the Radon--Nikodym derivative.
For each $n$ and each atom $C_{n,k}$ with $\mu_t(C_{n,k})>0$, we have
\[
\frac{\nu_t(C_{n,k})}{\mu_t(C_{n,k})}
=
\frac{1}{\mu_t(C_{n,k})}\int_{C_{n,k}} g_t\,d\mu_t,
\]
so $h_n(t,\cdot)$ coincides $\mu_t$-a.e.\ with the conditional expectation
$\mathbb E_{\mu_t}[g_t\mid \mathcal A_n]$ (viewing $g_t$ as a nonnegative integrable function).
Therefore, by the martingale convergence theorem (L\'evy upward theorem),
\[
h_n(t,\cdot)\longrightarrow g_t \qquad \mu_t\text{-a.e. as }n\to\infty.
\]
Define $h(t,x):=\limsup_{n\to\infty} h_n(t,x)$. Then $h$ is Borel as a pointwise $\limsup$
of Borel functions, and for each fixed $t$ we have $h(t,\cdot)=g_t$ $\mu_t$-a.e.
\end{proof}
The following theorem gives the main characterization of spatiality in this section: the random-set system is spatial exactly when it admits a dominated family of measures that factorizes strictly, rather than only up to equivalence.

\begin{theorem}\label{prop:spatial-iff-exact-factor-meas}
Let $\boldsymbol\mu=\{\mu_t\}_{t>0}$ be a measurable factorizing family over $[0,1]\times L$, and let
$\mathbb E^{\boldsymbol\mu}=\{E_t^{\boldsymbol\mu},U_{s,t}\}_{t>0}$ be the associated random-set system. Then $\mathbb E^{\boldsymbol\mu}$ is spatial if and only if there exists a family
$\boldsymbol\nu=\{\nu_t\}_{t>0}$ of probability measures $\nu_t$ on $(\mathscr C_t,\Sigma_t)$ such that:
\begin{enumerate}[label=\textup{(\roman*)}]
\item\label{it:nu-ac} $\boldsymbol\nu\ll \boldsymbol\mu$;
\item\label{it:nu-exact} the exact factorization identities hold:
\begin{equation}\label{eq:exact-factor-nu-meas}
\nu_{s+t}=(\nu_s\otimes \nu_t)\circ\oplus_{s,t}^{-1}\qquad (s,t>0);
\end{equation}
\item\label{it:nu-meas} the map
$t\mapsto \widetilde\nu_t(A)$ is Borel on $(0,\infty)$ for every $A\in\mathcal R$, where $\widetilde\nu_t:=(\sigma_t)_*\nu_t$, and $\sigma_t$ is the scaling homeomorphism \eqref{eq:scale-hom}, and $\mathcal R\subset\Sigma_1$ is the fixed countable ring from
Definition~\ref{def:fact-meas}\textup{(iva)}.
\end{enumerate}
\end{theorem}

\begin{proof}
By enlarging $\mathcal R$ if necessary (still countable), assume
$\mathscr C_1\in\mathcal R$, so $\mathcal R$ is a countable algebra generating $\Sigma_1$. For convenience, we write $\nu_{s,t}:=(\mu_s\otimes\mu_t)\circ\oplus_{s,t}^{-1}$ and 
$\Delta_{s,t}:=d\nu_{s,t}/d\mu_{s+t}$ $(\mu_{s+t}\text{-a.e.}).$ For $Z\in\mathscr C_{s+t}$ set
\[
Z_1:=Z\cap([0,s]\times L)\in\mathscr C_s,
\qquad
Z_2:=\bigl(Z\cap([s,s+t]\times L)\bigr)-(s,0)\in\mathscr C_t.
\]
By Definition~\ref{def:fact-meas}\textup{(ii)}, boundary-slice
ambiguities at time $s$ occur only on null sets and may be ignored throughout.

\smallskip\noindent
($\Rightarrow$)
Assume $\mathbb E^{\boldsymbol\mu}$ is spatial, and let $u=\{u_t\}_{t>0}$ be a normalized unit. Define probability measures $\nu_t$ by 
$d\nu_t := |u_t|^2\,d\mu_t.$ 
Then $\nu_t\ll\mu_t$ and $\nu_t(\mathscr C_t)=\|u_t\|^2=1$.

We verify~\eqref{eq:exact-factor-nu-meas}. For this, fix $s,t>0$ and $A\in\Sigma_{s+t}$.  Since $u$ is a unit and
the multiplication $U_{s,t}$ is given by \eqref{eq:Ust}, for $\mu_{s+t}$-a.e.\ $Z$, $
|u_{s+t}(Z)|^2
=
\Delta_{s,t}(Z)\,|u_s(Z_1)|^2\,|u_t(Z_2)|^2.$
Therefore
\begin{eqnarray*}
\nu_{s+t}(A)
&=&\int_{\mathscr C_{s+t}} \1_A(Z)\,|u_{s+t}(Z)|^2\,d\mu_{s+t}(Z)
=\int_{\mathscr C_{s+t}} \1_A(Z)\,|u_s(Z_1)|^2\,|u_t(Z_2)|^2\,d\nu_{s,t}(Z)\\
&=&\int_{\mathscr C_s\times\mathscr C_t}
\1_A(Z_1\oplus_{s,t}Z_2)\,|u_s(Z_1)|^2\,|u_t(Z_2)|^2\,d(\mu_s\otimes\mu_t)(Z_1,Z_2)\\
&=&\int_{\mathscr C_s\times\mathscr C_t}
\1_A(Z_1\oplus_{s,t}Z_2)\,d(\nu_s\otimes\nu_t)(Z_1,Z_2)
=((\nu_s\otimes\nu_t)\circ\oplus_{s,t}^{-1})(A),
\end{eqnarray*} which is \eqref{eq:exact-factor-nu-meas}. 

It remains to show \ref{it:nu-meas}. For this, put $\widetilde\nu_t:=(\sigma_t)_*\nu_t$ and
$\widetilde\mu_t:=(\sigma_t)_*\mu_t$.
In the scaled picture $\widetilde E_t:=L^2(\mathscr C_1,\widetilde\mu_t)$, the unit corresponds to
$\widetilde u_t:=S_t u_t$, where $S_t$ is the scaling unitary \eqref{eq:scaling-unity}.
By construction of the measurable structure (Theorem~\ref{th:random-system}), $t\mapsto \widetilde u_t$
is a measurable section of $\{\widetilde E_t\}_{t>0}$.

Fix $A\in\mathcal R$. Then
\[
\widetilde\nu_t(A)
=
\int_{\mathscr C_1}\1_A(W)\,|\widetilde u_t(W)|^2\,d\widetilde\mu_t(W)
=
\|M_{\1_A}\widetilde u_t\|^2_{\widetilde E_t}.
\]
As in the proof of Theorem~\ref{th:random-system}, the multiplication operator field
$t\mapsto M_{\1_A}\in\mathcal B(\widetilde E_t)$ is measurable: indeed, with the fundamental sequence
$\xi_n(t)=\1_{A_n}$ ($A_n\in\mathcal R$) one has
\[
\bigl\langle M_{\1_A}\xi_m(t),\xi_n(t)\bigr\rangle
=
\widetilde\mu_t(A\cap A_m\cap A_n),
\]
which is Borel in $t$ by Definition~\ref{def:fact-meas}\textup{(iva)}. Hence $t\mapsto M_{\1_A}\widetilde u_t$ is a measurable section, and therefore
$t\mapsto \|M_{\1_A}\widetilde u_t\|^2=\widetilde\nu_t(A)$ is Borel.
This proves~\ref{it:nu-meas}.

\smallskip\noindent
($\Leftarrow$)
Conversely, assume $\boldsymbol\nu=\{\nu_t\}_{t>0}$ satisfies \ref{it:nu-ac}--\ref{it:nu-meas}.
Set $\widetilde\nu_t:=(\sigma_t)_*\nu_t$ and $\widetilde\mu_t:=(\sigma_t)_*\mu_t$ on $(\mathscr C_1,\Sigma_1)$.
By \ref{it:nu-ac} we have $\widetilde\nu_t\ll \widetilde\mu_t$ for all $t>0$, and by \ref{it:nu-meas}
the families $t\mapsto \widetilde\mu_t$ and $t\mapsto \widetilde\nu_t$ satisfy the hypotheses of
Lemma~\ref{lem:measurable-RN} with parameter space $T=(0,\infty)$ and base space $X=\mathscr C_1$
(take $\mathcal A=\mathcal R$).
Hence there exists a Borel function $h:(0,\infty)\times\mathscr C_1\to[0,\infty)$ such that
\[
h(t,\cdot)=\frac{d\widetilde\nu_t}{d\widetilde\mu_t}\qquad \widetilde\mu_t\text{-a.e.}
\]
We define $\widetilde u_t\in \widetilde E_t=L^2(\mathscr C_1,\widetilde\mu_t)$ by
$\widetilde u_t(W):=h(t,W)^{1/2}.$ Then $\|\widetilde u_t\|^2=\int h(t,W)\,d\widetilde\mu_t(W)=\widetilde\nu_t(\mathscr C_1)=1$.

We claim $t\mapsto \widetilde u_t$ is a measurable section of $\{\widetilde E_t\}_{t>0}$.
Let $\mathcal R=\{A_n\}_{n\ge1}$ and consider $\xi_n(t)=\1_{A_n}$ as in Theorem~\ref{th:random-system}.
For each $n$,
\[
\langle \widetilde u_t,\xi_n(t)\rangle
=
\int_{\mathscr C_1} \1_{A_n}(W)\,h(t,W)^{1/2}\,d\widetilde\mu_t(W).
\]
The integrand $(t,W)\mapsto \1_{A_n}(W)\,h(t,W)^{1/2}$ is Borel, and
$t\mapsto \widetilde\mu_t(B)$ is Borel for $B\in\mathcal R$ by Definition~\ref{def:fact-meas}\textup{(iva)},
so Lemma~\ref{lem:meas-integral} implies that $t\mapsto \langle \widetilde u_t,\xi_n(t)\rangle$ is Borel.
Hence $\widetilde u$ is measurable.

Now set $u_t:=S_t^{-1}\widetilde u_t\in E_t^{\boldsymbol\mu}$.
Since the measurable structure on $E_t^{\boldsymbol\mu}$ is transported along $S_t$
(Theorem~\ref{th:random-system}), the section $t\mapsto u_t$ is measurable and $\|u_t\|=1$.

It remains to verify the unit identity.
Let $f_t:=\frac{d\nu_t}{d\mu_t}\in L^1(\mu_t)$. 
Then $\nu_s\otimes\nu_t\ll \mu_s\otimes\mu_t$ with density $(Z_1,Z_2)\mapsto f_s(Z_1)f_t(Z_2)$.
Pushing forward by $\oplus_{s,t}$ gives
$(\nu_s\otimes\nu_t)\circ\oplus_{s,t}^{-1}\ll (\mu_s\otimes\mu_t)\circ\oplus_{s,t}^{-1}=\nu_{s,t},$
and the Radon--Nikodym derivative with respect to $\nu_{s,t}$ is
$\frac{d\bigl((\nu_s\otimes\nu_t)\circ\oplus_{s,t}^{-1}\bigr)}{d\nu_{s,t}}(Z)
=
f_s(Z_1)\,f_t(Z_2)$
$(\nu_{s,t}\text{-a.e. }Z).$
Thus
\[
\frac{d\bigl((\nu_s\otimes\nu_t)\circ\oplus_{s,t}^{-1}\bigr)}{d\mu_{s+t}}(Z)
=
\Delta_{s,t}(Z)\,f_s(Z_1)\,f_t(Z_2)
\qquad(\mu_{s+t}\text{-a.e. }Z).
\]
But the left-hand side equals $d\nu_{s+t}/d\mu_{s+t}=f_{s+t}$ by~\eqref{eq:exact-factor-nu-meas}.
Therefore \[f_{s+t}(Z)=\Delta_{s,t}(Z)\,f_s(Z_1)\,f_t(Z_2),\quad (\mu_{s+t}\text{-a.e. }Z).\] Taking square roots gives
$u_{s+t}(Z)=\Delta_{s,t}(Z)^{1/2}\,u_s(Z_1)\,u_t(Z_2)$
$(\mu_{s+t}\text{-a.e. }Z),$
which is exactly the unit equation $u_{s+t}=U_{s,t}(u_s\otimes u_t)$ by~\eqref{eq:Ust}.
Hence $\mathbb E^{\boldsymbol\mu}$ is spatial.
\end{proof}

The next corollary explains that, when the mark space $L$ is compact, the empty configuration has positive mass and hence gives rise to a canonical vacuum unit.

\begin{corollary}If $L$ is compact then the random-set system $\mathbb E^{\boldsymbol\mu}=\{E_t^{\boldsymbol\mu},U_{s,t}\}_{t>0}$ is spatial. 

\end{corollary}
\begin{proof}
First, we show that  $\mu_t(\{\varnothing\})>0$
for every $t>0$. Fix $t>0$, and let $\mathscr C_t^{\{\ast\}}$ be the hyperspace of closed subsets of $[0,t]$. Consider the time projection map $\pi_t:\mathscr C_t\to\mathscr C_t^{\{\ast\}}$, \[\pi_t(Z):=\{r\in[0,t]: Z\cap(\{r\}\times L)\neq\varnothing\}.\]
Since $L$ is compact, the projection $[0,t]\times L\to[0,t]$ is a closed map, hence $\pi_t(Z)$ is closed.
Moreover, $\pi_t$ is Borel because for any open interval $I\subset[0,t]$ one has $\pi_t(Z)\cap I=\varnothing
\iff
Z\cap(I\times L)=\varnothing,$
and the right-hand side is a Fell-closed event in $\mathscr C_t$.

Set $\bar\mu_t:=(\pi_t)_*\mu_t$ on $\mathscr C_t^{\{\ast\}}$. Then $\{\bar\mu_t\}_{t>0}$ is again a factorizing
family on $[0,1]$ (in the sense of Definition~\ref{def:fact-meas} with $L$ a singleton), because
$\pi_{s+t}(Z_1\oplus_{s,t}Z_2)=\pi_s(Z_1)\oplus_{s,t}\pi_t(Z_2)$. 
Also, Definition~\ref{def:fact-meas}\textup{(ii)} implies $\bar\mu_t(\{F: r\in F\})=0$ for all $r$. Hence $\overline{\mathcal M}:=[\bar\mu_1]$ is a stationary factorizing measure type and is not
$\{\delta_{[0,1]}\}$.
By Liebscher's \cite[Corollary~4.2(ii)]{Liebscher03} for $[0,1]$ one gets $\bar\mu_1(\{\varnothing\})>0$ for every
$\bar\mu_1\in\overline{\mathcal M}$. Therefore
\[
\mu_1(\{\varnothing\})
=
\mu_1(\{\pi_1(Z)=\varnothing\})
=
\bar\mu_1(\{\varnothing\})
>
0,
\]
since $\pi_1(Z)=\varnothing$ iff $Z=\varnothing$.

For $t\in(0,1]$,  Lemma~\ref{lem:restriction-type} gives $(R_{0,t})_*\mu_1\sim\mu_t$, and since
$(R_{0,t})^{-1}(\{\varnothing\})\supset\{\varnothing\}$ has positive $\mu_1$--mass, we conclude
$\mu_t(\{\varnothing\})>0$. For general $t>1$, write $t=n+\alpha$ with $n\in\mathbb{N}$ and $\alpha\in (0,1]$. Using the factorization
 $\mu_t\sim (\mu_1^{(\otimes n)}\otimes\mu_\alpha )\circ  \oplus_{1,\cdots,1,\alpha}^{-1},$ and noticing that the pushforward on the right assigns $\{\varnothing\}$ mass $\mu_1(\{\varnothing\})^n\mu_\alpha(\{\varnothing\})>0$, we obtain  $\mu_t(\{\varnothing\})>0$.

It follows that the family $\boldsymbol\delta:=\{\delta_{\varnothing}\}_{t>0}$ satisfies
$\boldsymbol\delta\ll\boldsymbol\mu$ and the exact factorization identities
$\delta_{\varnothing}=(\delta_{\varnothing}\otimes\delta_{\varnothing})\circ\oplus^{-1}$.
Consequently, Theorem~\ref{prop:spatial-iff-exact-factor-meas} yields a canonical normalized unit
(the ``vacuum unit'')
\[
u^{(0)}_t
=
\mu_t(\{\varnothing\})^{-1/2}\,\1_{\{\varnothing\}}
\in L^2(\mathscr C_t,\mu_t).
\]\end{proof}

\begin{definition}
Let $\boldsymbol\mu=\{\mu_t\}_{t>0}$ be a measurable factorizing family over $[0,1]\times L$. Assume $\mathbb E^{\boldsymbol\mu}$ is spatial. Define $\mathrm{Unit}(\boldsymbol\mu)$ to be the set of all
families $\boldsymbol\nu=\{\nu_t\}_{t>0}$ of probability measures on $(\mathscr C_t,\Sigma_t)$ satisfying conditions (i) (ii) and (iii) of Theorem~\ref{prop:spatial-iff-exact-factor-meas}, and let $\mathrm{Unit}_+^1(\mathbb E^{\boldsymbol\mu})$ be the set of
positive normalized units of $\mathbb E^{\boldsymbol\mu}$, that is
\[
\mathrm{Unit}_+^1(\mathbb E^{\boldsymbol\mu})
:=
\bigl\{u\in\mathrm{Unit}(\mathbb E^{\boldsymbol\mu}):\ \|u_t\|=1,\ u_t\ge 0\ \mu_t\text{-a.e. for all }t\bigr\}.
\]
\end{definition}
The following corollary identifies positive normalized units with exactly factorizing dominated measure families, and packages the correspondence into a simple bijection.

\begin{corollary}\label{cor:unit-transfer}
Define $\Phi:\mathrm{Unit}(\boldsymbol\mu)\to \mathrm{Unit}_+^1(\mathbb E^{\boldsymbol\mu})$, 
by $\Phi(\boldsymbol\nu) = u^{\boldsymbol\nu}$, where 
$u^{\boldsymbol\nu}_t:=\Bigl(\frac{d\nu_t}{d\mu_t}\Bigr)^{1/2}\in L^2(\mathscr C_t,\mu_t).$ 
Then $\Phi$ is well-defined and is a bijection.
\end{corollary}
\begin{proof}
If $\boldsymbol\nu\in\mathrm{Unit}(\boldsymbol\mu)$, then the hypotheses of
Theorem~\ref{prop:spatial-iff-exact-factor-meas} are satisfied, and its proof shows that
$u^{\boldsymbol\nu}=\{u^{\boldsymbol\nu}_t\}_{t>0}$ is a measurable normalized unit and is positive
$\mu_t$-a.e.

Conversely, given $u\in\mathrm{Unit}_+^1(\mathbb E^{\boldsymbol\mu})$, define $\nu_t:=u_t^2\mu_t$.
Then $\nu_t\ll\mu_t$, $\nu_t(\mathscr C_t)=\|u_t\|^2=1$, and the unit identity implies exact factorization
as in the $(\Rightarrow)$ part of Theorem~\ref{prop:spatial-iff-exact-factor-meas}.
Measurability of $t\mapsto(\sigma_t)_*\nu_t(A)$ follows as in that proof.
Finally, by construction $\frac{d\nu_t}{d\mu_t}=u_t^2$ $\mu_t$-a.e., so $\Phi(\boldsymbol\nu)=u$.
\end{proof}

\begin{observation}
For a general (not necessarily positive) normalized unit $u$, the associated family
$\nu^u_t:=|u_t|^2\mu_t$ lies in $\mathrm{Unit}(\boldsymbol\mu)$ and satisfies
$\Phi(\nu^u)=|u|:=\{|u_t|\}_{t>0}$.
Thus, the failure of injectivity of $u\mapsto \nu^u$ is precisely the possible presence of unimodular cocycles:
if $u,v$ are normalized units with $\nu^u=\nu^v$, then $v_t = g_t\,u_t$ $\mu_t\text{-a.e.}$
for a measurable family $g_t:\mathscr C_t\to\mathbb T$ satisfying
$g_{s+t}(Z_1\oplus Z_2)=g_s(Z_1)g_t(Z_2)$ $(\mu_s\otimes\mu_t)$--a.e.
\end{observation}

The next corollary records the covariance kernel of two positive units in terms of the Hellinger affinity of their associated measure families.
\begin{corollary}
\label{def:affinity-unitmu}
Assume $\mathbb E^{\boldsymbol\mu}$ is spatial. 
For $\boldsymbol\nu,\boldsymbol\omega\in\mathrm{Unit}(\boldsymbol\mu)$, one has
\[
c_{\mathbb E^{\boldsymbol\mu}}(u^{\boldsymbol\nu},u^{\boldsymbol\omega})=\frac{1}{t}\ln \int_{\mathscr C_t}
\Bigl(\frac{d\nu_t}{d\mu_t}(Z)\Bigr)^{1/2}
\Bigl(\frac{d\omega_t}{d\mu_t}(Z)\Bigr)^{1/2}\,d\mu_t(Z),
\]
for any $t>0$, and the right-hand side is independent of $t$.
\end{corollary}
\begin{proof}Follows immediately from Corollary \ref{cor:unit-transfer}.
\end{proof}

The following criterion reduces the type~$\mathrm{I}$ property to a density statement for products of positive unit vectors coming from concatenated unit measure families.

\begin{corollary}
\label{cor:typeI-criterion}
Assume $\mathbb E^{\boldsymbol\mu}$ is spatial.
For $t>0$, let $\mathcal K_t$ be the collection of probability measures on $\mathscr C_t$ of the form
\[
\kappa
=
(\nu^{(1)}_{t_1}\otimes\cdots\otimes\nu^{(n)}_{t_n})
\circ\oplus_{t_1,\dots,t_n}^{-1},
\]
where $n\in\N$, $t_1+\cdots+t_n=t$, and each $\boldsymbol\nu^{(j)}\in\mathrm{Unit}(\boldsymbol\mu)$.
Then:
\begin{enumerate}[label=\textup{(\roman*)}]
\item If for every $t>0$ the set of vectors
\[
\Bigl\{\Bigl(\frac{d\kappa}{d\mu_t}\Bigr)^{1/2}:\ \kappa\in\mathcal K_t\Bigr\}
\]
is total in $E_t^{\boldsymbol\mu}=L^2(\mathscr C_t,\mu_t)$, then $\mathbb E^{\boldsymbol\mu}$ is of
type~$\mathrm{I}$.
\item Conversely, if $\mathbb E^{\boldsymbol\mu}$ is of type~$\mathrm{I}$ and every normalized unit is a
scalar phase times a positive normalized unit 
then those vectors are total in every $E_t^{\boldsymbol\mu}$.
\end{enumerate}
\end{corollary}

\begin{proof}
For any $\kappa\in\mathcal K_t$ we have $\kappa\ll\mu_t$ and $\Bigl(\frac{d\kappa}{d\mu_t}\Bigr)^{1/2}
=
u^{\boldsymbol\nu^{(1)}}_{t_1}\cdot u^{\boldsymbol\nu^{(2)}}_{t_2}\cdots u^{\boldsymbol\nu^{(n)}}_{t_n}
\quad\text{in }E_t^{\boldsymbol\mu}$.

(i) If these vectors are total, then the closed linear span of products of units equals $E_t^{\boldsymbol\mu}$ for
each $t$, which is exactly the definition of type~$\mathrm{I}$.

(ii) If the system is type~$\mathrm{I}$ and every unit is a scalar phase times a positive unit, then
the closed linear span of positive unit products already equals the span of all unit products,
hence equals  $E_t^{\boldsymbol\mu}$; the vectors above are precisely those positive unit products.
\end{proof}
\subsection{Fully $L$--spread Poisson random closed sets}\label{subsec:Poisson-Lspread}

Fix a 
Borel probability measure $\eta$ on $L$, and an intensity parameter $\lambda>0$. For each $t>0$, let $N^{(t)}$ be a Poisson point process on $(0,t)\times L$ with intensity measure \[\nu_t:=\lambda\,\mathrm{Leb}\!\restriction_{(0,t)}\ \otimes\ \eta.\]
Let $Z_t^{\text{Poi}}\subset(0,t)\times L$ be the (finite) set of atoms of $N^{(t)}$, viewed as a random closed subset
of $[0,t]\times L$ (it is almost surely finite, hence closed).
Let $\mu_t$ be the law of  $Z_t^{\text{Poi}}$ on $(\mathscr C_t,\Sigma_t)$, and set
\[\boldsymbol\mu^{(P,L)}:=\{\mu_t\}_{t>0}.\]

We write $N(Z):=\#Z\in\N\cup\{\infty\}$ for the (extended) cardinality of $Z\in\mathscr C_t$.
For Poisson samples $Z_t^{\text{Poi}}$ one has $N(Z_t^{\text{Poi}})<\infty$ almost surely.

\begin{lemma}\label{lem:counting-borel-L}
For each $t>0$, the map $Z\mapsto N(Z)$ is $\Sigma_t$--measurable as an $\N\cup\{\infty\}$--valued function.
\end{lemma}

\begin{proof}

Fix $t>0$ and let $\mathcal U_t^{\Q}$ be the countable base of open rectangles in $[0,t]\times L$, as in \eqref{eq:open-rectangles}. It can be easily verified (a similar argument is used in the proof of Lemma~\ref{lem:Sigma-generated-by-counts-UQ} below) that for every $n\ge 1$, 
\[
\{Z:N(Z)\ge n\}
=
\bigcup_{(U_1,\dots,U_n)\in\mathcal D_n}\ \bigcap_{k=1}^n H(U_k),
\]
where $\mathcal D_n$ is the (countable) set of all $n$--tuples $(U_1,\dots,U_n)\in(\mathcal U_t^{\Q})^n$
such that the rectangles $U_1,\dots,U_n$ are pairwise disjoint. Therefore $\{N\ge n\}\in\Sigma_t$ for all $n$, and $N$ is measurable.
\end{proof}

\begin{proposition}\label{prop:Poisson-Lspread-factor}
The fully $L$--spread Poisson family $\boldsymbol\mu^{(P,L)}$ is a measurable factorizing family over $[0,1]\times L$.
\end{proposition}

\begin{proof}
For $s,t>0$, the restrictions of a Poisson point process on $(0,s+t)\times L$ to
$(0,s)\times L$ and $(s,s+t)\times L$ are independent Poisson point processes with the same intensity
$\lambda\,dx\otimes\eta$ on their respective domains.
Equivalently, if $Z_1\sim\mu_s$ and $Z_2\sim\mu_t$ are independent, then $Z_1\oplus_{s,t}Z_2$ has law $\mu_{s+t}$.
Hence $\mu_{s+t}=(\mu_s\otimes\mu_t)\circ\oplus_{s,t}^{-1},$
so Definition~\textup{\ref{def:fact-meas}}\textup{(i)} holds with $\Delta_{s,t}\equiv 1$.

Next, fix $r\in[0,t]$. Since $\nu_t(\{r\}\times L)=\lambda\,\mathrm{Leb}(\{r\})\,\eta(L)=0$,
the Poisson process almost surely has no atom in $\{r\}\times L$. Thus
$\mu_t(\{Z:\ Z\cap(\{r\}\times L)\neq\varnothing\})=0$.

Non-degeneracy \textup{(iii)} is clear: since $N(Z_t^{\text{Poi}})\sim\mathrm{Poisson}(\lambda t)$, the law $\mu_t$ charges infinitely many disjoint events
$\{N(Z)=n\}$ with positive probabilities, so it is not supported on finitely many atoms.

For measurability \textup{(iva)}, take the countable ring $\mathcal R\subset\Sigma_1$ generated by the
void events $V_1(I\times G)=\{Z:\ Z\cap (I\times G)=\varnothing\}$ with $I\in\mathcal I_1^{\Q}$ and $G\in\mathcal G$. Scaling $\sigma_t$ on the time coordinate identifies $\sigma_t(Z_t)$ with a Poisson process on $(0,1)\times L$
of intensity $\lambda t\,dx\otimes\eta$.
Hence, for each basic rectangle $I\times G$,
\[
\widetilde\mu_t\bigl(V_1(I\times G)\bigr)
=\exp\!\bigl(-\lambda t\,\mathrm{Leb}(I)\,\eta(G)\bigr),
\]
a continuous function of $t$. Finite Boolean combinations are handled by inclusion--exclusion,
so $t\mapsto\widetilde\mu_t(A)$ is Borel for all $A\in\mathcal R$.  This proves \textup{(iva)}.
Condition \textup{(ivb)} holds with the jointly Borel choice $\Delta\equiv 1$.
\end{proof}

To determine the type and index of the Poisson random-set system $\mathbb E^{\boldsymbol\mu^{(P,L)}}$ associated with $\boldsymbol\mu^{(P,L)}$ we need a few preliminary results. The next lemma gives the explicit change-of-measure formula for Poisson point processes when the underlying finite intensity measure is changed absolutely continuously.

\begin{lemma}
\label{lem:poisson-equivalence-finite}
Let $(S,\mathcal S)$ be a measurable space and let $\nu,\nu'$ be \emph{finite} measures on it.
Let $\Pi_\nu$ and $\Pi_{\nu'}$ denote the laws of Poisson point processes on $S$ with intensity measures
$\nu$ and $\nu'$, viewed as random finite counting measures (equivalently, random finite subsets when
$S$ is standard Borel).
If $\nu'\ll\nu$ with density $f:=\frac{d\nu'}{d\nu}$, then $\Pi_{\nu'}\ll\Pi_\nu$ and
\begin{equation}\label{eq:RN-Poisson-finite}
\frac{d\Pi_{\nu'}}{d\Pi_{\nu}}(Z)
=\exp\!\bigl(\nu(S)-\nu'(S)\bigr)\ \prod_{x\in Z} f(x),
\end{equation}
with the empty product interpreted as $1$.
If moreover $\nu\sim\nu'$ then $\Pi_\nu\sim\Pi_{\nu'}$.
\end{lemma}

\begin{proof}
Write $\Lambda:=\nu(S)$ and $\Lambda':=\nu'(S)$.
Under $\Pi_\nu$, the total number of points $N$ is Poisson$(\Lambda)$ and, conditional on $N=n$,
the points are i.i.d.\ with law $\nu/\Lambda$; similarly for $\nu'$.
Using this description, the likelihood ratio for a configuration $Z=\{x_1,\dots,x_n\}$ is
\[
\frac{e^{-\Lambda'}\Lambda'^n}{e^{-\Lambda}\Lambda^n}\cdot
\prod_{k=1}^n\frac{d(\nu'/\Lambda')}{d(\nu/\Lambda)}(x_k)
=
e^{\Lambda-\Lambda'}\prod_{k=1}^n \frac{d\nu'}{d\nu}(x_k),
\]
which gives \eqref{eq:RN-Poisson-finite}. Mutual absolute continuity follows when $\nu\sim\nu'$.
\end{proof}
In the following lemma, we construct a large and explicit class of exactly factorizing measure families, hence positive normalized units, by changing the mark distribution in the Poisson model.
\begin{lemma}
\label{lem:poisson-unitmeasures}
Fix $a\in L^2(L,\eta)$ with $a\ge 0$ $\eta$--a.e., and put $q:=a^2\in L^1(L,\eta)$.
For each $t>0$, let $\nu_t^{(a)}$ be the law on $(\mathscr C_t,\Sigma_t)$ of a Poisson point process on
$(0,t)\times L$ with intensity measure
\[
\nu_{t}^{(a),\mathrm{int}}
:=\lambda\,\mathrm{Leb}\!\restriction_{(0,t)}\otimes (q\,\eta).
\]
Then $\boldsymbol\nu^{(a)}:=\{\nu_t^{(a)}\}_{t>0}$ lies in $\mathrm{Unit}(\boldsymbol\mu^{(P,L)})$.
Moreover,
\begin{equation}\label{eq:poisson-unit-density}
\frac{d\nu_t^{(a)}}{d\mu_t}(Z)
=
\exp\!\Bigl(\lambda t\bigl(1-\|a\|_{L^2(\eta)}^2\bigr)\Bigr)\,
\prod_{(r,\ell)\in Z} a(\ell)^2
\qquad(\mu_t\text{--a.e. }Z\in\mathscr C_t),
\end{equation}
and the associated positive normalized unit $u^{(a)}=\Phi(\boldsymbol\nu^{(a)})$ from
Corollary~\ref{cor:unit-transfer} is given by
\begin{equation}\label{eq:poisson-unit-vector}
u_t^{(a)}(Z)
=
\exp\!\Bigl(\tfrac{\lambda t}{2}\bigl(1-\|a\|_{L^2(\eta)}^2\bigr)\Bigr)\,
\prod_{(r,\ell)\in Z} a(\ell)
\qquad(\mu_t\text{--a.e. }Z\in\mathscr C_t).
\end{equation}
\end{lemma}

\begin{proof}
Fix $t>0$ and set $S_t:=(0,t)\times L$. The reference intensity for $\mu_t$ is
$\nu_t^{\mathrm{int}}:=\lambda\,\mathrm{Leb}\!\restriction_{(0,t)}\otimes\eta$. 
 The tilted intensity is $\nu_t^{(a),\mathrm{int}}=\lambda\,\mathrm{Leb}\!\restriction_{(0,t)}\otimes(q\eta)$,
and $\nu_t^{(a),\mathrm{int}}\ll \nu_t^{\mathrm{int}}$ with density $(r,\ell)\mapsto q(\ell)=a(\ell)^2$.
Lemma~\ref{lem:poisson-equivalence-finite} applied on $S_t$ yields \eqref{eq:poisson-unit-density}.

Exact factorization for $\boldsymbol\nu^{(a)}$ holds because restrictions of a Poisson point process
to disjoint time intervals are independent and have the same tilted intensity on each interval; hence
\[
\nu_{s+t}^{(a)}=(\nu_s^{(a)}\otimes \nu_t^{(a)})\circ \oplus_{s,t}^{-1}.
\]
Also $\boldsymbol\nu^{(a)}\ll\boldsymbol\mu^{(P,L)}$ by \eqref{eq:poisson-unit-density}, and the measurability
condition (iii) in Theorem~\ref{prop:spatial-iff-exact-factor-meas} follows as in
Proposition~\ref{prop:Poisson-Lspread-factor}:
after scaling to $[0,1]\times L$, void probabilities are of the form
\[t\mapsto \exp(-\lambda t\,\mathrm{Leb}(I)\int_G q\,d\eta),\] hence are Borel on the generating ring.
Thus $\boldsymbol\nu^{(a)}\in \mathrm{Unit}(\boldsymbol\mu^{(P,L)})$ and \eqref{eq:poisson-unit-vector}
is immediate from Corollary~\ref{cor:unit-transfer}.
\end{proof}
The next lemma shows that the hyperspace $\sigma$-field is generated by finitely many rectangle-counting variables and deduces the corresponding density statement in $L^2$.
\begin{lemma}
\label{lem:Sigma-generated-by-counts-UQ}
Fix $t>0$. For $U\in\mathcal U_t^{\Q}$ define the counting functional
\[
N_U(Z):=\#(Z\cap U)\in\N\cup\{\infty\},\qquad Z\in\mathscr C_t.
\]
Then $N_U$ is $\Sigma_t$--measurable for every $U\in\mathcal U_t^{\Q}$, and $\Sigma_t=\sigma\bigl(N_U:\ U\in\mathcal U_t^{\Q}\bigr).$
Consequently, if $\mathcal F(\mathbf U):=\sigma(N_{U_1},\dots,N_{U_m})$ for a finite tuple
$\mathbf U=(U_1,\dots,U_m)\in(\mathcal U_t^{\Q})^m$, then
\[
\overline{\bigcup_{\mathbf U}\,L^2(\mathcal F(\mathbf U),\mu_t)}^{\ \|\cdot\|_{L^2(\mu_t)}}=L^2(\mathscr C_t,\mu_t),
\]
where the union runs over all finite tuples $\mathbf U$.
\end{lemma}
\begin{proof}
Fix $U\in\mathcal U_t^{\Q}$. For each $n\ge1$, define $A_{U,n}:=\{Z\in\mathscr C_t:\ N_U(Z)\ge n\}.$
We claim that
\begin{equation}\label{eq:AU-n-representation}
A_{U,n}
=
\bigcup_{(V_1,\dots,V_n)\in\mathcal D_n(U)}\ \bigcap_{k=1}^n H(V_k),
\end{equation}
where $\mathcal D_n(U)$ denotes the set of all $n$--tuples $(V_1,\dots,V_n)\in(\mathcal U_t^{\Q})^n$
such that $V_k\subset U$ for all $k$, and $V_1,\dots,V_n$ are pairwise disjoint. 

First, let $Z\in A_{U,n}$, so $\#(Z\cap U)\ge n$. Choose $n$ distinct points
$x_1,\dots,x_n\in Z\cap U$.
The space $[0,t]\times L$ is locally compact, second countable and Hausdorff, hence regular and
second countable; therefore it is metrizable.
Fix a compatible metric $d$ on $[0,t]\times L$.
Since $U$ is open, for each $k$ pick $r_k>0$ such that the open ball
$B(x_k,r_k)\subset U$.
Also choose $r_k$ so small that $r_k<\frac12\min_{j\neq k} d(x_k,x_j)$.
Then the balls $B(x_k,r_k)$ are pairwise disjoint open subsets of $U$.
By definition of the base $\mathcal U_t^{\Q}$, for each $k$ there exists a rectangle
$V_k\in\mathcal U_t^{\Q}$ such that $x_k\in V_k\subset B(x_k,r_k)\subset U.$
In particular the $V_k$ are pairwise disjoint and contained in $U$, so
$(V_1,\dots,V_n)\in\mathcal D_n(U)$.
Since $x_k\in Z\cap V_k$, we have $Z\in H(V_k)$ for each $k$, hence
$Z\in \bigcap_{k=1}^n H(V_k)$ and therefore $Z$ belongs to the right-hand side of
\eqref{eq:AU-n-representation}.

Conversely, suppose $Z$ belongs to the right-hand side of \eqref{eq:AU-n-representation}.
Then for some pairwise disjoint $V_1,\dots,V_n\subset U$ we have $Z\in\bigcap_{k=1}^n H(V_k)$.
Choose points $y_k\in Z\cap V_k$.
Since the $V_k$ are disjoint, the points $y_1,\dots,y_n$ are distinct and belong to $Z\cap U$.
Thus $\#(Z\cap U)\ge n$, i.e.\ $Z\in A_{U,n}$.
This proves \eqref{eq:AU-n-representation}. Consequently $A_{U,n}\in\Sigma_t$ for all $n\ge1$, hence $N_U$ is $\Sigma_t$--measurable.

Next we show that $\Sigma_t=\sigma(N_U:\ U\in\mathcal U_t^{\Q})$.
By Lemma~\ref{lem:void-generate}, $\Sigma_t$ is generated by the void events $V_t(W):=\{Z\in\mathscr C_t:\ Z\cap W=\varnothing\},$ with $W\in\mathcal U_t^{\Q}.$
But for each $W\in\mathcal U_t^{\Q}$ we have $V_t(W)=\{N_W=0\}$, hence
$V_t(W)\in\sigma(N_U:\ U\in\mathcal U_t^{\Q})$.
Therefore $\Sigma_t\subset \sigma(N_U:\ U\in\mathcal U_t^{\Q}).$ The reverse inclusion holds because each $N_U$ is $\Sigma_t$--measurable, hence
$\sigma(N_U:\ U\in\mathcal U_t^{\Q})\subset\Sigma_t$.
Thus $\Sigma_t=\sigma(N_U:\ U\in\mathcal U_t^{\Q})$.

It remains to show the density of finite-$\sigma$--field subspaces in $L^2(\Sigma_t,\mu_t)$.
Since $\mathcal U_t^{\Q}$ is countable, enumerate it as $\mathcal U_t^{\Q}=\{U^{(1)},U^{(2)},\dots\}$ and set $\mathcal F_n:=\sigma\bigl(N_{U^{(1)}},\dots,N_{U^{(n)}}\bigr),$ for every $n\in\N.$
Notice that $\mathcal F_n\uparrow \Sigma_t$.

Let $f\in L^2(\mathscr C_t,\mu_t).$
Define $f_n:=\EE_{\mu_t}[f\mid \mathcal F_n]\in L^2(\mathcal F_n,\mu_t)$.
By the $L^2$ martingale convergence theorem, $f_n\to f$ in $L^2(\mu_t)$ as $n\to\infty$.
Hence $f$ lies in the $L^2$--closure of $\bigcup_n L^2(\mathcal F_n,\mu_t)$. Finally, each $\mathcal F_n$ equals $\mathcal F(\mathbf U)$ for the finite tuple
$\mathbf U=(U^{(1)},\dots,U^{(n)})$, so
$\bigcup_{n\ge1} L^2(\mathcal F_n,\mu_t)\subset \bigcup_{\mathbf U} L^2(\mathcal F(\mathbf U),\mu_t).$
Therefore
\[
\overline{\bigcup_{\mathbf U} L^2(\mathcal F(\mathbf U),\mu_t)}^{\ \|\cdot\|_{L^2(\mu_t)}}
=L^2(\mathscr C_t,\mu_t),
\]
as claimed.
\end{proof}

The following lemma is a simple completeness criterion: under finite exponential moments, exponentials in a finite family of counting variables already span the whole corresponding $L^2$-space.
\begin{lemma}
\label{lem:exp-total-count-vector}
Let $(M_1,\dots,M_m)$ be an $\N^m$--valued random vector on some probability space such that
$\EE\Bigl[\prod_{j=1}^m b_j^{\,2M_j}\Bigr]<\infty$
for all $(b_1,\dots,b_m)\in(0,\infty)^m.$
Then 
\[
\overline{\operatorname{span}}\Bigl\{\ \prod_{j=1}^m b_j^{\,M_j}\ :\ b_1,\dots,b_m>0\ \Bigr\}=L^2(\sigma(M_1,\dots,M_m)).
\]
\end{lemma}

\begin{proof}
Let $F\in L^2(\sigma(M_1,\dots,M_m))$ and assume $\EE\Bigl[F\,\prod_{j=1}^m b_j^{M_j}\Bigr]=0$, for all $(b_1,\dots,b_m)\in(0,\infty)^m$.
Write $F=f(M_1,\dots,M_m)$ for some $f:\N^m\to\C$. For $z=(z_1,\dots,z_m)\in\C^m$ define
\[
G(z):=\EE\Bigl[f(M)\prod_{j=1}^m z_j^{M_j}\Bigr].
\]
By Cauchy--Schwarz and the hypothesis with $b_j=|z_j|$, we have
$|G(z)|\le \|F\|_2\cdot \Bigl(\EE\bigl[\prod_{j=1}^m |z_j|^{2M_j}\bigr]\Bigr)^{1/2}<\infty,$
so $G$ is well-defined for all $z\in\C^m$. Since $M$ is $\N^m$--valued,
\[
G(z)=\sum_{n\in\N^m} f(n)\,\PP(M=n)\,\prod_{j=1}^m z_j^{n_j}
\]
hence $G$ is an entire function on $\C^m$. By assumption, $G(b)=0$ for all $b\in(0,\infty)^m$, so $G\equiv 0$ on $\C^m$ by the identity theorem.
Therefore every coefficient in the power series is $0$, i.e.
\[
f(n)\,\PP(M=n)=0\qquad\forall\,n\in\N^m.
\]
Hence $f(M)=0$ a.s., i.e. $F=0$ a.s. This proves totality.
\end{proof}
The next lemma shows that, in the Poisson model, monomials in finitely many rectangle counts are realized by square roots of Radon--Nikodym derivatives coming from concatenations of unit measure families.
\begin{lemma}\label{lem:monomials-from-Kt-UQ}
Fix $t>0$ and a finite tuple $\mathbf U=(U_1,\dots,U_m)$ with $U_i=I_i\times G_i\in\mathcal U_t^{\Q}$.
Let $b_1,\dots,b_m>0$.
Then there exists $\kappa\in\mathcal K_t$ (as in Corollary~\ref{cor:typeI-criterion}) and a constant $C>0$
such that
\[
\Bigl(\frac{d\kappa}{d\mu_t}\Bigr)^{1/2}(Z)
=
C\,\prod_{i=1}^m b_i^{\,N_{U_i}(Z)}
\qquad(\mu_t\text{--a.e. }Z\in\mathscr C_t).
\]
\end{lemma}

\begin{proof}
Write each $U_i=I_i\times G_i$ with $I_i\in\mathcal I_t^{\Q}$ and $G_i\in\mathcal G$.
Let $\mathcal Q$ be the finite set of all rational endpoints appearing in the $I_i$, together with $0$ and $t$.
List them as $0=q_0<q_1<\cdots<q_K=t$.

If $K=1$ (so $\mathcal Q=\{0,t\}$), set $J_1:=(0,t]$ and $\ell_1:=t$.
If $K\ge2$, define the disjoint intervals
\[
J_1:=[0,q_1),\quad
J_k:=(q_{k-1},q_k)\ (2\le k\le K-1),\quad
J_K:=(q_{K-1},t],
\]
so each $J_k\in\mathcal I_t^{\Q}$, and let $\ell_k:=|J_k|$ so that $\sum_{k=1}^K \ell_k=t$. For $k=1,\dots,K$, define the restriction/translate map $Z\mapsto Z|_{J_k}\in\mathscr C_{\ell_k}$ by $Z|_{J_k}:=\bigl((Z\cap (J_k\times L))-(q_{k-1},0)\bigr),$
(with the convention $q_{0}=0$). For each $k$, define $a_k:L\to(0,\infty)$ by
\[
a_k(\ell):=\prod_{i:\ J_k\cap (0,t)\subset I_i} b_i^{\,\1_{G_i}(\ell)}.
\]
Then $a_k$ is bounded, hence $a_k\in L^2(L,\eta)$ and $a_k\ge0$.

Let $\boldsymbol\nu^{(a_k)}\in\mathrm{Unit}(\boldsymbol\mu)$ be the unit measure family from
Lemma~\ref{lem:poisson-unitmeasures} corresponding to $a_k$, and define $\kappa
:=
(\nu^{(a_1)}_{\ell_1}\otimes\cdots\otimes\nu^{(a_K)}_{\ell_K})
\circ \oplus_{\ell_1,\dots,\ell_K}^{-1}\ \in\ \mathcal K_t.$ Since $\boldsymbol\mu$ factorizes exactly and $\Delta\equiv 1$ (Proposition~\ref{prop:Poisson-Lspread-factor}),
we have $\mu_t=(\mu_{\ell_1}\otimes\cdots\otimes\mu_{\ell_K})\circ\oplus_{\ell_1,\dots,\ell_K}^{-1}.$ 
Moreover, $\mu_t$ charges no deterministic time-slice, hence $\mu_t$--a.e.\ $Z$ has no points on the
boundary slices $\{q_j\}\times L$. On this full-measure set, the concatenation map has the inverse
$Z\mapsto (Z|_{J_1},\dots,Z|_{J_K})$, and therefore
\[
\frac{d\kappa}{d\mu_t}(Z)
=
\prod_{k=1}^K \frac{d\nu^{(a_k)}_{\ell_k}}{d\mu_{\ell_k}}\bigl(Z|_{J_k}\bigr)
\qquad(\mu_t\text{--a.e. }Z).
\]
Taking square-roots and using \eqref{eq:poisson-unit-vector} on each block yields, for $\mu_t$--a.e.\ $Z$,
\[
\Bigl(\frac{d\kappa}{d\mu_t}\Bigr)^{1/2}(Z)
=
C\ \prod_{k=1}^K\ \prod_{(r,\ell)\in Z:\ r\in J_k} a_k(\ell)
\]
for a deterministic constant $C>0$ (the product of the normalization factors). Substituting the definition of $a_k$ gives
\[
\prod_{(r,\ell)\in Z:\ r\in J_k} a_k(\ell)
=
\prod_{i:\ J_k\cap (0,t)\subset I_i} b_i^{\,\#\{(r,\ell)\in Z:\ r\in J_k,\ \ell\in G_i\}}.
\]
Multiplying over $k$ and swapping the order of products yields
\[
\prod_{k=1}^K \prod_{i:\ J_k\cap (0,t)\subset I_i} b_i^{\,\#(Z\cap(J_k\times G_i))}
=
\prod_{i=1}^m b_i^{\,\sum_{k:\ J_k\cap(0,t)\subset I_i} \#(Z\cap(J_k\times G_i))}.
\]

Finally, since $\{J_k\}$ partitions $[0,t]$ up to the boundary times $\{q_1,\dots,q_{K-1}\}$ and
$\mu_t$ charges no deterministic time-slice (Definition~\ref{def:fact-meas}(ii)), we have
\[
\sum_{k:\ J_k\cap(0,t)\subset I_i} \#(Z\cap(J_k\times G_i))
=
\#(Z\cap(I_i\times G_i))
=
N_{U_i}(Z)
\qquad(\mu_t\text{--a.e. }Z).
\]
Thus
\[
\Bigl(\frac{d\kappa}{d\mu_t}\Bigr)^{1/2}(Z)
=
C\,\prod_{i=1}^m b_i^{\,N_{U_i}(Z)}
\qquad(\mu_t\text{--a.e.}),
\]
as claimed.
\end{proof}
The main result of this subsection identifies the fully $L$--spread Poisson random-set system as a type~$\mathrm{I}$ system and computes its Arveson index explicitly. The case $L=\N$ and $\eta=\#$ was discussed by Liebscher in Example~4.1 and Note~4.4 of \cite{Liebscher03}.

\begin{theorem}\label{th:Poisson-random}
The Poisson random-set system $\mathbb E^{\boldsymbol\mu^{(P,L)}}$ is of type~$\mathrm{I}$ and index $
\ind\bigl(\mathbb E^{\boldsymbol\mu^{(P,L)}}\bigr)=\dim L^2(L,\eta).$

\end{theorem}

\begin{proof}
Write $\boldsymbol\mu:=\boldsymbol\mu^{(P,L)}$ and $E_t:=L^2(\mathscr C_t,\mu_t)$ for simplicity.
By Proposition~\ref{prop:Poisson-Lspread-factor} the factorization is exact and
$\Delta_{s,t}\equiv 1$ for all $s,t>0$. In particular, the constant section
$u_t\equiv 1\in E_t$ is a normalized unit 
so $\mathbb E^{\boldsymbol\mu}$ is spatial and Corollary~\ref{cor:typeI-criterion} applies. Fix $t>0$, and let $\mathcal K_t$ be as in Corollary~\ref{cor:typeI-criterion}.
We must show that the set of vectors $\{(d\kappa/d\mu_t)^{1/2}:\kappa\in\mathcal K_t\}$ is total in $E_t$.

For $U\in\mathcal U_t^{\Q}$ let $N_U(Z):=\#(Z\cap U)$.
By Lemma~\ref{lem:Sigma-generated-by-counts-UQ}, $\Sigma_t=\sigma\bigl(N_U:\ U\in\mathcal U_t^{\Q}\bigr),$
and therefore the union of subspaces
\[
\bigcup_{\mathbf U} L^2(\mathcal F(\mathbf U),\mu_t)
\]
is dense in $E_t$, where the union runs over all finite tuples $\mathbf U=(U_1,\dots,U_m)$, and $\mathcal F(\mathbf U):=\sigma(N_{U_1},\dots,N_{U_m}).$
Hence it suffices to prove that for each fixed finite tuple $\mathbf U$,
the subspace $L^2(\mathcal F(\mathbf U),\mu_t)$ is contained in the closed linear span of
$\{(d\kappa/d\mu_t)^{1/2}:\kappa\in\mathcal K_t\}$.

For this, fix a finite tuple $\mathbf U=(U_1,\dots,U_m)$ with $U_i\in\mathcal U_t^{\Q}$ and set
\[
M_i:=N_{U_i}\qquad (1\le i\le m).
\]
We claim that $(M_1,\dots,M_m)$ has finite exponential moments in the sense of
Lemma~\ref{lem:exp-total-count-vector}, i.e.
\[
\EE_{\mu_t}\Bigl[\prod_{i=1}^m b_i^{\,2M_i}\Bigr]<\infty
\qquad\forall\, (b_1,\dots,b_m)\in(0,\infty)^m.
\]
Indeed, put $U:=\bigcup_{i=1}^m U_i$ and $B:=\prod_{i=1}^m \max(b_i^2,1)$.
For each configuration $Z\in\mathscr C_t$ and each point $x\in Z\cap U$, one has
$\prod_{i:\,x\in U_i} b_i^2\le B$, hence
\[
\prod_{i=1}^m b_i^{\,2M_i(Z)}
=\prod_{i=1}^m b_i^{\,2\#(Z\cap U_i)}
\le B^{\#(Z\cap U)}=:B^{\,N_U(Z)}.
\]
Under $\mu_t$, $N_U(Z)$ is Poisson with parameter $\nu_t^{\mathrm{int}}(U)<\infty$
(where $\nu_t^{\mathrm{int}}=\lambda\,\mathrm{Leb}\!\restriction_{(0,t)}\otimes\eta$), so
$\EE_{\mu_t}[B^{N_U}]=\exp(\nu_t^{\mathrm{int}}(U)\,(B-1))<\infty$.
This proves the exponential-moment hypothesis.

Therefore Lemma~\ref{lem:exp-total-count-vector} applies and yields that the linear span of
\[
\Bigl\{\ \prod_{i=1}^m b_i^{\,M_i}\ :\ b_1,\dots,b_m>0\ \Bigr\}
=
\Bigl\{\ \prod_{i=1}^m b_i^{\,N_{U_i}}\ :\ b_1,\dots,b_m>0\ \Bigr\}
\]
is dense in $L^2(\mathcal F(\mathbf U),\mu_t)$.

Fix $b_1,\dots,b_m>0$. By Lemma~\ref{lem:monomials-from-Kt-UQ}, there exist $\kappa\in\mathcal K_t$
and a constant $C>0$ such that
\[
\Bigl(\frac{d\kappa}{d\mu_t}\Bigr)^{1/2}(Z)
=
C\,\prod_{i=1}^m b_i^{\,N_{U_i}(Z)}
\qquad(\mu_t\text{--a.e. }Z).
\]
Hence every monomial $\prod_{i=1}^m b_i^{N_{U_i}}$ belongs to the linear span of
$\{(d\kappa/d\mu_t)^{1/2}:\kappa\in\mathcal K_t\}$, and therefore
\[
L^2(\mathcal F(\mathbf U),\mu_t)
\subset
\overline{\mathrm{span}}\Bigl\{\Bigl(\frac{d\kappa}{d\mu_t}\Bigr)^{1/2}:\ \kappa\in\mathcal K_t\Bigr\}.
\]

Since the union of $L^2(\mathcal F(\mathbf U),\mu_t)$ over all finite tuples $\mathbf U$ is dense in $E_t$
(Lemma~\ref{lem:Sigma-generated-by-counts-UQ}), we conclude that
\[
\overline{\mathrm{span}}\Bigl\{\Bigl(\frac{d\kappa}{d\mu_t}\Bigr)^{1/2}:\ \kappa\in\mathcal K_t\Bigr\}
=E_t.
\]
As $t>0$ was arbitrary, Corollary~\ref{cor:typeI-criterion}\textup{(i)} implies that
$\mathbb E^{\boldsymbol\mu}$ is of type~$\mathrm{I}$.

It remains to show that $
\ind\bigl(\mathbb E^{\boldsymbol\mu}\bigr)=\dim L^2(L,\eta).$ For this, let $a,b\in L^2(L,\eta)$ with $a,b\ge 0$ $\eta$--a.e., and consider the corresponding positive normalized
units $u^{(a)},u^{(b)}$ from Lemma~\ref{lem:poisson-unitmeasures}.
By Corollary~\ref{def:affinity-unitmu},
\[
\langle u_t^{(a)},u_t^{(b)}\rangle
=
\int_{\mathscr C_t}
\Bigl(\frac{d\nu_t^{(a)}}{d\mu_t}\Bigr)^{1/2}
\Bigl(\frac{d\nu_t^{(b)}}{d\mu_t}\Bigr)^{1/2}\,d\mu_t.
\]
Using \eqref{eq:poisson-unit-density}, the integrand equals
\[
\exp\!\Bigl(\tfrac{\lambda t}{2}\bigl(2-\|a\|_2^2-\|b\|_2^2\bigr)\Bigr)\,
\prod_{(r,\ell)\in Z} (a(\ell)b(\ell)),
\]
where $\|\cdot\|_2$ is the $L^2(L,\eta)$ norm.
Thus
\begin{equation}\label{eq:inner-ab}
\langle u_t^{(a)},u_t^{(b)}\rangle
=
\exp\!\Bigl(\tfrac{\lambda t}{2}\bigl(2-\|a\|_2^2-\|b\|_2^2\bigr)\Bigr)\,
\EE_{\mu_t}\Bigl[\prod_{(r,\ell)\in Z_t} a(\ell)b(\ell)\Bigr].
\end{equation}

We compute the remaining expectation by conditioning on the Poisson count:
under $\mu_t$, $N:=\#Z_t\sim\mathrm{Poisson}(\lambda t)$ and conditional on $N=n$ the marks are i.i.d.\ with
law $\eta$. Hence
\[
\EE_{\mu_t}\Bigl[\prod_{(r,\ell)\in Z_t} a(\ell)b(\ell)\ \Big|\ N\Bigr]
=
\bigl(\EE_\eta[a b]\bigr)^{N}
=
\langle a,b\rangle_{L^2(\eta)}^{\,N}.
\]
Taking expectation in $N$ gives
\[
\EE_{\mu_t}\Bigl[\prod_{(r,\ell)\in Z_t} a(\ell)b(\ell)\Bigr]
=
\EE\bigl[\langle a,b\rangle^{N}\bigr]
=
\exp\!\bigl(\lambda t(\langle a,b\rangle-1)\bigr),
\]
since $\EE[c^N]=\exp(\lambda t(c-1))$ for $N\sim\mathrm{Poisson}(\lambda t)$.
Substituting into \eqref{eq:inner-ab} yields
\[
\langle u_t^{(a)},u_t^{(b)}\rangle
=
\exp\!\Bigl(\lambda t\Bigl(\langle a,b\rangle-\tfrac12\|a\|_2^2-\tfrac12\|b\|_2^2\Bigr)\Bigr)
=
\exp\!\Bigl(-\tfrac{\lambda t}{2}\|a-b\|_2^2\Bigr).
\]
Therefore the Arveson covariance kernel satisfies
\begin{equation}\label{eq:c-ab}
c_{\mathbb E}(u^{(a)},u^{(b)})
=
\frac{1}{t}\log\langle u_t^{(a)},u_t^{(b)}\rangle
=
\lambda\Bigl(\langle a,b\rangle-\tfrac12\|a\|_2^2-\tfrac12\|b\|_2^2\Bigr).
\end{equation}

Let $u^{(1)}$ denote the unit corresponding to $a\equiv 1$. Then $c_{\mathbb E}(u^{(1)},u^{(1)})=0$.
By the standard construction of the Hilbert space $H(\mathbb E)$ from the conditionally positive
definite kernel $c_{\mathbb E}$, the inner product of classes satisfies
\[
\bigl\langle [u^{(a)}]-[u^{(1)}],\ [u^{(b)}]-[u^{(1)}]\bigr\rangle_{H(\mathbb E)}
=
c_{\mathbb E}(u^{(a)},u^{(b)})-c_{\mathbb E}(u^{(a)},u^{(1)})-c_{\mathbb E}(u^{(1)},u^{(b)}).
\]
Using \eqref{eq:c-ab} and $c_{\mathbb E}(u^{(1)},u^{(1)})=0$ gives
\[
\bigl\langle [u^{(a)}]-[u^{(1)}],\ [u^{(b)}]-[u^{(1)}]\bigr\rangle_{H(\mathbb E)}
=
\lambda\,\langle a-1,\ b-1\rangle_{L^2(\eta)}.
\]
Hence the map
\[
\Psi:\ \mathrm{span}\{[u^{(a)}]-[u^{(1)}]\}\ \to\ L^2(L,\eta),
\qquad
\Psi([u^{(a)}]-[u^{(1)}])=\sqrt{\lambda}\,(a-1),
\]
is an isometry.

Finally, $\{\sqrt{\lambda}(a-1): a\ge 0,\ a\in L^2(L,\eta)\}$ is dense in $L^2(L,\eta)$:
given any bounded real $h\in L^2(L,\eta)$ and $\varepsilon>0$ small enough so that $1+\varepsilon h\ge 0$,
set $a:=1+\varepsilon h$; then $\sqrt{\lambda}(a-1)=\sqrt{\lambda}\,\varepsilon h$.
Since bounded real-valued functions are dense in $L^2(L,\eta)$, the range of $\Psi$ is dense.
Therefore $H(\mathbb E)\cong L^2(L,\eta)$ and
\[
\ind(\mathbb E)=\dim H(\mathbb E)=\dim L^2(L,\eta).
\]
\end{proof}

\begin{observation} 
(i) If $L^2(L,\eta)$ is finite-dimensional, which happens precisely when $\eta$ is supported on finitely many atoms, then 
$\ind\bigl(\mathbb E^{\boldsymbol\mu^{(P,L)}}\bigr)=\#\operatorname{supp}(\eta)$. 
Otherwise, $\ind\bigl(\mathbb E^{\boldsymbol\mu^{(P,L)}}\bigr)=\infty$.

(ii) Using the Wiener--It\^o chaos decomposition, which identifies $L^2$ of a Poisson point process with symmetric Fock space, one can show directly that the Poisson random-set system $\mathbb E^{\boldsymbol\mu^{(P,L)}}$ is isomorphic (as an Arveson system) to the symmetric Fock/CCR Arveson system $\{\Gamma_s(L^2(0,t)\otimes L^2(L,\eta))\}_{t>0}$. Hence, it is a type~$\mathrm{I}$ system with index $\dim L^2(L,\eta)$ by Arveson's classification of type~I systems \cite{arveson-continuous}.

The approach developed above is arguably more direct and probabilistic in nature and does not rely on Arveson’s classification theory of type~I Arveson systems.

(iii) A direct way to obtain a measurable factorizing family over $[0,1]\times L$ with
$\Delta_{s,t}\not\equiv 1$ is to randomize the Poisson intensity. We sketch the details below: take
$\rho=\tfrac12\delta_\lambda+\tfrac12\delta_{2\lambda}$ and define
\[
\mu_t^{(CP)}:=\tfrac12 P_{\lambda,t}+\tfrac12 P_{2\lambda,t},
\]
where $P_{\alpha,t}$ is the law on $\mathscr C_t$ of a Poisson point process on $(0,t)\times L$
with intensity $\alpha\,\mathrm{Leb}\!\restriction_{(0,t)}\otimes\eta$.
Let $\pi_t:=P_{\lambda,t}$. Then $\mu_t^{(CP)}\sim \pi_t$ and
\[
\frac{d\mu_t^{(CP)}}{d\pi_t}(Z)=\tfrac12+\tfrac12\,e^{-\lambda t}\,2^{N(Z)},
\qquad N(Z):=\#Z.
\]
Moreover, for $s,t>0$ set $\kappa_{s,t}:=(\mu_s^{(CP)}\otimes\mu_t^{(CP)})\circ\oplus_{s,t}^{-1}.$ Then the Radon--Nikodym derivative
$\Delta_{s,t}:=\frac{d\kappa_{s,t}}{d\mu_{s+t}^{(CP)}}$ $(\mu_{s+t}^{(CP)}\text{--a.e.})$
admits the explicit version
\[
\Delta_{s,t}(Z)
=
\frac{\Bigl(\frac12+\frac12e^{-\lambda s}2^{N_1(Z)}\Bigr)\Bigl(\frac12+\frac12e^{-\lambda t}2^{N_2(Z)}\Bigr)}
{\frac12+\frac12e^{-\lambda (s+t)}2^{N_1(Z)+N_2(Z)}}
\qquad(\mu_{s+t}^{(CP)}\text{--a.e. }Z\in\mathscr C_{s+t}),
\]
where $N_1(Z):=\#(Z\cap([0,s]\times L))$ and
$N_2(Z):=\#\bigl((Z\cap([s,s+t]\times L))-(s,0)\bigr)$.
As in Proposition~\ref{prop:Poisson-Lspread-factor}, the family
$\boldsymbol\mu^{(CP,L)}:=\{\mu_t^{(CP)}\}_{t>0}$ is a measurable factorizing family of measures.
Finally, since $\mu_t^{(CP)}\sim P_{\lambda,t}=\mu_t^{(P,L)}$ for each $t$, the Cox--Poisson mixture $\boldsymbol\mu^{(CP,L)}$
has the same factorizing measure type as $\boldsymbol\mu^{(P,L)}$, hence yields the same Arveson system
(up to isomorphism).
\end{observation}

\section{Type III random-set systems}
In this section we establish the main results of the paper by constructing
random-set systems of type $\mathrm{III}$. Starting from a suitable type $\mathrm{II}_0$
random-set seed, we introduce a marked infinite-product construction
indexed by $[0,1]\times\mathbb N$. The construction uses a sequence of
time-dilated copies of the seed and produces a measurable factorizing
family of probability measures on the hyperspaces of closed subsets of
$[0,t]\times\mathbb N$. Using Kakutani’s theorem for infinite products
of measures together with quantitative Hellinger estimates, we prove
that the resulting factorizing family generates a random-set
system with no units.

The key technical ingredient is a sufficient criterion ensuring
Hellinger-smallness of the seed. This criterion is formulated in terms
of block-occupancy events and quantitative control of one-block
restriction laws. We verify it for seeds arising from Brownian zero
sets after applying anchor-adapted localization, Palm-type
uniformization of the anchor distribution, and normalization of the
vacuum mass. This yields an explicit example of a type $\mathrm{III}$ random-set system generated by Brownian zero sets.

\subsection{Hellinger-small seeds}
\begin{definition}
A seed is a measurable factorizing family $\boldsymbol\nu=\{\nu_t\}_{t>0}$ over $[0,1]\times \{\ast\}$ such that the associated random-set system $\mathbb{E}^{\boldsymbol\nu}:=\{  E_t^{\boldsymbol\nu}, U_{s,t}^{\boldsymbol\nu}\}_{t>0}$ is type~$\mathrm{II}_0$.
\end{definition}

Throughout we fix $L:=\N$ with the discrete topology.
A closed set $Z\subset [0,t]\times \N$ can be regarded as a sequence
$Z =\{Z^{(n)}\}_{n\ge 1}$ with
\[
Z^{(n)}:=\{u\in[0,t]:(u,n)\in Z\}\in \mathscr C_t^{\{\ast\}},
\]
because $\N$ is discrete and the slices $[0,t]\times\{n\}$ are clopen.

The following lemma identifies the hyperspace $\mathscr C_t^{\{\N\}}$ with a countable product of one–dimensional hyperspaces.
\begin{lemma}
\label{lem:borel-product-id}
For each $t>0$, the map $\Phi_t:\mathscr C_t^{\{\N\}}\to \prod_{n\ge 1}\mathscr C_t^{\{\ast\}},$ $\Phi_t(Z):=\bigl(Z^{(n)}\bigr)_{n\ge 1},$
is a Borel isomorphism with inverse
$\Psi_t\bigl((Z^{(n)})_{n\ge 1}\bigr)
:=
\bigcup_{n\ge 1}\bigl(Z^{(n)}\times\{n\}\bigr)\subset[0,t]\times\N.$
\end{lemma}

\begin{proof}
Since $\N$ is discrete, $[0,t]\times\N$ is a disjoint union of clopen sets $[0,t]\times\{n\}$.
Therefore a set $Z\subset[0,t]\times\N$ is closed if and only if each slice $Z\cap([0,t]\times\{n\})$ is closed.
Hence $\Phi_t$ and $\Psi_t$ are mutually inverse bijections. 

To see measurability, note that, by Lemma~\ref{lem:void-generate}, the $\sigma$--field on $\mathscr C_t^{\{\N\}}$ is generated by void events
of the form $\{Z:Z\cap(U\times\{n\})=\varnothing\}$ where $U\subset[0,t]$ is open and $n\in\N$.
Under $\Phi_t$, such a void event becomes the cylinder event
$\{(Z^{(k)}): Z^{(n)}\cap U=\varnothing\}$, which is Borel in the product $\sigma$--field.
Thus $\Phi_t$ is Borel. The same argument shows $\Psi_t$ is Borel.
\end{proof}

In what follows we shall freely use the identification $\mathscr C_t^{\{\N\}}\cong \prod_{n\ge 1}\mathscr C_t^{\{\ast\}}$ as standard Borel spaces, given by Lemma~\ref{lem:borel-product-id}.

Let now $\boldsymbol\nu=\{\nu_t\}_{t>0}$ be a seed. Choose a positive sequence $\{a_n\}_{n\ge 1}$ such that
\begin{equation}\label{eq:an}
\sum_{n=1}^\infty a_n = \infty,
\qquad
\sum_{n=1}^\infty a_n^2 < \infty,
\end{equation}
(e.g.\ $a_n:=\frac1n$).
For each $n\ge 1$ and each $t>0$, define a probability measure $\nu_t^{(n)}$ on $\mathscr C_t^{\{\ast\}}$ by
\begin{equation}\label{eq:scaled-seed}
\nu_t^{(n)} := (D_{a_n,t})_*\nu_{a_n t},
\end{equation}
where $D_{a_n,t}:\mathscr C_{a_nt}^{\{\ast\}}\to \mathscr C_t^{\{\ast\}}$ is the time-dilation map $D_{a_n,t}(Z):=\{u/a_n:\ u\in Z\}\subset [0,t].$

For later use, set $\sigma_{s,t}^{(n)} := (\nu_s^{(n)}\otimes \nu_t^{(n)})\circ \oplus_{s,t}^{-1}$, $s,t>0,$ so that $\nu_{s+t}^{(n)}\sim\sigma_{s,t}^{(n)}$.
Define the Hellinger integral
\[
H_{s,t}^{(n)}:=\int_{\mathscr C_{s+t}^{\{\ast\}}}
\sqrt{
\frac{\mathrm d \nu_{s+t}^{(n)}}{\mathrm d \sigma_{s,t}^{(n)}}
}\ \mathrm d \sigma_{s,t}^{(n)}
\in(0,1].
\]

\begin{definition}\label{ass:hellinger}
A seed $\boldsymbol\nu=\{\nu_t\}_{t>0}$ is said to be Hellinger-small if for every fixed $s,t>0$
there exists a finite constant $C_{s,t}$ such that for all sufficiently small $\lambda>0$,
\begin{equation}\label{eq:hellinger-small}
1 - H_{\lambda s,\lambda t}^{(\nu)} \ \le\ C_{s,t}\,\lambda^2,
\end{equation}
where $H_{\lambda s,\lambda t}^{(\nu)}$ is the Hellinger integral between $\nu_{\lambda(s+t)}$ and
$(\nu_{\lambda s}\otimes \nu_{\lambda t})\circ\oplus^{-1}$.
\end{definition}
This is a natural $L^2$-smallness condition at short times. It states that the square-root density between the
two equivalent laws differs from $1$ at order $O(\lambda)$, and therefore the Hellinger deficit is $O(\lambda^2)$.
\begin{definition}\label{def:mu-t}
For each $t>0$, define the probability measure
$\mu_t^{(\nu)} := \bigotimes_{n=1}^\infty \nu_t^{(n)}$ on $\mathscr C_t^{\{\N\}}$.
\end{definition}
Equivalently, if $(Z_t^{(n)})_{n\ge 1}$ are independent with $Z_t^{(n)}\sim \nu_t^{(n)}$, then
$Z_t := \bigcup_{n\ge 1} \bigl(Z_t^{(n)}\times\{n\}\bigr)\subset [0,t]\times\N$
is a random closed set with law $\mu_t^{(\nu)}$.

The next proposition shows that the infinite product of scaled copies of a Hellinger-small seed yields a measurable factorizing family.

\begin{proposition}\label{th:meas-seed}
Let $\{a_n\}_{n\ge 1}$ satisfy \eqref{eq:an}, and let $\boldsymbol\nu=\{\nu_t\}_{t>0}$ be a Hellinger-small seed.
Then $\boldsymbol\mu_\nu:=\{\mu_t^{(\nu)}\}_{t>0}$ is a measurable factorizing family over $[0,1]\times\N$.
\end{proposition}

\begin{proof}
First, we show that $\boldsymbol\mu_{\nu}=\{\mu_t^{\nu}\}_{t>0}$ is factorizing in the sense of Definition~\ref{def:fact-meas}(i). For this, fix $s,t>0$.
Using the identification $\mathscr C_{s+t}^{\{\N\}}\ \cong\ \prod_{n\ge 1}\mathscr C_{s+t}^{\{\ast\}}$,
we have $\mu_{s+t}=\bigotimes_{n\ge 1}\nu_{s+t}^{(n)}$ and $(\mu_s\otimes\mu_t)\circ\oplus_{s,t}^{-1}
=
\bigotimes_{n\ge 1}\sigma_{s,t}^{(n)}.$
By Kakutani's theorem for infinite product measures \cite{Kakutani}, these two infinite product measures are equivalent if and only if
\[
\prod_{n\ge 1} H_{s,t}^{(n)} \ >\ 0.
\]
Now $H_{s,t}^{(n)}$ is the Hellinger integral between the scaled laws at time $s,t$.
By scaling \eqref{eq:scaled-seed}, we have
\[
H_{s,t}^{(n)} = H_{a_n s, a_n t}^{(\nu)}.
\]
Since the seed $\boldsymbol\nu=\{\nu_t\}_{t>0}$ is Hellinger-small, we have
$1-H_{s,t}^{(n)} \ \le\ C_{s,t}\,a_n^2$, for all large $n$.
But $\sum_n a_n^2<\infty$ so $\sum_n (1-H_{s,t}^{(n)})<\infty$, hence $\prod_n H_{s,t}^{(n)}>0$.
Therefore the infinite product measures are equivalent.

To verify the absence of deterministic time-slices, fix $t>0$ and $r\in[0,t]$, and consider the event
$A_{t,r}:=\Bigl\{Z\in\mathscr C_t^{\{\N\}}:\ Z\cap(\{r\}\times\N)\neq\varnothing\Bigr\}.$
Under the identification $\Phi_t(Z)=(Z^{(n)})_{n\ge1}$, we have $A_{t,r}=\bigcup_{n\ge 1}\Bigl\{(Z^{(k)})_{k\ge1}:\ r\in Z^{(n)}\Bigr\}.$
Therefore, by countable subadditivity and the product structure,
\[
\mu_t^{(\nu)}(A_{t,r})
\le
\sum_{n\ge1} \nu_t^{(n)}\bigl(\{Z\in\mathscr C_t^{\{\ast\}}:\ r\in Z\}\bigr).
\]
Now $\nu_t^{(n)}=(D_{a_n,t})_*\nu_{a_n t}$, and $r\in D_{a_n,t}(Z')$ iff $a_n r\in Z'$.
Hence $\nu_t^{(n)}(\{Z:r\in Z\})
=
\nu_{a_n t}(\{Z': a_n r\in Z'\}).$
Since the seed $\boldsymbol\nu$ is itself a factorizing family, it satisfies
Definition~\ref{def:fact-meas}\textup{(ii)} with $L=\{\ast\}$, i.e.
$\nu_u(\{Z':\rho\in Z'\})=0$ for all $u>0$ and all $\rho\in[0,u]$.
Applying this with $u=a_n t$ and $\rho=a_n r$ gives $\nu_{a_n t}(\{Z':a_nr\in Z'\})=0$.
Thus each summand is $0$, hence $\mu_t^{(\nu)}(A_{t,r})=0$.
This proves Definition~\ref{def:fact-meas}\textup{(ii)}.

Next, we check the non-degeneracy condition (iii).
Fix $t>0$. Since the seed satisfies Definition~\ref{def:fact-meas}\textup{(iii)},
each $\nu_u$ is not supported on finitely many atoms, hence in particular is not a Dirac mass.
Therefore for each $u>0$ there exist disjoint Borel sets $B_u,C_u\subset\mathscr C_u^{\{\ast\}}$ with
$\nu_u(B_u)>0$ and $\nu_u(C_u)>0$.
(For instance, since $\nu_u$ is not a Dirac mass there exists some Borel set $B_u$ with
$0<\nu_u(B_u)<1$; then $C_u:=B_u^c$ works.)

For each $n$, apply this to $u=a_n t$ and choose $B_{a_n t}\subset\mathscr C_{a_n t}^{\{\ast\}}$
with $0<\nu_{a_n t}(B_{a_n t})<1$; set $B_t^{(n)}:=D_{a_n,t}(B_{a_n t})\subset\mathscr C_t^{\{\ast\}}.$ Then $0<\nu_t^{(n)}(B_t^{(n)})<1$ and also $0<\nu_t^{(n)}((B_t^{(n)})^c)<1$.

Fix $m\in\N$. For each $\varepsilon=(\varepsilon_1,\dots,\varepsilon_m)\in\{0,1\}^m$ define the cylinder set
\[
C_\varepsilon
:=
\Bigl(\prod_{j=1}^m A_j^{(\varepsilon_j)}\Bigr)\times \prod_{n>m}\mathscr C_t^{\{\ast\}},
\qquad
A_j^{(1)}:=B_t^{(j)},\ \ A_j^{(0)}:=(B_t^{(j)})^c.
\]
The sets $C_\varepsilon$ are pairwise disjoint, and $\mu_t^{(\nu)}(C_\varepsilon)
=
\prod_{j=1}^m \nu_t^{(j)}(A_j^{(\varepsilon_j)})\ >\ 0.$
Hence for every $m$ there exist $2^m$ disjoint Borel sets of strictly positive $\mu_t^{(\nu)}$--mass.
In particular, $\mu_t^{(\nu)}$ cannot be supported on finitely many atoms.
This proves Definition~\ref{def:fact-meas}\textup{(iii)}.

We focus on the measurability condition (iva) of  Definition~\ref{def:fact-meas}.
Let $\sigma_t:\mathscr C_t^{\{\ast\}}\to \mathscr C_1^{\{\ast\}}$ be the time-scaling homeomorphism
$\sigma_t(Z)=\{r/t:\ r\in Z\}$, and let $\sigma_t^{\N}:\mathscr C_t^{\{\N\}}\to \mathscr C_1^{\{\N\}}$ be the
time-scaling on $[0,t]\times\N$, i.e.\ $\sigma_t^{\N}(Z)=\{(r/t,n): (r,n)\in Z\}$.
Write
\[
\widetilde\nu_t:=(\sigma_t)_*\nu_t,\qquad
\widetilde\mu_t:=(\sigma_t^{\N})_*\mu_t^{(\nu)}.
\]
Since the seed is measurable, Definition~\ref{def:fact-meas}\textup{(iva)} provides a fixed countable ring
$\mathcal R_{\ast}\subset\Sigma_1^{\{\ast\}}$ generating $\Sigma_1^{\{\ast\}}$ such that
$t\mapsto \widetilde\nu_t(A)$ is Borel for every $A\in\mathcal R_{\ast}$.

\smallskip
We notice that $\widetilde\nu^{(n)}_t:=(\sigma_t)_*\nu_t^{(n)}$ satisfies
\begin{equation}\label{eq:tilde-nu-n}
\widetilde\nu^{(n)}_t = \widetilde\nu_{a_n t}\qquad(t>0,\ n\ge1).
\end{equation}
Indeed, $\sigma_t\circ D_{a_n,t}=\sigma_{a_n t}$ as maps on closed subsets, hence $\widetilde\nu^{(n)}_t
=(\sigma_t)_*(D_{a_n,t})_*\nu_{a_n t}
=(\sigma_t\circ D_{a_n,t})_*\nu_{a_n t}
=(\sigma_{a_n t})_*\nu_{a_n t}
=\widetilde\nu_{a_n t},
$ as claimed.

Under the product identification (Lemma~\ref{lem:borel-product-id}), the map $\sigma_t^{\N}$ acts
coordinatewise, hence
\[
\widetilde\mu_t
=
\bigotimes_{n\ge1} \widetilde\nu^{(n)}_t
=
\bigotimes_{n\ge1} \widetilde\nu_{a_n t}.
\]

\smallskip
Now define a countable ring $\mathcal R_{\N}\subset\Sigma_1^{\{\N\}}$ as the cylinder ring
generated by sets of the form $\pi_n^{-1}(A),$ for all $n\in\N$ and $A\in\mathcal R_{\ast},$
where $\pi_n:\mathscr C_1^{\{\N\}}\to\mathscr C_1^{\{\ast\}}$ is the $n$th slice projection.
Equivalently, $\mathcal R_{\N}$ consists of all finite unions of finite intersections of cylinder sets
\[
\bigcap_{j=1}^m \pi_{n_j}^{-1}(A_j),\qquad A_j\in\mathcal R_{\ast},\ n_j\in\N.
\]
This ring is countable, and generates $\Sigma_1^{\{\N\}}$ because the
product $\sigma$--field is generated by cylinder sets. For a generator $B=\bigcap_{j=1}^m \pi_{n_j}^{-1}(A_j)$, using the infinite product structure of
$\widetilde\mu_t$ and \eqref{eq:tilde-nu-n}, we get
\[
\widetilde\mu_t(B)
=
\prod_{j=1}^m \widetilde\nu_{a_{n_j} t}(A_j).
\]
Each factor $t\mapsto \widetilde\nu_{a_{n_j} t}(A_j)$ is Borel because
$u\mapsto\widetilde\nu_u(A_j)$ is Borel by seed measurability and $u=a_{n_j}t$ is continuous in $t$.
Hence $t\mapsto \widetilde\mu_t(B)$ is Borel for generators, and therefore for all $B\in\mathcal R_{\N}$
. This proves Definition~\ref{def:fact-meas}\textup{(iva)}.

Finally, we verify the measurability condition (ivb).
Since the seed is measurable, Definition~\ref{def:fact-meas}\textup{(ivb)} provides a Borel map $\Delta^{(\nu)}:(0,\infty)^2\times \mathscr C_1^{\{\ast\}}\to(0,\infty)$
such that, for each $s,t>0$, $\Delta^{(\nu)}(s,t,\sigma_{s+t}(Z))$ equals the Radon--Nikodym derivative
\[
\frac{d\bigl((\nu_s\otimes\nu_t)\circ\oplus_{s,t}^{-1}\bigr)}{d\nu_{s+t}}(Z)
\qquad(\nu_{s+t}\text{-a.e. }Z\in\mathscr C_{s+t}^{\{\ast\}}).
\]

For $n\ge1$, define the coordinate cocycle $\Delta^{(n)}(s,t,W):=\Delta^{(\nu)}(a_n s,a_n t,W),$ for all 
$ s,t>0$ and $W\in\mathscr C_1^{\{\ast\}}.$
This is Borel because $\Delta^{(\nu)}$ is Borel and $(s,t)\mapsto(a_ns,a_nt)$ is continuous.

Let $W=(W^{(n)})_{n\ge1}\in\mathscr C_1^{\{\N\}}\cong\prod_{n\ge1}\mathscr C_1^{\{\ast\}}$.
Define partial products
\[
\Delta^{[N]}(s,t,W):=\prod_{n=1}^N \Delta^{(n)}(s,t,W^{(n)}),
\qquad N\in\N.
\]
Each $\Delta^{[N]}$ is Borel on $(0,\infty)^2\times\mathscr C_1^{\{\N\}}$. Fix $s,t>0$. Set
\[
\widetilde\nu_{s,t}:=(\sigma_{s+t}^{\N})_*\Bigl((\mu_s^{(\nu)}\otimes\mu_t^{(\nu)})\circ\oplus_{s,t}^{-1}\Bigr)
\quad\text{on }\mathscr C_1^{\{\N\}},
\]
and recall $\widetilde\mu_{s+t}=(\sigma_{s+t}^{\N})_*\mu_{s+t}^{(\nu)}$.
Then $\widetilde\nu_{s,t}\ll \widetilde\mu_{s+t}$ by the factorization property.
Let $M_{s,t}:=\frac{d\widetilde\nu_{s,t}}{d\widetilde\mu_{s+t}}$ be a Radon--Nikod\'ym derivative. We claim that for each $N$, one has
\[
\Delta^{[N]}(s,t,\cdot)=\mathbb E_{\widetilde\mu_{s+t}}\!\bigl[M_{s,t}\,\big|\,\mathcal F_N\bigr],
\]
where $\mathcal F_N$ is the $\sigma$--field generated by the first $N$ coordinates $(W^{(1)},\dots,W^{(N)})$.
Indeed, for any cylinder set $B\in\mathcal F_N$,
the measures $\widetilde\mu_{s+t}$ and $\widetilde\nu_{s,t}$ factorize into infinite products,
and on the first $N$ coordinates the density is exactly the finite product of the coordinate densities.
Concretely,
\[
\widetilde\nu_{s,t}(B)=\int_B \Delta^{[N]}(s,t,W)\,d\widetilde\mu_{s+t}(W),
\]
and also $\widetilde\nu_{s,t}(B)=\int_B M_{s,t}\,d\widetilde\mu_{s+t}$ by definition of $M_{s,t}$.
Thus $\Delta^{[N]}(s,t,\cdot)$ is the conditional expectation of $M_{s,t}$ onto $\mathcal F_N$.

By the martingale convergence theorem,
$\Delta^{[N]}(s,t,W)\to M_{s,t}(W)$ for $\widetilde\mu_{s+t}$--a.e.\ $W$.
Define
\[
\Delta(s,t,W):=\limsup_{N\to\infty}\Delta^{[N]}(s,t,W).
\]
Then $\Delta$ is Borel and satisfies $\Delta(s,t,\cdot)=\frac{d\widetilde\nu_{s,t}}{d\widetilde\mu_{s+t}}$ $(\widetilde\mu_{s+t}\text{--a.e.}).$
Finally, define $\Delta_{s,t}$ on $\mathscr C_{s+t}^{\{\N\}}$ by
\[
\Delta_{s,t}(Z):=\Delta\bigl(s,t,\sigma_{s+t}^{\N}(Z)\bigr).
\]
By construction, $\Delta$ is jointly Borel and satisfies Definition~\ref{def:fact-meas}\textup{(ivb)}. The proof is now complete.
\end{proof}
\begin{definition}
For $N\in\N$ and $t>0$, let $\pi_{\le N,t}:\prod_{n\ge1}\mathscr C_t^{\{\ast\}}\longrightarrow \prod_{n=1}^N \mathscr C_t^{\{\ast\}}$
be the coordinate projection. Under Lemma~\ref{lem:borel-product-id} we view $\pi_{\le N,t}$
as a Borel map $\mathscr C_t^{\{\N\}}\to \prod_{n=1}^N\mathscr C_t^{\{\ast\}}$.
We then define the $N$th marginal measure
\begin{equation}\label{eq:marginals}
\mu_t^{(\le N)} := (\pi_{\le N,t})_*\mu_t^{(\nu)} = \bigotimes_{n=1}^N \nu_t^{(n)}.
\end{equation}
\end{definition}
The next lemma analyzes the finite-coordinate marginals of the infinite-product construction.
\begin{lemma}
\label{lem:finite-marginal-II0}
For each $N\in\N$, the family $\boldsymbol\mu^{(\le N)}:=\{\mu_t^{(\le N)}\}_{t>0}$ is a measurable
factorizing family over $[0,1]\times\{1,\dots,N\}$, (equivalently on the product space
$\prod_{n=1}^N\mathscr C_t^{\{\ast\}}$). Moreover, its associated random-set system
$\mathbb E^{(\le N)}$ is of type
$\mathrm{II}_0$. 
\end{lemma}

\begin{proof}
Measurability and factorization for $\boldsymbol\mu^{(\le N)}$ are immediate because it is a finite
product of the measurable factorizing families $\{\nu_t^{(n)}\}_{t>0}$:
the concatenation map acts coordinatewise, and Radon--Nikod\'ym derivatives multiply over finitely many
coordinates.

Each factor system associated with $\{\nu_t^{(n)}\}$ is isomorphic to the seed system
$\mathbb E^{\boldsymbol\nu}$ by time-dilation (hence is type $\mathrm{II}_0$ and has a unique normalized
unit line). The $N$-fold tensor product of type $\mathrm{II}_0$ product systems again has a unique
normalized unit line, spanned by the tensor product of the factor units. Therefore $\mathbb E^{(\le N)}$
is type $\mathrm{II}_0$.
\end{proof}

\begin{definition}\label{def:n-dilated-unit}Let $u^{(\nu)}=\{u_t^{(\nu)}\}_{t>0}$ be the (necessarily unique up to phase) positive normalized unit of the
type $\mathrm{II}_0$ seed system $\mathbb E^{\boldsymbol\nu}$. For each $n\ge1$ define a normalized unit $u^{(n)}=\{u_t^{(n)}\}_{t>0}$ of the $n$th factor system
associated with $\{\nu_t^{(n)}\}$ by transport under the dilation map:
\[
u_t^{(n)}(Z) := u_{a_n t}^{(\nu)}\bigl(D_{a_n,t}^{-1}(Z)\bigr),
\qquad Z\in\mathscr C_t^{\{\ast\}}.
\]
Let $\1\in L^2(\nu_t)$ denote the constant-one function. Define the seed ``vacuum overlap''
\[
m(t):=\langle \1, u_t^{(\nu)}\rangle_{L^2(\nu_t)} \in (0,1].
\]
\end{definition}
\begin{remark}Notice that for each $n$ and $t$, one has
\begin{equation}\label{eq:mn-scaling}
m_n(t):=\langle \1, u_t^{(n)}\rangle_{L^2(\nu_t^{(n)})} = m(a_n t).
\end{equation}
Indeed, 
$\langle \1, u_t^{(n)}\rangle_{L^2(\nu_t^{(n)})}
=\int u_{a_nt}^{(\nu)}(D_{a_n,t}^{-1}(Z))\,d(D_{a_n,t})_*\nu_{a_nt}(Z)
=\int u_{a_nt}^{(\nu)}(Z')\,d\nu_{a_nt}(Z')
=m(a_nt).$
\end{remark}
The following theorem gives a simple sufficient condition ensuring that the infinite-product construction produces a type~$\mathrm{III}$ system.
\begin{theorem}
\label{th:typeIII-infiniteproduct}
Let $\{a_n\}_{n\ge1}$ satisfy \eqref{eq:an}, and let $\boldsymbol\nu$ be a Hellinger-small seed.
Assume there exist constants $c>0$ and $t_0>0$ such that
\begin{equation}\label{eq:linear-overlap}
1-m(t)\ \ge\ c\,t\qquad(0<t\le t_0).
\end{equation}
Then the random-set system $\mathbb E^{\boldsymbol\mu_\nu}$ associated with
$\boldsymbol\mu_\nu=\{\mu_t^{(\nu)}\}_{t>0}$ is of type $\mathrm{III}$.
\end{theorem}

\begin{proof}
Assume for contradiction that $\mathbb E^{\boldsymbol\mu_\nu}$ is spatial.  Using Theorem \ref{prop:spatial-iff-exact-factor-meas}, there exists a measurable family $\boldsymbol\eta=\{\eta_t\}_{t>0}$ of probability measures with $\boldsymbol\eta\ll \boldsymbol\mu_\nu$ and $\eta_{s+t}=(\eta_s\otimes\eta_t)\circ\oplus_{s,t}^{-1}$ for all $s,t>0$.

For each $N\in\N$ and $t>0$, define the marginal measures $\eta_t^{(\le N)} := (\pi_{\le N,t})_*\eta_t$
on $\prod_{n=1}^N\mathscr C_t^{\{\ast\}}.$ Because $\pi_{\le N,t}$ commutes with concatenation (coordinatewise) we have exact factorization:
\[
\eta_{s+t}^{(\le N)}=(\eta_s^{(\le N)}\otimes \eta_t^{(\le N)})\circ\oplus_{s,t}^{-1}.
\]
Also, since $\eta_t\ll\mu_t^{(\nu)}$, the marginal satisfies $\eta_t^{(\le N)}\ll \mu_t^{(\le N)}$, where $\mu_t^{(\le N)}$ is the marginal \eqref{eq:marginals}. Thus $\boldsymbol\eta^{(\le N)}:=\{\eta_t^{(\le N)}\}_{t>0}$ is an exactly factorizing family
dominated by $\boldsymbol\mu^{(\le N)}$, hence it corresponds to a positive normalized unit in
$\mathbb E^{(\le N)}$.

By Lemma~\ref{lem:finite-marginal-II0}, $\mathbb E^{(\le N)}$ is of type $\mathrm{II}_0$, hence it has a
unique positive normalized unit. Concretely, let  $u^{(n)}=\{u_t^{(n)}\}_{t>0}$ be the positive normalized unit of the
$n$th coordinate system (the $a_n$--dilation of $u^{(\nu)}$), as in Definition \ref{def:n-dilated-unit}, and let $\omega_t^{(n)}$ be its associated
factorizing measure family on $\mathscr C_t^{\{\ast\}}$:
\begin{equation}\label{eq:omega-n}
d\omega_t^{(n)} := |u_t^{(n)}|^2\,d\nu_t^{(n)},
\qquad t>0.
\end{equation}
Then the unique positive unit in $\mathbb E^{(\le N)}$ is the tensor product
$u^{(1)}\cdots u^{(N)}$, and its associated factorizing measure family is $\omega_t^{(\le N)} := \bigotimes_{n=1}^N \omega_t^{(n)}.$ Uniqueness of the positive normalized unit implies
\begin{equation}\label{eq:eta-marg-equals-omega}
\eta_t^{(\le N)}=\omega_t^{(\le N)}\qquad\text{for all }t>0\text{ and }N\in\N.
\end{equation}

For every $t>0$, let $\mathcal C_t$ be the collection of cylinder sets in
$\prod_{n\ge1}\mathscr C_t^{\{\ast\}}\cong\mathscr C_t^{\{\N\}}$ depending on finitely many coordinates.
Then $\mathcal C_t$ is a $\pi$--system generating the full product $\sigma$--field.
For a cylinder set $B$ depending only on the first $N$ coordinates we have
$\1_B=\1_{\pi_{\le N,t}^{-1}(B_N)}$ for some $B_N$ in the $N$--fold product $\sigma$--field, hence
$\eta_t(B)=\eta_t^{(\le N)}(B_N)=\omega_t^{(\le N)}(B_N)=\omega_t(B),$
where $\omega_t:=\bigotimes_{n\ge1}\omega_t^{(n)}$.
By the $\pi$--$\lambda$ theorem,
this implies $\eta_t=\omega_t$ for all $t>0.$ Therefore, the family
$\{\eta_t^{(\le N)}\}_{N\ge1}$ determines $\eta_t$ uniquely.

Since $\mu_t^{(\nu)}=\bigotimes_{n\ge1}\nu_t^{(n)}$ and
$\omega_t=\bigotimes_{n\ge1}\omega_t^{(n)}$, for each fixed $t>0$, Kakutani's theorem implies that $\omega_t\ll\mu_t^{(\nu)}$ can only hold if the
infinite product of Hellinger affinities is strictly positive; namely if
\[
\prod_{n\ge1}\int_{\mathscr C_t^{\{\ast\}}} \sqrt{\frac{d\omega_t^{(n)}}{d\nu_t^{(n)}}}\,d\nu_t^{(n)}
=
\prod_{n\ge1}\langle \1, u_t^{(n)}\rangle
=
\prod_{n\ge1} m(a_n t)
\;>\;0.
\]
We show this product equals $0$ for sufficiently small $t$. For this, fix $t\in(0,t_0]$. For all large $n$ we have $a_n t\le t_0$ (since $a_n\to0$), hence by
\eqref{eq:linear-overlap} and \eqref{eq:mn-scaling}, $1-m(a_n t)\ \ge\ c\,a_n t.$
Because $\sum_n a_n=\infty$, it follows that $\sum_{n=1}^\infty \bigl(1-m(a_n t)\bigr)=\infty.$ Using $-\log x\ge 1-x$ for $x\in(0,1]$, we get $\sum_{n=1}^\infty -\log m(a_n t)=\infty$,
hence
\[\prod_{n\ge1} m(a_n t)=\exp\!\Bigl(-\sum_{n\ge1}-\log m(a_n t)\Bigr)=0,
\]
as claimed. Therefore $\omega_t$ and $\mu_t^{(\nu)}$ are mutually singular (in particular $\omega_t\not\ll \mu_t^{(\nu)}$).

But, as shown above, the existence of a factorizing $\boldsymbol\eta\ll\boldsymbol\mu_\nu$ forces $\eta_t$ to coincide with the product construction $\omega_t$. This contradicts $\eta_t\ll\mu_t^{(\nu)}$ because $\omega_t\perp\mu_t^{(\nu)}$.
Thus no such exactly factorizing family $\boldsymbol\eta$ exists, hence $\mathbb E^{\boldsymbol\mu_\nu}$
has no units and is type $\mathrm{III}$.
\end{proof}

\subsection{Palm-uniformized representatives and Hellinger-small seeds}\label{subsec:palm-unif-small}

In this subsection we construct convenient representatives of factorizing families by combining anchor-adapted localization with Palm uniformization. The resulting measures have a uniform anchor and strong diameter control, properties that will be crucial for establishing the small-seed estimates used later.
\begin{definition}[Anchor and spread]\label{def:anchor-spread}
For $t>0$, define the anchor map $\alpha_t:\mathscr C_t^{\{\ast\}}\to[0,t]$,
\[
\alpha_t(Z):=\begin{cases}
\inf Z,& Z\neq\varnothing,\\
0,& Z=\varnothing,
\end{cases}
\]
and the spread map $\mathrm{diam}_t:\mathscr C_t^{\{\ast\}}\to[0,t],$
\[
\mathrm{diam}_t(Z):=\begin{cases}
\sup Z-\inf Z,& Z\neq\varnothing,\\
0,& Z=\varnothing.
\end{cases}
\]
\end{definition}

 We verify below that these basic hyperspace functionals are measurable, hence admissible for tilting and conditioning.

\begin{lemma}
\label{lem:anchor-borel}
For each $t>0$ the maps $Z\mapsto \alpha_t(Z)$ and $Z\mapsto \mathrm{diam}_t(Z)$ are Borel.
\end{lemma}

\begin{proof}
Fix $a\in[0,t]$. Since $\alpha_t(\varnothing)=0$, for $a>0$ we have $\{\alpha_t(Z)>a\}
=
\{Z\neq\varnothing\}\cap\{Z\cap[0,a]=\varnothing\}.$
The event $\{Z\cap[0,a]=\varnothing\}$ is a miss set for the compact $[0,a]\subset[0,t]$, hence Borel. Also $\{Z=\varnothing\}=\{Z\cap[0,t]=\varnothing\}$ is a miss set for the compact $[0,t]$,
hence $\{Z\neq\varnothing\}$ is Borel. Therefore $\{\alpha_t>a\}$ is Borel for all $a$,
so $\alpha_t$ is Borel.

To show that $\mathrm{diam}_t$ is Borel, define \begin{equation}\label{eq:sup}s_t(Z):=0\;\text{if}\;Z=\varnothing,\quad \text{and}\; s_t(Z):=\sup Z\; \text{if}\; Z\neq\varnothing.\end{equation}
Then for any $b\in(0,t]$, we have $\{s_t(Z)<b\}=\{Z\cap[b,t]=\varnothing\},$ which is again a miss event for the compact $[b,t]$, hence Borel. Thus $s_t$ is Borel. Finally, $\mathrm{diam}_t(Z)=s_t(Z)-\alpha_t(Z)$, hence $\mathrm{diam}_t$ is Borel.
\end{proof}
\begin{notation}
Let $\mu_t$ be a probability measure on $\mathscr C_t^{\{\ast\}}$, for some fixed $t>0$. We set
\[
p_t^{\mu}:=\mu_t(\{\varnothing\})\in[0,1),\qquad
\mu_t^{\neq}:=\mu_t(\,\cdot\cap\{Z\neq\varnothing\}),
\qquad
\kappa_t^{\mu}:=(\alpha_t)_*\mu_t^{\neq}
\ \ \text{on }(0,t].
\]
\end{notation}
\begin{definition}\label{def:palm-unif}
We say that a family $\boldsymbol\mu=\{\mu_t\}_{t>0}$ of probability measures $\mu_t$ on $\mathscr C_t^{\{\ast\}}$ is Palm-uniformizable 
if for every $t>0$ one has $\kappa_t^{\mu}\sim \mathrm{Leb}\!\restriction_{(0,t]}$.
\end{definition}
The next lemma constructs a representative in the same measure class for which the anchor becomes exactly uniformly distributed.

\begin{lemma}\label{lem:palm-uniformization} Let $\mu_t$ be a probability measure on $\mathscr C_t^{\{\ast\}}$, for some fixed $t>0$. 
Assume $\kappa_t^{\mu}\sim \mathrm{Leb}\!\restriction_{(0,t]}$, and choose a Borel version $h_t:(0,t]\to(0,\infty)$ of the Radon--Nikodym derivative
$h_t=d(\mathrm{Leb}\!\restriction_{(0,t]})/d\kappa_t^{\mu}$. Define a new probability measure $\mu_t^{\mathrm{unif}}$ on $\mathscr C_t^{\{\ast\}}$ by
\begin{equation}\label{eq:mu-unif-def}
d\mu_t^{\mathrm{unif}}(Z)
:=
\frac{1}{C_t}\Bigl(
\1_{\{\varnothing\}}(Z)
+\1_{\{Z\neq\varnothing\}}(Z)\,h_t(\alpha_t(Z))
\Bigr)\,d\mu_t(Z),
\end{equation}
where $C_t$ is the normalizing constant.

Then:
\begin{enumerate}[label=\textup{(\roman*)}]
\item $\mu_t^{\mathrm{unif}}\sim \mu_t$;
\item $\mu_t^{\mathrm{unif}}(\{\varnothing\})=p_t^{\mu}/C_t$, and $C_t=p_t^\mu+t$;
\item under $\mu_t^{\mathrm{unif}}(\cdot\,|\,Z\neq\varnothing)$, the anchor $\alpha_t(Z)$ is exactly
$\mathrm{Unif}(0,t)$.
\end{enumerate}
\end{lemma}

\begin{proof}
(i) The density in \eqref{eq:mu-unif-def} is strictly positive $\mu_t$--a.e.:
it equals $1$ on $\{\varnothing\}$, and equals $h_t(\alpha_t(Z))$ on $\{Z\neq\varnothing\}$.
Since $\kappa_t^\mu\sim\mathrm{Leb}$, we may (and do) choose $h_t>0$ $\kappa_t^\mu$--a.e., hence
$h_t(\alpha_t(Z))>0$ $\mu_t$--a.e. on $\{Z\neq\varnothing\}$.
Thus $\mu_t^{\mathrm{unif}}\sim\mu_t$.

(ii) Immediate from \eqref{eq:mu-unif-def}: $\mu_t^{\mathrm{unif}}(\{\varnothing\})
=\frac{1}{C_t}\int_{\{\varnothing\}}1\,d\mu_t=\frac{p_t^\mu}{C_t}.$ Also \[C_t=
\mu_t(\{\varnothing\})+\int_{\{Z\neq\varnothing\}} h_t(\alpha_t(Z))\,d\mu_t(Z)=p_t^\mu+\int_{(0,t]} h_t(x)\,d\kappa_t^\mu(x)
=p_t^\mu+t.\]

(iii) Let $A\subset(0,t]$ be Borel. Then, using the change of variables $x=\alpha_t(Z)$,
\[
\mu_t^{\mathrm{unif}}\bigl(\alpha_t\in A,\ Z\neq\varnothing\bigr)
=\frac{1}{C_t}\int \1_A(\alpha_t(Z))\,h_t(\alpha_t(Z))\,d\mu_t^{\neq}(Z)
=\frac{1}{C_t}\int_A h_t(x)\,d\kappa_t^{\mu}(x)
=\frac{1}{C_t}\,\mathrm{Leb}(A).
\]
Also $\mu_t^{\mathrm{unif}}(Z\neq\varnothing)=\mathrm{Leb}((0,t])/C_t=t/C_t$.
Therefore
\[
\mu_t^{\mathrm{unif}}(\alpha_t\in A\mid Z\neq\varnothing)
=
\frac{\mathrm{Leb}(A)/C_t}{t/C_t}
=
\frac{\mathrm{Leb}(A)}{t},
\]
i.e.\ $\alpha_t\sim\mathrm{Unif}(0,t)$ conditional on $Z\neq\varnothing$.
\end{proof}
The following elementary lemma gives a universal tail bound under exponential tilting.
\begin{lemma}
\label{lem:local-tail}
Let $\mu$ be a probability measure on a measurable space and let $Y\ge0$ be measurable.
Fix $L>0$ and define the tilted probability $\mu^{(L)}$ by
\[
d\mu^{(L)}:=\frac{e^{-LY}}{\int e^{-LY}\,d\mu}\,d\mu.
\]
Then for every $r>0$, we have $\mu^{(L)}(Y\ge r)\le e^{-Lr}.$
\end{lemma}

\begin{proof}
On $\{Y\ge r\}$ one has $e^{-LY}\le e^{-Lr}$, hence
\[
\mu^{(L)}(Y\ge r)=\frac{\int \1_{\{Y\ge r\}}e^{-LY}\,d\mu}{\int e^{-LY}\,d\mu}
\le \frac{e^{-Lr}\int e^{-LY}\,d\mu}{\int e^{-LY}\,d\mu}=e^{-Lr}.
\]
\end{proof}

\begin{definition}[Anchor-adapted logarithmic localization]\label{def:log-local-anchor}
Fix a Borel measurable nonincreasing function $L:(0,1)\to(0,\infty)$
such that $L(r)\to\infty$ as $r\downarrow0$ (e.g.\ $L(r)=|\ln r|$). Let $\boldsymbol\mu=\{\mu_t\}_{t>0}$ be a family of probability measures $\mu _t$ on
$\mathscr C_t^{\{\ast\}}$ satisfying 
$\mu_t(Z\neq\varnothing,\ \alpha_t(Z)=0)=0$ for all $t\in(0,1)$.

Define the anchor-adapted localized tilt $\boldsymbol\mu^{\mathrm{loc},\alpha}$ by, for $t\in(0,1)$,
\begin{equation}\label{eq:loc-tilt-anchor}
d\mu_t^{\mathrm{loc},\alpha}(Z)
:=
\frac{1}{Z_t^{\mathrm{loc},\alpha}}
\Bigl(
\1_{\{\varnothing\}}(Z)
+
\1_{\{Z\neq\varnothing, \ \alpha_t(Z)>0\}}(Z)\,
\exp\!\Bigl(-L(\alpha_t(Z))\,\frac{\mathrm{diam}_t(Z)}{\alpha_t(Z)^2}\Bigr)+
\1_{\{Z\neq\varnothing,\ \alpha_t(Z)=0\}}(Z)
\Bigr)\,d\mu_t(Z),
\end{equation}
where $Z_t^{\mathrm{loc},\alpha}$ is the normalizing constant.
For $t\ge1$ set $\mu_t^{\mathrm{loc},\alpha}:=\mu_t$.
\end{definition}

\begin{corollary}
\label{cor:diam-control}
For $t\in(0,1)$, under $\mu_t^{\mathrm{loc},\alpha}(\cdot\,|\,Z\neq\varnothing)$ one has
\[
\mu_t^{\mathrm{loc},\alpha}\bigl(\mathrm{diam}_t(Z)\ge t^2 \,\big|\,Z\neq\varnothing\bigr)
\ \le\ e^{-L(t)}.
\]
In particular, if $L(t)=|\ln t|$, then the RHS is $t$.
\end{corollary}

\begin{proof}
Fix $t\in(0,1)$ and set $E_t:=\bigl\{Z\neq\varnothing:\ \diam_t(Z)\ge t^2\bigr\}.$ Let $\PP$ denote the probability measure $\mu_t^{\mathrm{loc},\alpha}(\,\cdot\,\mid Z\neq\varnothing)$.
By Lemma~\ref{lem:anchor-borel} the anchor map $\alpha_t$ is Borel, hence (since we are on a
standard Borel space) there exists a regular conditional probability kernel $x\ \longmapsto\ \PP_x(\,\cdot\,)$, $(x\in(0,t])$
such that $\PP_x=\PP(\,\cdot\,\mid \alpha_t=x)$ and $\PP=\int \PP_x\,\eta(dx)$, where
$\eta:=\PP\circ \alpha_t^{-1}$ is the $\alpha_t$--marginal of $\PP$.

Now fix $x\in(0,t]$. By definition of the anchor-adapted localization,
on the fiber $\{\alpha_t=x\}$ the Radon--Nikodym density of $\PP_x$ with respect to the
corresponding conditional base law $\mu_t(\,\cdot\,\mid Z\neq\varnothing,\ \alpha_t=x)$ is
proportional to
\[
Z\ \longmapsto\ \exp\!\Bigl(-L(x)\,\frac{\diam_t(Z)}{x^2}\Bigr).
\]
Consequently, we may apply Lemma~\ref{lem:local-tail} on this fiber (with scale $x^2$ and
threshold $t^2$) to obtain, for $\eta$-a.e.\ $x\in(0,t]$,
\begin{equation}\label{eq:cor-diam-control-fiber}
\PP_x(E_t)\ \le\ \exp\!\Bigl(-L(x)\,\frac{t^2}{x^2}\Bigr).
\end{equation}
Finally, since $x\le t$ and $L$ is nonincreasing, we have $L(x)\ge L(t)$ and
$\frac{t^2}{x^2}\ge 1$, hence $L(x)\,\frac{t^2}{x^2}\ \ge\ L(t),$
so $\exp\!\Bigl(-L(x)\,\frac{t^2}{x^2}\Bigr)\ \le\ e^{-L(t)}.$
Integrating \eqref{eq:cor-diam-control-fiber} against $\eta$ yields
\[
\mu_t^{\mathrm{loc},\alpha}\bigl(\diam_t(Z)\ge t^2\mid Z\neq\varnothing\bigr)
=\PP(E_t)
=\int_{(0,t]} \PP_x(E_t)\,\eta(dx)
\le e^{-L(t)},
\]
which is the desired bound.
\end{proof}
The next lemma shows that anchor-adapted localization yields strong diameter control at scale $t^2$.

\begin{lemma}
\label{lem:diam-control-unif-anchor}
Fix $t\in(0,1)$ and let $\mu_t$ be a probability measure on $\mathscr C_t^{\{\ast\}}$.
Let $\mu_t^{\mathrm{loc},\alpha}$ be the anchor-adapted localized tilt from
Definition~\ref{def:log-local-anchor}, and assume that
$\kappa_t^{\mathrm{loc},\alpha}:=(\alpha_t)_*\mu_t^{\mathrm{loc},\alpha}(\cdot\cap\{Z\neq\varnothing\})
\sim \mathrm{Leb}\!\restriction_{(0,t]}$.
Let $\mu_t^{\mathrm{unif},\alpha}:=(\mu_t^{\mathrm{loc},\alpha})^{\mathrm{unif}}$
be its Palm-uniformization (Lemma~\ref{lem:palm-uniformization}).

Then under $\mu_t^{\mathrm{unif},\alpha}(\cdot\mid Z\neq\varnothing)$ the anchor is $\mathrm{Unif}(0,t)$ and
\[
\mu_t^{\mathrm{unif},\alpha}\bigl(\mathrm{diam}_t(Z)\ge t^2\mid Z\neq\varnothing\bigr)
\ \le\ e^{-L(t)}.
\]
In particular, if $L(r)=|\ln r|$, then the right-hand side equals $t$.
\end{lemma}

\begin{proof}
By Lemma~\ref{lem:palm-uniformization} applied to $\mu_t^{\mathrm{loc},\alpha}$,
under $\mu_t^{\mathrm{unif},\alpha}(\cdot\mid Z\neq\varnothing)$ we have $\alpha_t\sim \mathrm{Unif}(0,t)$,
and Palm-uniformization does not change the conditional law on each anchor-fiber.

Fix $x\in(0,t]$ and consider the conditional law given $\alpha_t=x$.
On that fiber, the density in \eqref{eq:loc-tilt-anchor} is proportional to
$\exp\!\bigl(-L(x)\,\mathrm{diam}_t(Z)/x^2\bigr)$, i.e.\ it is an exponential tilt of
$Y:=\mathrm{diam}_t(Z)/x^2$ with parameter $L(x)$.
Lemma~\ref{lem:local-tail} therefore gives, for every $r\ge0$,
\[
\mathbb P\Bigl(\mathrm{diam}_t(Z)\ge r x^2\ \Big|\ \alpha_t=x,\ Z\neq\varnothing\Bigr)\le e^{-L(x)\,r}.
\]
Apply this with $r=t^2/x^2\ge1$ (since $x\le t$) to get
\[
\mathbb P\Bigl(\mathrm{diam}_t(Z)\ge t^2\ \Big|\ \alpha_t=x,\ Z\neq\varnothing\Bigr)
\le e^{-L(x)\,t^2/x^2}\le e^{-L(t)},
\]
because $L$ is nonincreasing and $x\le t$ implies $L(x)\ge L(t)$ and $t^2/x^2\ge1$.
Now integrate over $x\sim \mathrm{Unif}(0,t)$ to obtain the claimed bound.
\end{proof}

\begin{remark}[Anchor-adapted localization preserves Palm-uniformizability]
Fix $t\in(0,1)$ and let $\mu_t^{\mathrm{loc},\alpha}$ be the anchor-adapted localized tilt from
Definition~\ref{def:log-local-anchor}. The weight

\[w_t(Z):=
\1_{\{\varnothing\}}(Z)
+
\1_{\{Z\neq\varnothing,\ \alpha_t(Z)>0\}}(Z)\,
\exp\!\Bigl(-L(\alpha_t(Z))\,\frac{\mathrm{diam}_t(Z)}{\alpha_t(Z)^2}\Bigr)
+
\1_{\{Z\neq\varnothing,\ \alpha_t(Z)=0\}}(Z)\]
is strictly positive for all $Z\in\mathscr C_t^{\{\ast\}}$, hence $\mu_t^{\mathrm{loc},\alpha}\sim \mu_t$.
Restricting to $\{Z\neq\varnothing\}$ gives
$\mu_t^{\mathrm{loc},\alpha}(\cdot\cap\{Z\neq\varnothing\})\sim \mu_t(\cdot\cap\{Z\neq\varnothing\})$,
and since $\alpha_t$ is Borel,
\[
\kappa_t^{\mathrm{loc},\alpha}
:=
(\alpha_t)_*\bigl(\mu_t^{\mathrm{loc},\alpha}(\cdot\cap\{Z\neq\varnothing\})\bigr)
\ \sim\
(\alpha_t)_*\bigl(\mu_t(\cdot\cap\{Z\neq\varnothing\})\bigr)
=
\kappa_t^{\mu}.
\]
In particular, if $\kappa_t^\mu\sim \mathrm{Leb}\!\restriction_{(0,t]}$, then also
$\kappa_t^{\mathrm{loc},\alpha}\sim \mathrm{Leb}\!\restriction_{(0,t]}$.
\end{remark}

The next theorem is structural. From a Palm-uniformizable factorizing family $\boldsymbol\mu$, we obtain a new family by anchor-adapted localization,
exact Palm uniformization of the anchor, and normalization of the vacuum mass to $\nu_t({\varnothing}) = e^{-\beta t}$. The measure type is preserved.

\begin{theorem}
\label{prop:palm-loc-vacuum}
Let $\boldsymbol\mu=\{\mu_t\}_{t>0}$ be a Palm-uniformizable measurable factorizing family over $[0,1]\times \{\ast\}$ such that $\mu_t(\{\varnothing\})>0$ for every $t>0$.
Fix $\beta>0$ and $L(t)=|\ln t|$ for $t\in(0,1)$. Define $\mu_t^{\mathrm{loc},\alpha}$ by \eqref{eq:loc-tilt-anchor}, and define $\mu_t^{\mathrm{unif},\alpha} := (\mu_t^{\mathrm{loc},\alpha})^{\mathrm{unif}}$ to be the Palm
uniformization of $\mu_t^{\mathrm{loc}}$. 
Finally, define 
\begin{equation}\label{eq:nu-def-final}
\nu_t
:=
e^{-\beta t}\,\delta_{\varnothing}
+(1-e^{-\beta t})\,\mu_t^{\mathrm{unif},\alpha}(\,\cdot\,|\,Z\neq\varnothing).
\end{equation}
Then:
\begin{enumerate}[label=\textup{(\roman*)}]
\item $\boldsymbol\nu=\{\nu_t\}_{t>0}$ is a measurable factorizing family and $\nu_t\sim \mu_t$ for every $t$.
\item The associated random-set system $\mathbb E^{\boldsymbol\nu}$ is isomorphic to $\mathbb E^{\boldsymbol\mu}$.
\item The vacuum unit $v_t=\nu_t(\{\varnothing\})^{-1/2}\1_{\{\varnothing\}}$ satisfies
$m(t):=\langle \1,v_t\rangle=\sqrt{\nu_t(\{\varnothing\})}=e^{-\beta t/2}$, hence the linear overlap
\[
1-m(t)\ge \frac{\beta}{4}\,t\qquad (0<t\le 2/\beta).
\]
\end{enumerate}
\end{theorem}
\begin{proof}
We prove \textup{(i)}--\textup{(iii)} in order.

\smallskip\noindent
Proof of \textup{(i)}. For each $t>0$, the measure $\mu_t^{\mathrm{loc},\alpha}$ is obtained from $\mu_t$ by multiplying by the strictly positive weight appearing in Definition~\ref{def:log-local-anchor}. Hence $\mu_t^{\mathrm{loc},\alpha}\sim \mu_t.$
By Lemma~\ref{lem:palm-uniformization}, the Palm-uniformized measure
$\mu_t^{\mathrm{unif},\alpha}=(\mu_t^{\mathrm{loc},\alpha})^{\mathrm{unif}}$ satisfies $\mu_t^{\mathrm{unif},\alpha}\sim \mu_t^{\mathrm{loc},\alpha},$
and therefore $
\mu_t^{\mathrm{unif},\alpha}\sim \mu_t$ for all 
$t>0$.

We claim that $\nu_t\sim \mu_t$ for every $t>0$.
Indeed, if $\mu_t(A)=0$, then necessarily $\varnothing\notin A$ because
$\mu_t(\{\varnothing\})>0$ by hypothesis. Thus $A\subset\{Z\neq\varnothing\}$, and since
$\mu_t^{\mathrm{unif},\alpha}\sim\mu_t$, we have
$\mu_t^{\mathrm{unif},\alpha}(A)=0$. Hence
\[
\nu_t(A)
=
e^{-\beta t}\delta_{\varnothing}(A)
+(1-e^{-\beta t})\,\mu_t^{\mathrm{unif},\alpha}(A\mid Z\neq\varnothing)
=0.
\]
Conversely, if $\nu_t(A)=0$, then $\varnothing\notin A$ because
$\nu_t(\{\varnothing\})=e^{-\beta t}>0$. Hence again
$A\subset\{Z\neq\varnothing\}$, and
\[
0=\nu_t(A)=(1-e^{-\beta t})\,\mu_t^{\mathrm{unif},\alpha}(A\mid Z\neq\varnothing)
\]
implies $\mu_t^{\mathrm{unif},\alpha}(A)=0$. Since
$\mu_t^{\mathrm{unif},\alpha}\sim\mu_t$, we obtain $\mu_t(A)=0$.
Thus $\nu_t\sim\mu_t$ for every $t>0$.

We next verify the factorization property.
Since $\nu_s\sim\mu_s$ and $\nu_t\sim\mu_t$, we have
$\nu_s\otimes\nu_t\sim \mu_s\otimes\mu_t$, and therefore $(\nu_s\otimes\nu_t)\circ\oplus_{s,t}^{-1}
\sim
(\mu_s\otimes\mu_t)\circ\oplus_{s,t}^{-1}
\sim
\mu_{s+t}
\sim
\nu_{s+t}.$
Hence $\boldsymbol\nu=\{\nu_t\}_{t>0}$ satisfies
Definition~\ref{def:fact-meas}\textup{(i)}.

Condition~\textup{(ii)} of Definition~\ref{def:fact-meas} (absence of deterministic time-slices) is inherited from
$\boldsymbol\mu$ because $\nu_t\sim\mu_t$ for every $t$:
if $r\in[0,t]$, then
$\mu_t\bigl(\{Z: Z\cap(\{r\}\times\{\ast\})\neq\varnothing\}\bigr)=0$ implies that
$\nu_t\bigl(\{Z: Z\cap(\{r\}\times\{\ast\})\neq\varnothing\}\bigr)=0.$

Condition~\textup{(iii)} of Definition~\ref{def:fact-meas} (non-degeneracy) is also inherited from equivalence:
since each $\mu_t$ is not supported on finitely many atoms and $\nu_t\sim\mu_t$,
the same holds for $\nu_t$.

It remains to prove the measurability conditions \textup{(iva)} and \textup{(ivb)} of Definition~\ref{def:fact-meas}. First, we prove \textup{(iva)}.
For simplicity, write $\mathscr C:=\mathscr C_1^{\{\ast\}},$
$\Sigma:=\Sigma_1^{\{\ast\}},$ and let $\sigma_t:\mathscr C_t^{\{\ast\}}\to\mathscr C$ be the time-scaling homeomorphism.
Set
\[
\widetilde\mu_t:=(\sigma_t)_*\mu_t,\qquad
\widetilde\mu_t^{\mathrm{loc}}:=(\sigma_t)_*\mu_t^{\mathrm{loc},\alpha},\qquad
\widetilde\mu_t^{\mathrm{unif}}:=(\sigma_t)_*\mu_t^{\mathrm{unif},\alpha},\qquad
\widetilde\nu_t:=(\sigma_t)_*\nu_t.
\]

Enlarging the countable ring $\mathcal R\subset\Sigma$ from
Definition~\ref{def:fact-meas}\textup{(iva)} if necessary (still countable), assume that
$\mathcal R$ is a countable algebra generating $\Sigma$ and that
$\{\varnothing\}\in\mathcal R$.
By hypothesis, $t\mapsto \widetilde\mu_t(A)$ is Borel on $(0,\infty)$, for each $A\in\mathcal R$.

Let $\alpha_1,\diam_1:\mathscr C\to[0,1]$ be the anchor and spread maps from
Definition~\ref{def:anchor-spread}. For $(t,W)\in(0,\infty)\times\mathscr C$, define $w^{\mathrm{loc},\alpha}(t,W)
:=$
\[
\begin{cases}
\1_{\{\varnothing\}}(W)
+\1_{\{W\neq\varnothing,\ \alpha_1(W)>0\}}(W)\,
\exp\!\Bigl(
-\,L\bigl(t\,\alpha_1(W)\bigr)\,
\frac{\diam_1(W)}{t\,\alpha_1(W)^2}
\Bigr)
+\1_{\{W\neq\varnothing,\ \alpha_1(W)=0\}}(W),
& 0<t<1,\\[1.2ex]
1,
& t\ge 1.
\end{cases}
\]
This function is jointly Borel because $\alpha_1$ and $\diam_1$ are Borel by
Lemma~\ref{lem:anchor-borel}, $L$ is Borel on $(0,1)$, and the two cases are separated by the Borel sets
$\{t<1\}$ and $\{t\ge1\}$.

A direct calculation from \eqref{eq:loc-tilt-anchor} shows that
\[
d\widetilde\mu_t^{\mathrm{loc}}(W)
=
\frac{w^{\mathrm{loc},\alpha}(t,W)}{Z_t^{\mathrm{loc}}}\,d\widetilde\mu_t(W),
\;\text{where}\;
Z_t^{\mathrm{loc}}
:=
\int_{\mathscr C} w^{\mathrm{loc},\alpha}(t,W)\,d\widetilde\mu_t(W).
\]
Indeed, for $0<t<1$ and $W=\sigma_t(Z)$ one has $\alpha_t(Z)=t\,\alpha_1(W),$
$\diam_t(Z)=t\,\diam_1(W),$
so the exponent in \eqref{eq:loc-tilt-anchor} becomes
\[
-\,L(\alpha_t(Z))\,\frac{\diam_t(Z)}{\alpha_t(Z)^2}
=
-\,L\bigl(t\,\alpha_1(W)\bigr)\,\frac{\diam_1(W)}{t\,\alpha_1(W)^2},
\]
while for $t\ge1$ we have $\mu_t^{\mathrm{loc},\alpha}=\mu_t$ by definition.

Since $w^{\mathrm{loc},\alpha}$ is Borel and
$t\mapsto \widetilde\mu_t(A)$ is Borel for every $A\in\mathcal R$,
Lemma~\ref{lem:meas-integral} implies that
$t\mapsto Z_t^{\mathrm{loc}}$ is Borel, and also that
\[
t\longmapsto \widetilde\mu_t^{\mathrm{loc}}(A)
=
\frac{1}{Z_t^{\mathrm{loc}}}\int_{\mathscr C}\1_A(W)\,w^{\mathrm{loc},\alpha}(t,W)\,d\widetilde\mu_t(W)
\]
is Borel for every $A\in\mathcal R$.

Next, define the finite measure on $(0,1]$ by
\[
\widetilde\kappa_t^{\mathrm{loc}}
:=
(\alpha_1)_*\Bigl(\widetilde\mu_t^{\mathrm{loc}}(\,\cdot\,\cap\{W\neq\varnothing,\ \alpha_1(W)>0\})\Bigr).
\]
Because $\mu_t^{\mathrm{loc},\alpha}\sim\mu_t$ and $\mu_t$ satisfies
Definition~\ref{def:fact-meas}\textup{(ii)}, we have $\widetilde\mu_t^{\mathrm{loc}}(\{W\neq\varnothing,\ \alpha_1(W)=0\})=0,$
so this is exactly the anchor pushforward of the nonempty part.

We claim that $\widetilde\kappa_t^{\mathrm{loc}}\sim \lambda_t:=t\,\mathrm{Leb}\!\restriction_{(0,1]}$ for all $t>0$. Indeed, on the unscaled space we have
\[
(\alpha_t)_*\bigl(\mu_t^{\mathrm{loc},\alpha}(\cdot\cap\{Z\neq\varnothing\})\bigr)
\sim
(\alpha_t)_*\bigl(\mu_t(\cdot\cap\{Z\neq\varnothing\})\bigr)
=
\kappa_t^\mu
\sim
\mathrm{Leb}\!\restriction_{(0,t]},
\]
because $\mu_t^{\mathrm{loc},\alpha}\sim\mu_t$ and $\alpha_t$ is Borel.
After scaling by $x\mapsto x/t$, this becomes
$\widetilde\kappa_t^{\mathrm{loc}}\sim \lambda_t$ on $(0,1]$.

Fix a countable algebra $\mathcal A\subset\operatorname{Bor}((0,1])$ generating the Borel
$\sigma$--field on $(0,1]$, e.g.\ finite unions of rational subintervals. For $B\in\mathcal A$,
\[
\widetilde\kappa_t^{\mathrm{loc}}(B)
=
\int_{\mathscr C}
\1_{\{W\neq\varnothing,\ \alpha_1(W)>0\}}(W)\,\1_B(\alpha_1(W))
\,d\widetilde\mu_t^{\mathrm{loc}}(W),
\]
and the integrand is Borel on $(0,\infty)\times\mathscr C$.
Applying Lemma~\ref{lem:meas-integral} to the family
$\{\widetilde\mu_t^{\mathrm{loc}}\}_{t>0}$ shows that
$t\mapsto \widetilde\kappa_t^{\mathrm{loc}}(B)$ is Borel for every $B\in\mathcal A$.
Also, $t\mapsto \lambda_t(B)=t\,\mathrm{Leb}(B)$
is Borel. Therefore Lemma~\ref{lem:measurable-RN}, applied on the base space $(0,1]$ with parameter space
$(0,\infty)$, yields a Borel function $\widetilde h:(0,\infty)\times(0,1]\to(0,\infty)$
such that for every $t>0$,
\[
\widetilde h(t,\cdot)=\frac{d\lambda_t}{d\widetilde\kappa_t^{\mathrm{loc}}}
\qquad \widetilde\kappa_t^{\mathrm{loc}}\text{-a.e.}
\]

Now define the Borel weight
\[
w^{\mathrm{unif}}(t,W)
:=
\1_{\{\varnothing\}}(W)
+\1_{\{W\neq\varnothing,\ \alpha_1(W)>0\}}(W)\,\widetilde h(t,\alpha_1(W))
+\1_{\{W\neq\varnothing,\ \alpha_1(W)=0\}}(W).
\]
Let
\[
p_t^{\mathrm{loc}}:=\widetilde\mu_t^{\mathrm{loc}}(\{\varnothing\}),
\qquad
C_t^{\mathrm{unif}}
:=
\int_{\mathscr C} w^{\mathrm{unif}}(t,W)\,d\widetilde\mu_t^{\mathrm{loc}}(W).
\]
Since $\widetilde\mu_t^{\mathrm{loc}}(\{W\neq\varnothing,\alpha_1(W)=0\})=0$, we obtain
\begin{align*}
C_t^{\mathrm{unif}}
&=
p_t^{\mathrm{loc}}
+\int_{\mathscr C}
\1_{\{W\neq\varnothing,\alpha_1(W)>0\}}(W)\,\widetilde h(t,\alpha_1(W))
\,d\widetilde\mu_t^{\mathrm{loc}}(W)\\
&=
p_t^{\mathrm{loc}}
+\int_{(0,1]} \widetilde h(t,x)\,d\widetilde\kappa_t^{\mathrm{loc}}(x)\\
&=
p_t^{\mathrm{loc}}+\lambda_t((0,1])\\
&=
p_t^{\mathrm{loc}}+t.
\end{align*}
In particular, $t\mapsto C_t^{\mathrm{unif}}$ is Borel. Again by the Palm-uniformization formula, now expressed on the scaled space, we have
\[
d\widetilde\mu_t^{\mathrm{unif}}(W)
=
\frac{w^{\mathrm{unif}}(t,W)}{C_t^{\mathrm{unif}}}\,d\widetilde\mu_t^{\mathrm{loc}}(W).
\]
Hence, for every $A\in\mathcal R$,
\[
\widetilde\mu_t^{\mathrm{unif}}(A)
=
\frac{1}{C_t^{\mathrm{unif}}}
\int_{\mathscr C}\1_A(W)\,w^{\mathrm{unif}}(t,W)\,d\widetilde\mu_t^{\mathrm{loc}}(W),
\]
and Lemma~\ref{lem:meas-integral} shows that
$t\mapsto \widetilde\mu_t^{\mathrm{unif}}(A)$ is Borel. Finally, by \eqref{eq:nu-def-final},
\[
\widetilde\nu_t(A)
=
e^{-\beta t}\,\1_A(\varnothing)
+
(1-e^{-\beta t})\,
\frac{\widetilde\mu_t^{\mathrm{unif}}(A\cap\{W\neq\varnothing\})}
{\widetilde\mu_t^{\mathrm{unif}}(\{W\neq\varnothing\})}.
\]
Since
\[
\widetilde\mu_t^{\mathrm{unif}}(\{W\neq\varnothing\})
=
\frac{t}{C_t^{\mathrm{unif}}}>0,
\]
and both numerator and denominator are Borel in $t$, it follows that
$t\mapsto \widetilde\nu_t(A)$ is Borel for every $A\in\mathcal R$.
This proves Definition~\ref{def:fact-meas}\textup{(iva)} for $\boldsymbol\nu$.

Next we show that $\boldsymbol\nu$ satisfies condition \textup{(ivb)} of Definition~\ref{def:fact-meas}. For this, let $\Delta^{(\mu)}:(0,\infty)^2\times\mathscr C\to(0,\infty)$ be the jointly Borel cocycle from
Definition~\ref{def:fact-meas}\textup{(ivb)} for $\boldsymbol\mu$.
Since $\widetilde\nu_t\sim\widetilde\mu_t$ for every $t>0$, and both families satisfy
\textup{(iva)} on the same countable algebra $\mathcal R$, Lemma~\ref{lem:measurable-RN} yields a Borel function $f:(0,\infty)\times\mathscr C\to[0,\infty)$
such that
\[
f(t,\cdot)=\frac{d\widetilde\nu_t}{d\widetilde\mu_t}
\qquad \widetilde\mu_t\text{-a.e.}
\]
Because $\widetilde\nu_t\sim\widetilde\mu_t$, this derivative is strictly positive
$\widetilde\mu_t$-a.e. for every $t$, so after modifying $f$ on a Borel null set if necessary we may and do assume that
$f:(0,\infty)\times\mathscr C\to(0,\infty)$
is Borel and still represents the same Radon--Nikodym derivatives.

For $s,t>0$ and $u=\frac{s}{s+t}$, let $\rho_u,\tau_u:\mathscr C\to\mathscr C$ be the restriction--rescale
maps from the proof of Theorem~\ref{th:random-system}:
\[
\rho_u(W):=\sigma_u\bigl(W\cap[0,u]\bigr),
\qquad
\tau_u(W):=\sigma_{1-u}\bigl((W\cap[u,1])-u\bigr).
\]
These are Borel. Define
\begin{equation}\label{eq:Delta-nu-def-proof-rewritten}
\Delta^{(\nu)}(s,t,W)
:=
\Delta^{(\mu)}(s,t,W)\,
\frac{ f\bigl(s,\rho_{s/(s+t)}(W)\bigr)\, f\bigl(t,\tau_{s/(s+t)}(W)\bigr)}{f(s+t,W)}.
\end{equation}
This is a jointly Borel map $(0,\infty)^2\times\mathscr C\to(0,\infty)$.

Let
\[
\widetilde\mu_{s,t}:=(\sigma_{s+t})_*\Bigl((\mu_s\otimes\mu_t)\circ\oplus_{s,t}^{-1}\Bigr),
\qquad
\widetilde\nu_{s,t}:=(\sigma_{s+t})_*\Bigl((\nu_s\otimes\nu_t)\circ\oplus_{s,t}^{-1}\Bigr).
\]
By Definition~\ref{def:fact-meas}\textup{(ivb)} for $\boldsymbol\mu$,
\[
\frac{d\widetilde\mu_{s,t}}{d\widetilde\mu_{s+t}}(W)
=
\Delta^{(\mu)}(s,t,W)
\qquad \widetilde\mu_{s+t}\text{-a.e.}
\]
Since $d\widetilde\nu_r=f(r,\cdot)\,d\widetilde\mu_r$, a standard change-of-variables computation gives
\[
\frac{d\widetilde\nu_{s,t}}{d\widetilde\mu_{s,t}}(W)
=
f\bigl(s,\rho_{s/(s+t)}(W)\bigr)\,
f\bigl(t,\tau_{s/(s+t)}(W)\bigr)
\qquad \widetilde\mu_{s,t}\text{-a.e.}
\]
Combining this with the chain rule and
\[
\frac{d\widetilde\mu_{s+t}}{d\widetilde\nu_{s+t}}(W)=\frac{1}{f(s+t,W)},
\]
we obtain
\[
\frac{d\widetilde\nu_{s,t}}{d\widetilde\nu_{s+t}}(W)
=
\Delta^{(\mu)}(s,t,W)\,
\frac{ f\bigl(s,\rho_{s/(s+t)}(W)\bigr)\, f\bigl(t,\tau_{s/(s+t)}(W)\bigr)}{f(s+t,W)}
=
\Delta^{(\nu)}(s,t,W)
\]
for $\widetilde\nu_{s+t}$-a.e.\ $W$.
Thus Definition~\ref{def:fact-meas}\textup{(ivb)} holds for $\boldsymbol\nu$. This completes the proof of \textup{(i)}.

\smallskip\noindent
{Proof of \textup{(ii)}.}
For each $t>0$, since $\nu_t\sim\mu_t$, the change-of-measure map
\[
T_t:L^2(\mathscr C_t,\mu_t)\to L^2(\mathscr C_t,\nu_t),
\qquad
(T_t f)(Z)=\Bigl(\frac{d\mu_t}{d\nu_t}(Z)\Bigr)^{1/2}f(Z),
\]
is unitary.
A direct Radon--Nikodym chain-rule computation shows that the family
$\{T_t\}_{t>0}$ intertwines the multiplication maps of
$\mathbb E^{\boldsymbol\mu}$ and $\mathbb E^{\boldsymbol\nu}$.
Hence $\mathbb E^{\boldsymbol\nu}$ and $\mathbb E^{\boldsymbol\mu}$ are isomorphic as product systems.

\smallskip\noindent
{Proof of \textup{(iii)}.}
By construction, $\nu_t(\{\varnothing\})=e^{-\beta t},$ so the vacuum unit $v_t=\nu_t(\{\varnothing\})^{-1/2}\1_{\{\varnothing\}}$
satisfies $m(t):=\langle \1,v_t\rangle = \sqrt{\nu_t(\{\varnothing\})}=e^{-\beta t/2}.$ If $0<t\le 2/\beta$, then with $x=\beta t/2\in(0,1]$ we have
$1-e^{-x}\ge x/2$, and therefore
\[
1-m(t)=1-e^{-\beta t/2}\ge \frac{\beta}{4}\,t.
\]
This proves \textup{(iii)} and completes the proof of the theorem.
\end{proof}

\subsection{A sufficient criterion for Hellinger-smallness}

In this subsection we derive a convenient sufficient criterion for Hellinger-smallness. The argument isolates the contributions of the four block-occupancy events and reduces the short-time Hellinger estimate to manageable one-block comparison bounds.

We begin by recalling the Hellinger affinity, the associated Hellinger distance, and the total variation distance for finite measures.

\begin{definition}\label{def:hellinger-finite}
Let $(X,\Sigma)$ be a measurable space and let $\rho,\eta$ be finite measures on $(X,\Sigma)$.

\textup{(i)} Let $m$ be any finite measure such that $\rho\ll m$ and $\eta\ll m$.
The Hellinger affinity is defined by
\[
H(\rho,\eta):=\int_X 
\sqrt{\frac{d\rho}{dm}\,\frac{d\eta}{dm}}\,dm
\in[0,\infty).
\]
This quantity does not depend on the choice of dominating measure $m$. The Hellinger distance is defined by
\[
d_{\mathrm{Hell}}^2(\rho,\eta)
:=\rho(X)+\eta(X)-2H(\rho,\eta)
\in[0,\infty).
\]

If $\rho$ and $\eta$ are probability measures, then $H(\rho,\eta)\in[0,1]$ and $d_{\mathrm{Hell}}^2(\rho,\eta)=2\bigl(1-H(\rho,\eta)\bigr).$
Moreover, if $\rho\sim\eta$ then $H(\rho,\eta)>0$.

\medskip

\textup{(ii)} The total variation distance between $\rho$ and $\eta$ is defined by $\|\rho-\eta\|_{\mathrm{TV}}
:=\sup_{A\in\Sigma}|\rho(A)-\eta(A)|.$
If $\rho$ and $\eta$ are probability measures and $m$ is any dominating
measure, then equivalently
\begin{equation}\label{eq:st-ident}
\|\rho-\eta\|_{\mathrm{TV}}
=\frac12\int_X
\left|
\frac{d\rho}{dm}-\frac{d\eta}{dm}
\right|\,dm .
\end{equation}
In particular, for probability measures one has the elementary inequality $1-H(\rho,\eta)\le \|\rho-\eta\|_{\mathrm{TV}} .$
\end{definition}

We record additivity over countable partitions and the explicit formula for scaling by weights. If $\{E_k\}_{k\in\N}$ a a measurable partition of $X$, 
then
\begin{equation}\label{eq:partition}
H(\rho,\eta)=\sum_{k\in\N} H(\rho\!\restriction_{E_k},\eta\!\restriction_{E_k}),
\qquad
d_{\mathrm{Hell}}^2(\rho,\eta)=\sum_{k\in\N} d_{\mathrm{Hell}}^2(\rho\!\restriction_{E_k},\eta\!\restriction_{E_k}).
\end{equation}
Also \begin{equation}\label{eq:hell-estimate}d_{\mathrm{Hell}}^2(p\,\mu,q\,\nu)
=
(\sqrt p-\sqrt q)^2
+
2\sqrt{pq}\,\bigl(1-H(\mu,\nu)\bigr),\end{equation}
for any two probability measures $\mu$, $\nu$ and for all $p,\,q\ge 0$.

The following proposition gives a convenient sufficient criterion for Hellinger-smallness by decomposing the short-time law according to the four possible block-occupancy patterns relative to the cut point.

\begin{proposition}
\label{lem:hellinger-small-criterion-uncond}
Let $\boldsymbol\nu=\{\nu_t\}_{t>0}$ be a measurable factorizing family over $[0,1]\times\{\ast\}$.
Fix $s,t>0$ and put $u=\lambda(s+t)$, $u_1=\lambda s$, $u_2=\lambda t$.
Let $\sigma_{u_1,u_2}:=(\nu_{u_1}\otimes\nu_{u_2})\circ\oplus^{-1}$ on $\mathscr C_u^{\{\ast\}}$. Consider the the block-occupancy events 
\begin{eqnarray*}
E_{00}&:=&\{Z\cap[0,u_1]=\varnothing,\ Z\cap[u_1,u]=\varnothing\}=\{\varnothing\},\\
E_{10}&:=&\{Z\cap[0,u_1]\neq\varnothing,\ Z\cap[u_1,u]=\varnothing\},\\
E_{01}&:=&\{Z\cap[0,u_1]=\varnothing,\ Z\cap[u_1,u]\neq\varnothing\},\\
E_{11}&:=&\{Z\cap[0,u_1]\neq\varnothing,\ Z\cap[u_1,u]\neq\varnothing\}.
\end{eqnarray*}
Assume that as $\lambda\downarrow0$ the following estimates hold:
\begin{enumerate}[label=\textup{(\roman*)}]
\item\label{it:prob-linear2}
(linear nonemptiness)
$\nu_u(Z\neq\varnothing)=O(u)$ and $\nu_{u_i}(Z\neq\varnothing)=O(u_i)$.
\item\label{it:two-block2}
(quadratic two-block overlap)
$\nu_u(E_{11})=O(u^2)$.
\item\label{it:vac-mix2}
(vacuum weight consistency to second order)
\[
\bigl|\nu_u(\{\varnothing\})-\nu_{u_1}(\{\varnothing\})\,\nu_{u_2}(\{\varnothing\})\bigr|
=O(u^2).
\]
\item\label{it:oneblock-hell-uncond}
(\emph{unconditional} one-block Hellinger control)
Let $\nu_u^{10}:=\nu_u\!\restriction_{E_{10}}$ and $\nu_u^{01}:=\nu_u\!\restriction_{E_{01}}$.
Then
\[
d_{\mathrm{Hell}}^2\!\Bigl(\nu_u^{10},\ \sigma_{u_1,u_2}\!\restriction_{E_{10}}\Bigr)=O(u^2),
\qquad
d_{\mathrm{Hell}}^2\!\Bigl(\nu_u^{01},\ \sigma_{u_1,u_2}\!\restriction_{E_{01}}\Bigr)=O(u^2).
\]
\end{enumerate}

Then $1-H_{u_1,u_2}^{(\nu)}=O(u^2)$.
\end{proposition}

\begin{proof}
Let $\sigma:=\sigma_{u_1,u_2}$ for brevity. Since $\boldsymbol\nu$ is factorizing, $\nu_u\sim\sigma$ and
$H(\nu_u,\sigma)$ equals $H_{u_1,u_2}^{(\nu)}$.
Using the partition $\mathscr C_u^{\{\ast\}}=E_{00}\sqcup E_{10}\sqcup E_{01}\sqcup E_{11}$ and
\eqref{eq:partition}, we have
\[
d_{\mathrm{Hell}}^2(\nu_u,\sigma)
=
\sum_{ij\in\{00,10,01,11\}}
d_{\mathrm{Hell}}^2\bigl(\nu_u\!\restriction_{E_{ij}},\ \sigma\!\restriction_{E_{ij}}\bigr).
\]

We bound each term.

\smallskip\noindent
\textit{(a) The $E_{10}$ and $E_{01}$ terms.}
These are $O(u^2)$ by assumption~\ref{it:oneblock-hell-uncond}.

\smallskip\noindent
\textit{(b) The two-block term $E_{11}$.}
Always have $d_{\mathrm{Hell}}^2(\rho,\eta)\le \rho(X)+\eta(X)$ for finite measures, hence
\[
d_{\mathrm{Hell}}^2(\nu_u\!\restriction_{E_{11}},\sigma\!\restriction_{E_{11}})
\le \nu_u(E_{11})+\sigma(E_{11}).
\]
By assumption~\ref{it:two-block2}, $\nu_u(E_{11})=O(u^2)$. Under $\sigma$ the left/right components are independent; moreover, by the no-deterministic-points
assumption, the cutpoint event $\{u_1\in Z\}$ is $\sigma$-null, so $E_{11}$ agrees $\sigma$-a.s. with
$\{Z\cap[0,u_1]\neq\varnothing\}\cap\{Z\cap[u_1,u]\neq\varnothing\}$ in the product sense.
Therefore
\[
\sigma(E_{11})
=
\nu_{u_1}(Z\neq\varnothing)\,\nu_{u_2}(Z\neq\varnothing)
=
O(u_1u_2)=O(u^2)
\]
by~\ref{it:prob-linear2}. Thus the $E_{11}$ contribution is $O(u^2)$.

\smallskip\noindent
\textit{(c) The vacuum term $E_{00}=\{\varnothing\}$.}
On $\{\varnothing\}$,
\[
d_{\mathrm{Hell}}^2(\nu_u\!\restriction_{\{\varnothing\}},\sigma\!\restriction_{\{\varnothing\}})
=
\bigl(\sqrt{\nu_u(\{\varnothing\})}-\sqrt{\sigma(\{\varnothing\})}\bigr)^2.
\]
Since $\sigma(\{\varnothing\})=\nu_{u_1}(\{\varnothing\})\nu_{u_2}(\{\varnothing\})$ and
$\nu_u(Z\neq\varnothing)=O(u)$ implies $\nu_u(\{\varnothing\})=1-O(u)$, both arguments of $\sqrt{\cdot}$
lie in $[1/2,1]$ for $\lambda$ small. Hence $\sqrt{\cdot}$ is Lipschitz there and
assumption~\ref{it:vac-mix2} yields this term is $O(u^4)$ (in particular $O(u^2)$).

\smallskip
Summing (a)--(c), we obtain $d_{\mathrm{Hell}}^2(\nu_u,\sigma)=O(u^2)$.
Since both $\nu_u$ and $\sigma$ are probability measures,
Definition~\ref{def:hellinger-finite} yields
\[
1-H(\nu_u,\sigma)=\frac12\,d_{\mathrm{Hell}}^2(\nu_u,\sigma)=O(u^2),
\]
which is exactly the small-seed bound.
\end{proof}

\medskip

For the family $\boldsymbol\nu$ built in Theorem~\ref{prop:palm-loc-vacuum}, item \ref{it:prob-linear2} of Proposition~\ref{lem:hellinger-small-criterion-uncond} holds
because $\nu_t(\varnothing)=e^{-\beta t}$.
Item \ref{it:two-block2} follows from: under $\nu_u(\cdot\,|\,Z\neq\varnothing)$ the anchor is uniform
(Lemma~\ref{lem:palm-uniformization}), while Corollary~\ref{cor:diam-control} gives
$\mathrm{diam}_u(Z)\le u^2$ with probability $1-O(u)$; hence the straddling event $E_{11}$ forces either
$\mathrm{diam}_u>u^2$ (probability $O(u)$ given nonempty) or the anchor to fall within $u^2$ of the cut
(probability $O(u)$ given nonempty). Since $\nu_u(Z\neq\varnothing)=O(u)$, this yields $\nu_u(E_{11})=O(u^2)$.
Item \ref{it:vac-mix2} is also automatically satisfied. 
The genuinely model-dependent part is \ref{it:oneblock-hell-uncond}; the next corollary gives a convenient replacement for it by showing that, in practice, it is enough to verify conditional one-block comparison estimates together with second-order matching of the corresponding block weights.
\begin{corollary}
\label{lem:hellinger-small-criterion-cond}
In the setting and notation of Proposition~\ref{lem:hellinger-small-criterion-uncond}, assume
\ref{it:prob-linear2}--\ref{it:two-block2}--\ref{it:vac-mix2}.
Assume further that as $\lambda\downarrow0$:

\begin{enumerate}[label=\textup{(\roman*)},resume]
\item[{\em(iv(a))}]\label{it:weights-oneblock}
(one-block weight matching to second order)
\[
\bigl|\nu_u(E_{10})-\sigma(E_{10})\bigr|=O(u^2),
\qquad
\bigl|\nu_u(E_{01})-\sigma(E_{01})\bigr|=O(u^2).
\]
\item[{\em(iv(b))}]\label{it:cond-oneblock-hell}
(conditional one-block Hellinger control)
Let $\widehat\nu_u^{10}$ denote the conditional law of $R_{0,u_1}(Z)$ under $\nu_u(\,\cdot\,|\,E_{10})$,
and let $\widehat\nu_u^{01}$ denote the conditional law of $R_{u_1,u}(Z)$ under $\nu_u(\,\cdot\,|\,E_{01})$.
Then
\[
1-H\!\Bigl(\widehat\nu_u^{10},\ \nu_{u_1}(\,\cdot\,|\,Z\neq\varnothing)\Bigr)=O(u),
\qquad
1-H\!\Bigl(\widehat\nu_u^{01},\ \nu_{u_2}(\,\cdot\,|\,Z\neq\varnothing)\Bigr)=O(u).
\]
\end{enumerate}

Then $1-H_{u_1,u_2}^{(\nu)}=O(u^2)$, i.e.\ $\boldsymbol\nu$ is Hellinger-small.
\end{corollary}

\begin{proof}
We verify the unconditional one-block bounds~\ref{it:oneblock-hell-uncond} of Proposition~\ref{lem:hellinger-small-criterion-uncond}. Consider $E_{10}$ and write
\[
\nu_u\!\restriction_{E_{10}} = p\,\mu,\qquad
\sigma\!\restriction_{E_{10}} = q\,\nu,
\]
where $p=\nu_u(E_{10})$, $q=\sigma(E_{10})$, and $\mu,\nu$ are the corresponding conditional probability measures on $E_{10}$.

Under $\sigma$, conditioning on $E_{10}$ means “left block nonempty, right block empty”. Up to the
cutpoint, which is $\sigma$-null by no deterministic points, the left restriction $R_{0,u_1}(Z)$ has law
$\nu_{u_1}(\,\cdot\,|\,Z\neq\varnothing)$. By definition, under $\mu$ the same restriction has law
$\widehat\nu_u^{10}$.

On $E_{10}$ the map $R_{0,u_1}$ is a measurable bijection onto
$\{W\in\mathscr C_{u_1}^{\{\ast\}}:W\neq\varnothing,\ u_1\notin W\}$, and this target has full measure
under $\nu_{u_1}(\cdot\mid Z\neq\varnothing)$ (again by no deterministic points). Since Hellinger affinity
is invariant under measurable isomorphisms, and under restriction to full-measure subsets,
assumption~(iv(b)) implies $1-H(\mu,\nu)=O(u)$.

Now apply \eqref{eq:hell-estimate}
\[
d_{\mathrm{Hell}}^2(\nu_u\!\restriction_{E_{10}},\sigma\!\restriction_{E_{10}})
=
(\sqrt p-\sqrt q)^2 + 2\sqrt{pq}\,\bigl(1-H(\mu,\nu)\bigr).
\]
By~(iv(a)), $|p-q|=O(u^2)$, and for $p,q\ge0$ one has
$(\sqrt p-\sqrt q)^2\le |p-q|$, hence $(\sqrt p-\sqrt q)^2=O(u^2)$.
Also $\sqrt{pq}=O(u)$ by~\ref{it:prob-linear2} and $1-H(\mu,\nu)=O(u)$, so the second term is $O(u^2)$.
Therefore
\[
d_{\mathrm{Hell}}^2\bigl(\nu_u\!\restriction_{E_{10}},\sigma\!\restriction_{E_{10}}\bigr)=O(u^2).
\]
The same argument on $E_{01}$ yields
$d_{\mathrm{Hell}}^2(\nu_u\!\restriction_{E_{01}},\sigma\!\restriction_{E_{01}})=O(u^2)$.
Thus~\ref{it:oneblock-hell-uncond} holds, and Proposition~\ref{lem:hellinger-small-criterion-uncond}
gives $1-H_{u_1,u_2}^{(\nu)}=O(u^2)$.
\end{proof}

The following lemma provides an abstract total-variation estimate showing that, under uniform-anchor and diameter-control assumptions, conditioning on a one-block event perturbs the corresponding restriction law by only $O(u)$.
\begin{lemma}
\label{lem:bm-one-block-stability} Let $\boldsymbol\nu=\{\nu_t\}_{t>0}$ be a measurable factorizing family over $[0,1]\times\{\ast\}$.
Fix $s,t>0$ and set $u=\lambda(s+t)$, $u_1=\lambda s$, $u_2=\lambda t$, as in Proposition~\ref{lem:hellinger-small-criterion-uncond}. Assume that for all sufficiently small $u$, the following conditions are satisfied:
\begin{enumerate}[label=\textup{(\roman*)}]
\item[(A)] under $\nu_u(\cdot\,|\,Z\neq\varnothing)$ the anchor $\alpha_u(Z)$ is $\mathrm{Unif}(0,u)$;
\item[(D)] (diameter control) $\nu_u\bigl(\mathrm{diam}_u(Z)\ge u^2\ \big|\ Z\neq\varnothing\bigr)\le u$.
\end{enumerate}
Consider the block-occupancy event $E_{10}:=\{Z\cap[0,u_1]\neq\varnothing,\ Z\cap(u_1,u]=\varnothing\}$ and 
$E_{01}:=\{Z\cap[0,u_1]=\varnothing,\ Z\cap[u_1,u]\neq\varnothing\}$ from Proposition~\ref{lem:hellinger-small-criterion-uncond}.
As in Corollary~\ref{lem:hellinger-small-criterion-cond}, let $\widehat\nu_u^{10}$ be the conditional law of $R_{0,u_1}(Z)$ under $\nu_u(\cdot\,|\,E_{10})$, and
let $\widehat\nu_u^{01}$ be the conditional law of $R_{u_1,u}(Z)$
under $\nu_u(\cdot\,|\,E_{01})$. Define the comparison laws:
\[
\bar\nu_{u}^{(u_1)}:=\Law_{\nu_u}\bigl(R_{0,u_1}(Z)\ \big|\ R_{0,u_1}(Z)\neq\varnothing\bigr),
\qquad
\bar\nu_{u}^{(u_2)}:=\Law_{\nu_u}\bigl(R_{u_1,u}(Z)\ \big|\ R_{u_1,u}(Z)\neq\varnothing\bigr).
\]
Then there exists $C_{s,t}<\infty$ such that for all sufficiently small $\lambda>0$,
\[
\bigl\|\widehat\nu_u^{10}-\bar\nu_{u}^{(u_1)}\bigr\|_{\mathrm{TV}}\le C_{s,t}\,u,
\qquad
\bigl\|\widehat\nu_u^{01}-\bar\nu_{u}^{(u_2)}\bigr\|_{\mathrm{TV}}\le C_{s,t}\,u.
\]
\end{lemma}

\begin{proof}
We use the elementary inequality: if $X$ is a random variable and $A,C$ are events with
$\mathbb P(A\cap C)>0$, then
\begin{equation}\label{eq:tv-by-bad-event-corrected}
\bigl\|\Law(X\,|\,A\cap C)-\Law(X\,|\,A)\bigr\|_{\mathrm{TV}}
\le \mathbb P(C^c\,|\,A).
\end{equation}
Indeed, for every $F\in\sigma(X)$,
$\mathbb P(F\,|\,A)
=
\mathbb P(F\,|\,A\cap C)\,\mathbb P(C\,|\,A)
+
\mathbb P(F\,|\,A\cap C^c)\,\mathbb P(C^c\,|\,A),$
and therefore
\[
\sup_{F\in\sigma(X)}
\bigl|
\mathbb P(F\,|\,A\cap C)-\mathbb P(F\,|\,A)
\bigr|
\le
\mathbb P(C^c\,|\,A).
\]

\smallskip\noindent
Let now
\[
X:=R_{0,u_1}(Z)=Z\cap[0,u_1],\qquad
A:=\{X\neq\varnothing\}=\{Z\cap[0,u_1]\neq\varnothing\},
\qquad
C:=\{Z\cap(u_1,u]=\varnothing\}.
\]
Then $E_{10}=A\cap C$, and by definition
\[
\widehat\nu_u^{10}=\Law_{\nu_u}(X\,|\,A\cap C),
\qquad
\bar\nu_{u}^{(u_1)}=\Law_{\nu_u}(X\,|\,A).
\]
Hence \eqref{eq:tv-by-bad-event-corrected} gives
\begin{equation}\label{eq:left-TV-corrected}
\bigl\|\widehat\nu_u^{10}-\bar\nu_{u}^{(u_1)}\bigr\|_{\mathrm{TV}}
\le
\nu_u(C^c\,|\,A)=\nu_u\bigl(Z\cap(u_1,u]\neq\varnothing\ \big|\ Z\cap[0,u_1]\neq\varnothing\bigr).
\end{equation}

We now bound $\nu_u(C^c\,|\,A)$. On the event $A$ we have $Z\neq\varnothing$ and
$\alpha_u(Z)\le u_1$ (because $\alpha_u(Z)=\inf Z$ for nonempty closed $Z$).
If, in addition, $\diam_u(Z)<u^2$ and $\alpha_u(Z)\le u_1-u^2$, then
\[
\sup Z \le \alpha_u(Z)+\diam_u(Z) < (u_1-u^2)+u^2=u_1,
\]
so $Z\cap(u_1,u]=\varnothing$, i.e.\ $C$ occurs. Thus
\[
C^c\cap A
\subseteq
\{\diam_u(Z)\ge u^2\}
\cup
\{\alpha_u(Z)\in(u_1-u^2,u_1]\},
\]
and therefore
\begin{equation}\label{eq:left-split-corrected}
\nu_u(C^c\,|\,A)
\le
\nu_u(\diam_u(Z)\ge u^2\,|\,A)
+
\nu_u(\alpha_u(Z)\in(u_1-u^2,u_1]\,|\,A).
\end{equation}

Since $\nu_u$ is a factorizing family, it satisfies Definition~\ref{def:fact-meas}\textup{(ii)}, so
$\nu_u(\{u_1\in Z\})=0$. Hence, up to a null set,
\[
A=\{Z\neq\varnothing,\ \alpha_u(Z)<u_1\}.
\]
By assumption~(A), under $\nu_u(\cdot\,|\,Z\neq\varnothing)$ the anchor is $\mathrm{Unif}(0,u)$.
Conditioning further on $A$ therefore yields $\alpha_u\sim \mathrm{Unif}(0,u_1)$, and so
\[
\nu_u(\alpha_u(Z)\in(u_1-u^2,u_1]\,|\,A)=\frac{u^2}{u_1}=\frac{s+t}{s}\,u.
\]

For the diameter term, since $A\subseteq\{Z\neq\varnothing\}$,
\[
\nu_u(\diam_u(Z)\ge u^2\,|\,A)
\le
\frac{\nu_u(\diam_u(Z)\ge u^2\,|\,Z\neq\varnothing)}
{\nu_u(A\,|\,Z\neq\varnothing)}.
\]
By~(D), the numerator is at most $u$, while by~(A),
\[
\nu_u(A\,|\,Z\neq\varnothing)
=
\nu_u(\alpha_u<u_1\,|\,Z\neq\varnothing)
=
\frac{u_1}{u}
=
\frac{s}{s+t}.
\]
Hence
\[
\nu_u(\diam_u(Z)\ge u^2\,|\,A)
\le
\frac{s+t}{s}\,u.
\]
Substituting into \eqref{eq:left-split-corrected} gives $\nu_u(C^c\,|\,A)\le 2\frac{s+t}{s}\,u.$
Together with \eqref{eq:left-TV-corrected}, this yields
\[
\bigl\|\widehat\nu_u^{10}-\bar\nu_{u}^{(u_1)}\bigr\|_{\mathrm{TV}}
\le
2\frac{s+t}{s}\,u.
\]

\smallskip\noindent
The right-block estimate is obtained by a similar argument. For this, set
\[
Y:=R_{u_1,u}(Z),\qquad
A':=\{Y\neq\varnothing\}=\{Z\cap[u_1,u]\neq\varnothing\},
\qquad
C':=\{Z\cap[0,u_1]=\varnothing\}.
\]
Then $E_{01}=A'\cap C'$, and by definition $\widehat\nu_u^{01}=\Law_{\nu_u}(Y\,|\,A'\cap C'),$
$\bar\nu_{u}^{(u_2)}=\Law_{\nu_u}(Y\,|\,A').$
Applying \eqref{eq:tv-by-bad-event-corrected} again gives
\begin{equation}\label{eq:right-TV-corrected}
\bigl\|\widehat\nu_u^{01}-\bar\nu_{u}^{(u_2)}\bigr\|_{\mathrm{TV}}
\le
\nu_u((C')^c\,|\,A').
\end{equation}

We bound $\nu_u((C')^c\,|\,A')$. If $(C')^c\cap A'$ occurs, then $Z$ has at least one point in
$[0,u_1]$ and at least one point in $[u_1,u]$. Hence $\alpha_u(Z)\le u_1$. If, moreover,
$\diam_u(Z)<u^2$, then $u_1-\alpha_u(Z)\le \diam_u(Z)<u^2,$
so necessarily $\alpha_u(Z)\in(u_1-u^2,u_1]$. Thus
\[
(C')^c\cap A'
\subseteq
\{\diam_u(Z)\ge u^2\}
\cup
\{\alpha_u(Z)\in(u_1-u^2,u_1]\}.
\]
Therefore
\[
\nu_u((C')^c\,|\,A')
=
\frac{\nu_u((C')^c\cap A'\,|\,Z\neq\varnothing)}
{\nu_u(A'\,|\,Z\neq\varnothing)}
\le
\frac{
\nu_u(\diam_u(Z)\ge u^2\,|\,Z\neq\varnothing)
+
\nu_u(\alpha_u(Z)\in(u_1-u^2,u_1]\,|\,Z\neq\varnothing)
}{
\nu_u(A'\,|\,Z\neq\varnothing)
}.
\]
By~(D), the first term in the numerator is at most $u$. By~(A),
\[
\nu_u(\alpha_u(Z)\in(u_1-u^2,u_1]\,|\,Z\neq\varnothing)
=
\frac{u^2}{u}
=
u.
\]
Finally, the event $\{\alpha_u(Z)\ge u_1\}$ is contained in $A'$, so
\[
\nu_u(A'\,|\,Z\neq\varnothing)
\ge
\nu_u(\alpha_u(Z)\ge u_1\,|\,Z\neq\varnothing)
=
\frac{u_2}{u}
=
\frac{t}{s+t}.
\]
Hence $\nu_u((C')^c\,|\,A')
\le
\frac{u+u}{u_2/u}
=
2\frac{s+t}{t}\,u.$
Combining this with \eqref{eq:right-TV-corrected}, we obtain
\[
\bigl\|\widehat\nu_u^{01}-\bar\nu_{u}^{(u_2)}\bigr\|_{\mathrm{TV}}
\le
2\frac{s+t}{t}\,u.
\]

Choosing $C_{s,t}:=2(s+t)\Bigl(\frac1s\vee \frac1t\Bigr)$
gives both estimates at once.
\end{proof}

\subsection{Brownian zero-set seed and the type $\rm{III}$ construction}\label{sub:typeIII}

In this subsection we apply the general framework developed above to the zero set of Brownian motion started away from the origin. We first verify that the Brownian zero-set laws form a measurable Palm-uniformizable seed, then construct the corresponding normalized representative, check the one-block estimates needed for Hellinger-smallness, and finally deduce a type~$\mathrm{III}$ random-set product system from the infinite-product construction.

Fix $a>0$ and let $\{B_t\}_{t\ge 0}$ be a one-dimensional Brownian motion started at $B_0=a$ on a filtered
probability space $(\Omega,\mathcal F,(\mathcal F_u)_{u\ge0},\mathbb P_a)$.
For $t>0$ define the random closed zero set in $[0,t]$ by
\begin{equation}\label{eq:brownian-zero-set}
Z_t^{\mathrm{BM},a}:=\{u\in[0,t]:B_u=0\}\in\mathscr C_t^{\{\ast\}},
\qquad
\mu_t^{\mathrm{BM},a}:=\Law_{\mathbb P_a}(Z_t^{\mathrm{BM},a}).
\end{equation}
When necessary, we may realize Brownian motion on the canonical path space; the law of the zero set depends only on the law of the path.
Set \[\boldsymbol\mu^{\mathrm{BM},a}:=\{\mu_t^{\mathrm{BM},a}\}_{t>0}.\] 

We begin by collecting several classical facts about Brownian hitting times, last zeros, and Brownian meanders that will be used repeatedly in the analysis of the Brownian seed. Our primary references are \cite{RevuzYor1999} and \cite{BorodinSalminen2002}.

\begin{lemma}
\label{lem:bm-hit-lastzero}
(i) Under $\mathbb P_a$,  the first hitting time of $0$, $T_0=\inf\{u>0:B_u=0\}$, has the density
\begin{equation}\label{eq:bm-hit-density1}
f_a(x)=\frac{a}{\sqrt{2\pi}}\,x^{-3/2}\exp\!\Bigl(-\frac{a^2}{2x}\Bigr),\qquad x>0.
\end{equation}
Moreover, $\mu_t^{\mathrm{BM},a}(\varnothing)=\mathbb P_a(T_0>t)>0$.

(ii) Let $W$ be a standard Brownian motion started at $0$, and for $T>0$ put
\[
g_T:=\sup\{s\in[0,T]:W_s=0\}.
\]
Then $g_T$ has the arcsine density
\begin{equation}\label{eq:arcsine-density}
\mathbb P(g_T\in ds)=\frac{1}{\pi}\,\frac{1}{\sqrt{s(T-s)}}\,\mathbf 1_{(0,T)}(s)\,ds.
\end{equation}
Consequently, for every $c\ge 0$,
\begin{equation}\label{eq:laplace-arcsine}
\Phi_c(T):=\mathbb E\bigl[e^{-c g_T}\bigr]
=\int_0^T e^{-cs}\,\frac{ds}{\pi\sqrt{s(T-s)}}
= e^{-cT/2}\,I_0\!\Bigl(\frac{cT}{2}\Bigr),
\end{equation}
where $I_0$ is the modified Bessel function of the first kind.

(iii)\label{it:meander-endpoint}
Let $W^{(\varepsilon)}$ be Brownian motion started at $\varepsilon>0$ and let
$\mathbb P_\varepsilon$ be its law. Then the conditional law of $W_A^{(\varepsilon)}$ given
$\{T_0>A\}$ has a density on $(0,\infty)$
\[
y\ \longmapsto\ \frac{p_A(\varepsilon,y)-p_A(\varepsilon,-y)}{\mathbb P_\varepsilon(T_0>A)},
\qquad y>0,
\]
where $p_A(x,y)=\frac{1}{\sqrt{2\pi A}}e^{-(y-x)^2/(2A)}$ is the Gaussian kernel.
Moreover, as $\varepsilon\downarrow 0$, these conditional laws converge weakly to a law
on $(0,\infty)$ with density
\begin{equation}\label{eq:rayleigh-density}
y\ \longmapsto\ \frac{y}{A}\exp\!\Bigl(-\frac{y^2}{2A}\Bigr),\qquad y>0,
\end{equation}
i.e.\ the Rayleigh law with parameter $\sqrt A$.
\end{lemma}

\begin{proof}
(i) is standard and follows from the reflection principle.

(ii) The arcsine density \eqref{eq:arcsine-density} is classical. For \eqref{eq:laplace-arcsine},
set $s=\frac{T}{2}(1+\cos\theta)$ so that $ds=\!-\frac{T}{2}\sin\theta\,d\theta$ and
$\sqrt{s(T-s)}=\frac{T}{2}\sin\theta$. Then
\[
\Phi_c(T)=\frac1\pi\int_0^\pi \exp\!\Bigl(-\frac{cT}{2}(1+\cos\theta)\Bigr)\,d\theta
=e^{-cT/2}\,\frac1\pi\int_0^\pi e^{-(cT/2)\cos\theta}\,d\theta
=e^{-cT/2}I_0\!\Bigl(\frac{cT}{2}\Bigr).
\]

(iii) Fix $\varepsilon>0$. By the reflection principle, Brownian motion killed upon hitting $0$
has transition density on $(0,\infty)$ given by the Dirichlet heat kernel
\[
p_A^{(0,\infty)}(\varepsilon,y)=p_A(\varepsilon,y)-p_A(\varepsilon,-y),
\qquad y>0.
\]
Hence, for $y>0$,
\[
\mathbb P_\varepsilon\bigl(W_A\in dy,\ T_0>A\bigr)=p_A^{(0,\infty)}(\varepsilon,y)\,dy,
\qquad
\mathbb P_\varepsilon(T_0>A)=\int_0^\infty p_A^{(0,\infty)}(\varepsilon,y)\,dy.
\]
Therefore the conditional density of $W_A$ given $\{T_0>A\}$ is exactly
$y\mapsto p_A^{(0,\infty)}(\varepsilon,y)/\mathbb P_\varepsilon(T_0>A)$.

We now let $\varepsilon\downarrow0$.
First, the survival probability is explicit (again by reflection):
\[
\mathbb P_\varepsilon(T_0>A)=\mathbb P\Bigl(\sup_{0\le s\le A}W_s<\varepsilon\Bigr)
=2\Phi\Bigl(\frac{\varepsilon}{\sqrt A}\Bigr)-1
\sim \sqrt{\frac{2}{\pi}}\frac{\varepsilon}{\sqrt A},
\qquad \varepsilon\downarrow0,
\]
where $\Phi$ is the standard normal cdf.

Second, for each fixed $y>0$, a Taylor expansion in $\varepsilon$ gives
\[
p_A(\varepsilon,y)-p_A(\varepsilon,-y)
=
p_A(y-\varepsilon)-p_A(y+\varepsilon)
=
-2\varepsilon\,\partial_y p_A(y)+o(\varepsilon)
=
2\varepsilon\,\frac{y}{A}\,p_A(0,y)+o(\varepsilon),
\]
and $p_A(0,y)=\frac{1}{\sqrt{2\pi A}}e^{-y^2/(2A)}$.
Combining the numerator and denominator asymptotics yields
\[
\frac{p_A(\varepsilon,y)-p_A(\varepsilon,-y)}{\mathbb P_\varepsilon(T_0>A)}
\ \longrightarrow\
\frac{y}{A}\exp\!\Bigl(-\frac{y^2}{2A}\Bigr),
\qquad y>0,
\]
which is exactly \eqref{eq:rayleigh-density}. 
\end{proof}

The next lemma shows that changing the nonzero starting point of Brownian motion does not affect the zero-set law beyond equivalence of measures.
\begin{lemma}\text{(cf. \cite[Lemma 1.1]{Tsirelson00})}\label{lem:bm-start-equivalence}
For every $t>0$ and all $x,y\in\mathbb R\setminus\{0\}$ one has $\mu^{\mathrm{BM},x}_t\sim \mu^{\mathrm{BM},y}_t$.
\end{lemma}

\begin{proof}
Fix $t>0$ and $x\neq 0$, and let $T_0:=\inf\{u>0:B_u=0\}$ be the first hitting time of $0$.
Then
$\{Z_t=\varnothing\}=\{T_0>t\}.$
On the event $\{T_0=r\le t\}$ the path has no zeros in $[0,r)$ and $B_r=0$, hence
\begin{equation}\label{eq:bm-decomp-at-hit}
Z_t^{\mathrm{BM},a}=\{T_0\}\,\cup\,\bigl(T_0+Z_{t-T_0}^{\mathrm{BM},0}
\bigr) \quad\text{a.s}
\end{equation}
where $Z_{\cdot}^{\mathrm{BM},0}$ denotes the Brownian zero set for a Brownian motion started at $0$ (independent of the past),
and where we interpret $T_0+Z_{t-T_0}^{\mathrm{BM},0}:=\{T_0+u:u\in Z_{t-T_0}^{\mathrm{BM},0}\}\subset[0,t]$.
The identity \eqref{eq:bm-decomp-at-hit} follows from the strong Markov property at the stopping time $T_0$:
after time $T_0$ the process restarts from $0$ and is independent of the pre-$T_0$ history, while before $T_0$
it stays away from $0$.

Let $Q_{t,r}$ be the (probability) law on $\mathscr C_t^{\{\ast\}}$ of the random set
$\{r\}\cup(r+Z_{t-r}^{\mathrm{BM},0})$.
Then, disintegrating with respect to $T_0$,
\begin{equation}\label{eq:bm-disintegration}
\mu^{\mathrm{BM},x}_t(A)
=
\mathbb P_x(T_0>t)\,\1_A(\varnothing)
+\int_{0}^{t} Q_{t,r}(A)\,\mathbb P_x(T_0\in dr),
\qquad A\in\Sigma_t^{\{\ast\}}.
\end{equation}
By Lemma~\ref{lem:bm-hit-lastzero}(i),
\begin{equation}\label{eq:bm-hit-density}
\mathbb P_x(T_0\in dr)
=
\frac{|x|}{\sqrt{2\pi}\,r^{3/2}}\exp\!\Bigl(-\frac{x^2}{2r}\Bigr)\,dr,
\qquad
 (r>0),
\end{equation}
Therefore \begin{equation}\label{eq:bm-hit-density2}\mathbb P_x(T_0\in dr)\sim \mathrm{Leb}\!\restriction_{(0,t]}\end{equation} on $(0,t]$ for every $x\neq 0$.

Now let $x,y\neq 0$ and let $A\in\Sigma_t^{\{\ast\}}$ with $\mu^{\mathrm{BM},y}_t(A)=0$.
By \eqref{eq:bm-disintegration} applied to $y$, the term $\mathbb P_y(T_0>t)\,\1_A(\varnothing)$ forces
$\varnothing\notin A$ (since $\mathbb P_y(T_0>t)>0$), and then the integral term gives
$Q_{t,r}(A)=0$ for $\mathrm{Leb}$-a.e.\ $r\in(0,t]$ because the density \eqref{eq:bm-hit-density} is strictly positive.
Plugging this into \eqref{eq:bm-disintegration} for $x$ yields $\mu^{\mathrm{BM},x}_t(A)=0$.
Thus $\mu^{\mathrm{BM},x}_t\ll \mu^{\mathrm{BM},y}_t$.
By symmetry, $\mu^{\mathrm{BM},y}_t\ll \mu^{\mathrm{BM},x}_t$, proving $\boldsymbol\mu^{\mathrm{BM},x}\sim \boldsymbol\mu^{\mathrm{BM},y}$.
\end{proof}
The following proposition verifies that Brownian zero sets fit the abstract framework developed earlier by forming a measurable factorizing family.

\begin{proposition}
\label{prop:bm-meas-factorizing}
The family $\boldsymbol\mu^{\mathrm{BM},a}=\{\mu_t^{\mathrm{BM},a}\}_{t>0}$
is a measurable factorizing family over $[0,1]\times\{\ast\}$.
\end{proposition}

\begin{proof}
We first establish factorization in the sense of Definition~\ref{def:fact-meas}(i). Fix $s,t>0$ and set $u=s+t$.
Let $L:\mathscr C_u^{\{\ast\}}\to\mathscr C_s^{\{\ast\}}$ and $R:\mathscr C_u^{\{\ast\}}\to\mathscr C_t^{\{\ast\}}$ be the restriction maps
\[
L(Z):=Z\cap[0,s],\qquad R(Z):=(Z\cap[s,u])-s.
\]
These maps are Borel. 
Let $J_{s,t}$ be the joint law of $(L(Z_u^{\mathrm{BM},a}),R(Z_u^{\mathrm{BM},a}))$ under $\mathbb P_a$:
\[
J_{s,t}:=\Law_{\mathbb P_a}\bigl(L(Z_u^{\mathrm{BM},a}),R(Z_u^{\mathrm{BM},a})\bigr)\quad\text{on }\mathscr C_s^{\{\ast\}}\times \mathscr C_t^{\{\ast\}}.
\]
Then $L(Z_u^{\mathrm{BM},a})$ has law $\mu_s^{\mathrm{BM},a}$ 
and $R(Z_u^{\mathrm{BM},a})$ has some law on $\mathscr C_t^{\{\ast\}}$. Moreover, the reconstruction identity
\[
Z_u^{\mathrm{BM},a}=L(Z_u^{\mathrm{BM},a})\ \oplus_{s,t}\ R(Z_u^{\mathrm{BM},a})
\]
holds almost surely, hence
\begin{equation}\label{eq:bm-reconstruct}
\mu_u^{\mathrm{BM},a}= J_{s,t}\circ \oplus_{s,t}^{-1}.
\end{equation}

We show that $J_{s,t}\sim \mu_s^{\mathrm{BM},a}\otimes\mu_t^{\mathrm{BM},a}$ on the product space. Since the left marginal of $J_{s,t}$ is $\mu_s^{\mathrm{BM},a}$,
it suffices to show that, for $\mu_s^{\mathrm{BM},a}$-a.e.\ $z$, the conditional law of $R(Z_u^{\mathrm{BM},a})$ given $L(Z_u^{\mathrm{BM},a})=z$ is equivalent to $\mu_t^{\mathrm{BM},a}$.
Let $\{K_z\}_{z\in\mathscr C_s^{\{\ast\}}}$ be a regular conditional distribution of $R(Z_u^{\mathrm{BM},a})$ given $L(Z_u^{\mathrm{BM},a})=z$ under $J_{s,t}$. By the (strong) Markov property at the deterministic time $s$, conditional on $\mathcal F_s$ the future process
$\{B_{s+r}\}_{0\le r\le t}$ is a Brownian motion started at $B_s$ and independent of the past.
Therefore, for every $A\in\Sigma_t^{\{\ast\}}$,
\begin{equation}\label{eq:bm-kernel-markov}
\mathbb P_a\bigl(R(Z_u^{\mathrm{BM},a})\in A\,\big|\,\mathcal F_s\bigr)=\mu_t^{\mathrm{BM},B_s}
(A)\qquad\text{a.s.}
\end{equation}
Taking conditional expectation of \eqref{eq:bm-kernel-markov} with respect to $\sigma(L(Z_u^{\mathrm{BM},a}))$ yields
\begin{equation}\label{eq:bm-kernel-mixture}
K_{L(Z_u^{\mathrm{BM},a})}(A)
=
\mathbb P_a\bigl(R(Z_u^{\mathrm{BM},a})\in A\,\big|\,L(Z_u^{\mathrm{BM},a})\bigr)
=
\mathbb E_a\bigl[\mu_t^{\mathrm{BM},B_s}(A)
\,\big|\,L(Z_u^{\mathrm{BM},a})\bigr]\qquad\text{a.s.}
\end{equation}
Thus $K_z$ is a mixture of the measures $\mu_t^{\mathrm{BM},x}$
, with mixing distribution given by the conditional law of $B_s$
given $L(Z_u^{\mathrm{BM},a})=z$. Since $\mathbb P_a(B_s=0)=0$, we have 
$\mu_s^{\mathrm{BM},a}$-a.s.\ that this conditional law assigns mass $0$ to $\{0\}$.

By Lemma~\ref{lem:bm-start-equivalence}, for every $x\neq 0$ we have $\mu_t^{\mathrm{BM},x}\sim \mu_t^{\mathrm{BM},a}$, 
so all $\mu_t^{\mathrm{BM},x}$ 
($x\neq 0$) have the same null sets as $\mu_t^{\mathrm{BM},a}$. 
It follows immediately from \eqref{eq:bm-kernel-mixture} that $K_z\sim \mu_t^{\mathrm{BM},a}$ 
for 
$\mu_s^{\mathrm{BM},a}$-a.e.\ $z$:
indeed, if 
$\mu_t^{\mathrm{BM},a}(A)=0$ then 
$\mu_t^{\mathrm{BM},x}(A)=0$ for all $x\neq 0$, hence $K_z(A)=0$;
conversely, if $K_z(A)=0$ then 
$\mu_t^{\mathrm{BM},x}(A)=0$ for mixing-a.e.\ $x$, and therefore (since null sets are common)
$\mu_t^{\mathrm{BM},a}(A)=0$.
Hence $J_{s,t}\sim \mu_s^{\mathrm{BM},a}\otimes\mu_t^{\mathrm{BM},a}$, and pushing forward by $\oplus_{s,t}$ gives
\[
(\mu_s^{\mathrm{BM},a}\otimes\mu_t^{\mathrm{BM},a})\circ \oplus_{s,t}^{-1}\ \sim\ J_{s,t}\circ \oplus_{s,t}^{-1}\ =\ \mu_{s+t}^{\mathrm{BM},a}
\]
using \eqref{eq:bm-reconstruct}. This is exactly Definition~\ref{def:fact-meas}\textup{(i)}.

\smallskip\noindent

The absence of deterministic time-slices condition (Definition~\ref{def:fact-meas}\textup{(ii)}) is immediate: for any fixed $r\in(0,t]$,
$\mu_t^{\mathrm{BM},a}(\{Z:r\in Z\})=\mathbb P_a(B_r=0)=0,$
since $B_r$ has a continuous density. Thus the family has no deterministic points.

We verify the non-degeneracy condition in Definition~\ref{def:fact-meas}\textup{(iii)}.
Suppose towards a contradiction that $\mu_t^{\mathrm{BM},a}$ is supported on finitely many atoms, i.e.
$\mu_t^{\mathrm{BM},a}=\sum_{i=1}^N p_i\,\delta_{C_i}$ with $p_i>0$ and distinct $C_i\in\mathscr C_t^{\{\ast\}}$.
Define the Borel functional
\[
\beta_t:\mathscr C_t^{\{\ast\}}\to[0,t]\cup\{t+1\},\qquad
\beta_t(C):=
\begin{cases}
\inf C, & C\neq\varnothing,\\
t+1, & C=\varnothing.
\end{cases}
\]
Let $T_0=\inf\{u\ge 0:B_u=0\}$ be the first hitting time of $0$. Then
$\beta_t(Z_t^{\mathrm{BM},a})=T_0$ on $\{T_0\le t\}$ and $\beta_t(Z_t^{\mathrm{BM},a})=t+1$ on $\{T_0>t\}$.
Using again Lemma~\ref{lem:bm-hit-lastzero}(i), we have $\mathbb P_a(\tau_0=r)=0$ for every $r\in(0,t]$. Hence
$\mu_t^{\mathrm{BM},a}(\beta_t=r)=0$ for all $r\in(0,t]$.
On the other hand, under the finite-atomic representation,
$(\beta_t)_\#\mu_t^{\mathrm{BM},a}=\sum_{i=1}^N p_i\,\delta_{\beta_t(C_i)}$
is purely atomic, and since $\mathbb P_a(T_0\le t)>0$ there exists at least one $i$ with
$\beta_t(C_i)\in(0,t]$, implying $\mu_t^{\mathrm{BM},a}(\beta_t=\beta_t(C_i))\ge p_i>0$,
a contradiction. Therefore $\mu_t^{\mathrm{BM},a}$ is not supported on finitely many atoms.

We next verify the measurability condition
Definition~\ref{def:fact-meas}\textup{(iva)}. Write, for simplicity, $\mathscr C:=\mathscr C_1^{\{\ast\}},$ $\Sigma:=\Sigma_1^{\{\ast\}},$ and let $\sigma_t:\mathscr C_t^{\{\ast\}}\to\mathscr C$ be the time-scaling homeomorphism.
Set $\widetilde\mu_t:=(\sigma_t)_*\mu_t^{\mathrm{BM},a}.$ Let $\mathcal R$ be the countable ring generated by the miss sets $M(K):=\{W\in\mathscr C: W\cap K=\varnothing\},$
where $K$ ranges over finite unions of closed intervals with rational endpoints in $[0,1]$,
and the hit sets $H(U):=\{W\in\mathscr C: W\cap U\neq\varnothing\},$
where $U$ ranges over finite unions of open intervals with rational endpoints in $[0,1]$.
Since such closed and open rational intervals form countable bases for the compact and open subsets of $[0,1]$,
the ring $\mathcal R$ is countable and generates $\Sigma$.

Realize Brownian motion on the canonical Wiener space $\Omega:=C([0,\infty),\mathbb R)$
with its Borel $\sigma$--field, and let $\mathbb P_a$ be Wiener measure started at $a$.
Consider the map $\Psi:(0,\infty)\times\Omega\to\mathscr C$,
\[
\Psi(t,\omega):=\sigma_t(Z_t^{\mathrm{BM},a}(\omega))
=
\{u\in[0,1]:\omega(tu)=0\}.
\]
Then, by definition, $\widetilde\mu_t=\Law_{\mathbb P_a}(\Psi(t,\cdot)).$

We claim that $\Psi$ is Borel. It suffices to check the preimages of the generators of $\mathcal R$.

Let $K\subset[0,1]$ be compact. Then
\[\{\Psi(t,\omega)\in M(K)\}
=
\{\Psi(t,\omega)\cap K=\varnothing\}
=
\Bigl\{\inf_{u\in K}|\omega(tu)|>0\Bigr\}.
\]
The map $(t,\omega,u)\mapsto |\omega(tu)|$
is continuous on $(0,\infty)\times\Omega\times[0,1]$, hence $
(t,\omega)\mapsto \inf_{u\in K}|\omega(tu)|$ is continuous
as a infimum over a compact set. Thus 
$\{(t,\omega):\Psi(t,\omega)\in M(K)\}$
is open, in particular Borel.

Let $U\subset[0,1]$ be open.
Choose compact sets $K_n\subset U$ such that $K_n\uparrow U$ (for example, one may take $K_n:=\{u\in U:\ \mathrm{dist}(u,U^c)\ge 1/n\}$). Then each $K_n$ is compact and contained in $U$, and $U=\bigcup_{n=1}^\infty K_n.$
Therefore
\[
\{\Psi(t,\omega)\in H(U)\}
=
\{\Psi(t,\omega)\cap U\neq\varnothing\}
=
\bigcup_{n=1}^\infty \{\Psi(t,\omega)\cap K_n\neq\varnothing\}.
\]
For each $n$, $\{\Psi(t,\omega)\cap K_n\neq\varnothing\}
=
\{\inf_{u\in K_n}|\omega(tu)|=0\},$
and the right-hand side is Borel because
$(t,\omega)\mapsto \inf_{u\in K_n}|\omega(tu)|$ is continuous.
Hence $\{(t,\omega):\Psi(t,\omega)\in H(U)\}$
is Borel.

Thus the preimage under $\Psi$ of every generator of $\mathcal R$ is Borel, and therefore
for every $A\in\mathcal R$ the map $(t,\omega)\mapsto \1_A(\Psi(t,\omega))$
is Borel on $(0,\infty)\times\Omega$. Finally, since $0\le \1_A(\Psi(t,\omega))\le 1$, standard measurability of parameter-dependent integrals yields that
\[
t\longmapsto \widetilde\mu_t(A)
=
\mathbb E_a\bigl[\1_A(\Psi(t,\cdot))\bigr]
\]
is Borel for every $A\in\mathcal R$. This proves Definition~\ref{def:fact-meas}\textup{(iva)}.

Finally, we verify measurability of the Radon–Nikodym derivatives (Definition~\ref{def:fact-meas}\textup{(ivb)}).
For $s,t>0$, let
\[
\widetilde\mu_{s,t}:=(\sigma_{s+t})_*\Bigl((\mu_s^{\mathrm{BM},a}\otimes\mu_t^{\mathrm{BM},a})\circ \oplus_{s,t}^{-1}\Bigr)\quad\text{on }(\mathscr C,\Sigma).
\]
By factorization, $\widetilde\mu_{s,t}\sim \widetilde\mu_{s+t}$ for all $s,t>0$.

We now check the measurability needed to apply Lemma~\ref{lem:measurable-RN}.
Let $u=\frac{s}{s+t}\in(0,1)$ and define the Borel gluing map $G_u:\mathscr C\times\mathscr C\to\mathscr C$ by
\[
G_u(W_1,W_2):=\sigma_u^{-1}(W_1)\ \cup\ \bigl(u+\sigma_{1-u}^{-1}(W_2)\bigr),
\]
so that $\widetilde\mu_{s,t}=(\widetilde\mu_s\otimes \widetilde\mu_t)\circ G_{s/(s+t)}^{-1}$.
Since $u\mapsto G_u$ is Borel in the sense that $(u,W_1,W_2)\mapsto G_u(W_1,W_2)$ is Borel, and since
$t\mapsto \widetilde\mu_t(A)$ is Borel on the ring $\mathcal R$ introduced above,
Lemma~\ref{lem:meas-integral} implies that $(s,t)\mapsto \widetilde\mu_{s,t}(A)$ is Borel for each $A\in\mathcal R$.
Therefore Lemma~\ref{lem:measurable-RN} (with parameter space $(0,\infty)^2$ and base space $\mathscr C$) yields a jointly Borel function
\[
\Delta:(0,\infty)^2\times\mathscr C\to(0,\infty)
\quad\text{such that}\quad
\Delta(s,t,\cdot)=\frac{d\widetilde\mu_{s,t}}{d\widetilde\mu_{s+t}}
\ \ \widetilde\mu_{s+t}\text{-a.e.}
\]
This is precisely Definition~\ref{def:fact-meas}\textup{(ivb)}.
\end{proof}

\begin{proposition}
\label{lem:bm-palm-unif}
The measurable factorizing family 
$\boldsymbol\mu^{\mathrm{BM},a}=\{\mu_t^{\mathrm{BM},a}\}_{t>0}$ 
is a Palm-uniformizable seed.
\end{proposition}

\begin{proof}
By \eqref{eq:bm-hit-density2}, for each $t>0$ the anchor marginal 
$\kappa_t^{\mu}$ associated with $\mu_t^{\mathrm{BM},a}$ 
is equivalent to $\mathrm{Leb}\!\restriction_{(0,t]}$. 
Thus $\boldsymbol\mu^{\mathrm{BM},a}$ is Palm-uniformizable.

Moreover, by a fundamental result of Tsirelson~\cite[Section 2]{Tsirelson00}, 
the associated random-set system is of type $\mathrm{II}_0$.
\end{proof}

We now apply Theorem~\ref{prop:palm-loc-vacuum} with $L(t)=|\ln t|$ (for $t\in(0,1)$) to the family
$\boldsymbol\mu^{\mathrm{BM},a}$.
This produces an equivalent representative\begin{equation}\label{eq:brownian-meas}
\boldsymbol\nu^{\mathrm{BM},a}=\{\nu_t^{\mathrm{BM},a}\}_{t>0}\end{equation} satisfying
$\nu_t^{\mathrm{BM},a}(\varnothing)=e^{-\beta t}$ and such that, conditional on $\{Z\neq\varnothing\}$,
the anchor $\alpha_t(Z)$ is exactly $\mathrm{Unif}(0,t)$. Recall that $\nu_t^{\mathrm{BM},a}$ is obtained from $\mu_t^{\mathrm{BM},a}$ by anchor--adapted localization (Definition~\ref{def:log-local-anchor} with $L(r)=|\ln r|$),
then Palm uniformization (Lemma~\ref{lem:palm-uniformization}),
and finally vacuum normalization (Theorem~\ref{prop:palm-loc-vacuum}).

By Theorem~\ref{prop:palm-loc-vacuum} and the discussion after
Proposition~\ref{lem:hellinger-small-criterion-uncond}, for $\boldsymbol\nu^{\mathrm{BM},a}$
the hypotheses \textup{(i)}, \textup{(ii)} and \textup{(iii)} of
Proposition~\ref{lem:hellinger-small-criterion-uncond} hold for every fixed $(s,t)$.
It remains to verify \textup{(iv(a))} and \textup{(iv(b))} of
Corollary~\ref{lem:hellinger-small-criterion-cond}.

The following proposition computes explicitly the Radon–Nikodym density of the Brownian representative obtained after anchor-adapted localization, Palm uniformization, and vacuum normalization. Although the explicit expression will not be needed in the arguments below, we include it for completeness and future reference.
\begin{proposition}
\label{prop:bm-explicit-density}
For any $t\in(0,1]$ the Radon--Nikodym derivative
$f_t:=d\nu_t^{\mathrm{BM},a}/d\mu_t^{\mathrm{BM},a}$ is
\[
f_t(Z)=\frac{e^{-\beta t}}{\mu_t^{\mathrm{BM},a}(\varnothing)}\,\1_{\{Z=\varnothing\}}
\;+\;
\frac{1-e^{-\beta t}}{t}\,
\frac{\exp\!\bigl(-c(\alpha_t(Z))\,\mathrm{diam}_t(Z)\bigr)}{f_a(\alpha_t(Z))\,\Phi_{c(\alpha_t(Z))}(t-\alpha_t(Z))}\,
\1_{\{Z\neq\varnothing\}},
\]
where $f_a$ is the hitting-time density, $\Phi_c$ is the arcsine Laplace transform from Lemma~\ref{lem:bm-hit-lastzero}(ii), and  $c(x):=|\ln x |/x^2$ for $x\in(0,1)$.
\end{proposition}

\begin{proof}
Fix $t\in(0,1]$ and write $\mu_t:=\mu_t^{\mathrm{BM},a}$ for simplicity. Consider the (strictly positive) localization weight
\[
w_t(Z):=\1_{\{Z=\varnothing\}}(Z)
+\1_{\{Z\neq\varnothing\}}(Z)\exp\!\bigl(-c(\alpha_t(Z))\,\mathrm{diam}_t(Z)\bigr),
\qquad Z\in\mathscr C_t^{\{\ast\}}.
\]
Let $Z_t^{\mathrm{loc},\alpha}:=\EE_{\mu_t}\!\bigl[w_t(Z)\bigr]\in(0,\infty),$ and let $\mu_t^{\mathrm{loc},\alpha}$ be the anchor-adapted localized tilt of Definition~\ref{def:log-local-anchor}, i.e. 
\[
\frac{d\mu_t^{\mathrm{loc},\alpha}}{d\mu_t}(Z)=\frac{w_t(Z)}{Z_t^{\mathrm{loc},\alpha}}.
\]
On $\{Z\neq\varnothing\}$ this reads
\[
\frac{d\mu_t^{\mathrm{loc},\alpha}}{d\mu_t}(Z)
=\frac{1}{Z_t^{\mathrm{loc},\alpha}}\exp\!\bigl(-c(\alpha_t(Z))\,\mathrm{diam}_t(Z)\bigr).
\]
Under $\PP_a$, let $T_0:=\inf\{u>0:B_u=0\}$ be the first hitting time of $0$.
Then on $\{Z\neq\varnothing\}=\{T_0\le t\}$ we have $\alpha_t(Z)=T_0$ and, by the strong Markov property
at time $T_0$, the post-$T_0$ process is a Brownian motion started at $0$ independent of the past.
Let $g_T:=\sup\{s\in[0,T]:W_s=0\}$ for a standard Brownian motion $W$ started at $0$.
Then, conditional on $\{T_0=x\le t\}$, the last zero time in $[0,t]$ equals $x+g_{t-x}$, hence
\[
\mathrm{diam}_t(Z)=\sup Z-\inf Z=(x+g_{t-x})-x=g_{t-\alpha_t(Z)}.
\]
Therefore, conditional on $\{\alpha_t=x\}$, the localization factor has conditional expectation
\[
\EE_{\mu_t}\!\Bigl[\exp\!\bigl(-c(\alpha_t(Z))\mathrm{diam}_t(Z)\bigr)\,\Big|\,\alpha_t=x\Bigr]
=\EE\!\bigl[e^{-c(x)g_{t-x}}\bigr]
=\Phi_{c(x)}(t-x).
\]

Let $\kappa_t^{\mathrm{loc},\alpha}:=(\alpha_t)_*\bigl(\mu_t^{\mathrm{loc},\alpha}(\cdot\cap\{Z\neq\varnothing\})\bigr)$ be the anchor pushforward on $(0,t]$.
Using that $\alpha_t=T_0$ has density $f_a$ on $(0,\infty)$ (Lemma~\ref{lem:bm-hit-lastzero}(i)), we obtain
for $x\in(0,t]$,
\begin{align*}
\kappa_t^{\mathrm{loc},\alpha}(dx)
&=\mu_t^{\mathrm{loc},\alpha}(\alpha_t\in dx,\ Z\neq\varnothing)\\
&=\frac{1}{Z_t^{\mathrm{loc},\alpha}}
\EE_{\mu_t}\!\Bigl[\1_{\{\alpha_t\in dx\}}\exp\!\bigl(-c(\alpha_t)\mathrm{diam}_t\bigr)\Bigr]\\
&=\frac{1}{Z_t^{\mathrm{loc},\alpha}}\,f_a(x)\,\Phi_{c(x)}(t-x)\,dx.
\end{align*}
Consequently, a (Borel) Radon--Nikodym derivative of $\mathrm{Leb}\!\restriction_{(0,t]}$ w.r.t.\
$\kappa_t^{\mathrm{loc},\alpha}$ is given by
\[
h_t(x):=\frac{d(\mathrm{Leb}\restriction_{(0,t]})}{d\kappa_t^{\mathrm{loc},\alpha}}(x)
=\frac{Z_t^{\mathrm{loc},\alpha}}{f_a(x)\,\Phi_{c(x)}(t-x)},
\qquad x\in(0,t],
\]
defined $\kappa_t^{\mathrm{loc},\alpha}$--a.e.

Let $\mu_t^{\mathrm{unif},\alpha}:=(\mu_t^{\mathrm{loc},\alpha})^{\mathrm{unif}}$ be the Palm-uniformization
of $\mu_t^{\mathrm{loc},\alpha}$ as in Lemma~\ref{lem:palm-uniformization}. By that lemma,
conditioning on $\{Z\neq\varnothing\}$, the resulting law is obtained from $\mu_t^{\mathrm{loc},\alpha}$
by multiplying by $h_t(\alpha_t(Z))$ and renormalizing so that $\alpha_t$ becomes uniform on $(0,t)$.
Equivalently, on $\{Z\neq\varnothing\}$ one has
\[
\frac{d\,\mu_t^{\mathrm{unif},\alpha}(\,\cdot\,|\,Z\neq\varnothing)}{d\mu_t^{\mathrm{loc},\alpha}}(Z)
=\frac{1}{t}\,h_t(\alpha_t(Z)).
\]
Combining this with $\frac{d\mu_t^{\mathrm{loc},\alpha}}{d\mu_t}=\frac{w_t}{Z_t^{\mathrm{loc},\alpha}}$
gives, on $\{Z\neq\varnothing\}$,
\begin{align*}
\frac{d\,\mu_t^{\mathrm{unif},\alpha}(\,\cdot\,|\,Z\neq\varnothing)}{d\mu_t}(Z)
&=\frac{1}{t}\,h_t(\alpha_t(Z))\,\frac{w_t(Z)}{Z_t^{\mathrm{loc},\alpha}}\\
&=\frac{1}{t}\,
\frac{Z_t^{\mathrm{loc},\alpha}}{f_a(\alpha_t(Z))\,\Phi_{c(\alpha_t(Z))}(t-\alpha_t(Z))}
\cdot
\frac{\exp\!\bigl(-c(\alpha_t(Z))\,\mathrm{diam}_t(Z)\bigr)}{Z_t^{\mathrm{loc},\alpha}}\\
&=\frac{1}{t}\,
\frac{\exp\!\bigl(-c(\alpha_t(Z))\,\mathrm{diam}_t(Z)\bigr)}{f_a(\alpha_t(Z))\,\Phi_{c(\alpha_t(Z))}(t-\alpha_t(Z))}.
\end{align*}

By Theorem~\ref{prop:palm-loc-vacuum}, the final representative is
\[
\nu_t^{\mathrm{BM},a}
=
e^{-\beta t}\,\delta_{\varnothing}
+(1-e^{-\beta t})\,\mu_t^{\mathrm{unif},\alpha}(\,\cdot\,|\,Z\neq\varnothing).
\]
Since $\mu_t(\{\varnothing\})=\PP_a(T_0>t)>0$, the Radon--Nikodym derivative $f_t=d\nu_t^{\mathrm{BM},a}/d\mu_t$
is \[
f_t(\varnothing)=\frac{\nu_t^{\mathrm{BM},a}(\{\varnothing\})}{\mu_t(\{\varnothing\})}
=\frac{e^{-\beta t}}{\mu_t^{\mathrm{BM},a}(\varnothing)}
\] on $\{Z=\varnothing\}$, and
\[
f_t(Z)
=(1-e^{-\beta t})\,
\frac{d\,\mu_t^{\mathrm{unif},\alpha}(\,\cdot\,|\,Z\neq\varnothing)}{d\mu_t}(Z)
=
\frac{1-e^{-\beta t}}{t}\,
\frac{\exp\!\bigl(-c(\alpha_t(Z))\,\mathrm{diam}_t(Z)\bigr)}{f_a(\alpha_t(Z))\,\Phi_{c(\alpha_t(Z))}(t-\alpha_t(Z))}
\]
on $\{Z\neq\varnothing\}$. This is exactly the claimed formula.
\end{proof}

The next proposition verifies the conditional one-block Hellinger control \textup{(iv(b))} for the Brownian seed.

\begin{proposition}
\label{lem:bm-one-block}
The seed $\boldsymbol\nu^{\mathrm{BM},a}$ satisfies condition \textup{(iv(a))} of
Corollary~\ref{lem:hellinger-small-criterion-cond}.
\end{proposition}

\begin{proof}
Fix $s,t>0$ and set $u=\lambda(s+t)$, $u_1=\lambda s$, $u_2=\lambda t$ as in
Corollary~\ref{lem:hellinger-small-criterion-cond}.
Write $\nu_r:=\nu_r^{\mathrm{BM},a}$ and $\sigma:=(\nu_{u_1}\otimes \nu_{u_2})\circ\oplus^{-1}$ for simplicity.

We compute the $\sigma$-weights. Under $\sigma$ the left and right blocks are independent, hence
\[
\sigma(E_{10})=\nu_{u_1}(Z\neq\varnothing)\,\nu_{u_2}(Z=\varnothing),
\qquad
\sigma(E_{01})=\nu_{u_1}(Z=\varnothing)\,\nu_{u_2}(Z\neq\varnothing).
\]
By vacuum normalization, $\nu_r(Z=\varnothing)=e^{-\beta r}$ and $\nu_r(Z\neq\varnothing)=1-e^{-\beta r}$ for all $r>0$.
Therefore
\[
\sigma(E_{10})=(1-e^{-\beta u_1})\,e^{-\beta u_2},
\qquad
\sigma(E_{01})=e^{-\beta u_1}\,(1-e^{-\beta u_2}),
\]
and using $1-e^{-x}=x+O(x^2)$ as $x\downarrow0$ we obtain
\begin{equation}\label{eq:sigma-asympt-bm}
\sigma(E_{10})=\beta u_1+O(u^2),\qquad \sigma(E_{01})=\beta u_2+O(u^2).
\end{equation}

By Palm uniformization (Lemma~\ref{lem:palm-uniformization} applied inside Theorem~\ref{prop:palm-loc-vacuum}),
under $\nu_u(\cdot\,|\,Z\neq\varnothing)$ the anchor $\alpha_u(Z)$ is $\mathrm{Unif}(0,u)$.
By Lemma~\ref{lem:diam-control-unif-anchor} (with $t=u$ and $L(u)=|\ln u|$),
\begin{equation}\label{eq:diam-tail-u-bm}
\nu_u\bigl(\mathrm{diam}_u(Z)\ge u^2 \,\big|\,Z\neq\varnothing\bigr)
\le e^{-|\ln u|}=u,
\end{equation}
for all sufficiently small $u\in(0,1)$.
Let $G_u:=\{\mathrm{diam}_u(Z)<u^2\}$.

On $G_u$, if $\alpha_u(Z)\le u_1-u^2$ then necessarily $Z\subset[0,u_1]$, hence $E_{10}$ occurs.
Thus
\[
\nu_u(E_{10}\,|\,Z\neq\varnothing)
\ge
\nu_u(\alpha_u\le u_1-u^2\,|\,Z\neq\varnothing)-\nu_u(G_u^c\,|\,Z\neq\varnothing)
=
\frac{u_1-u^2}{u}-O(u)
=
\frac{u_1}{u}+O(u).
\]
Similarly, $E_{10}\subset\{\alpha_u\le u_1\}$ and on $\{\alpha_u\le u_1\}$ the only way $E_{10}$ can fail is
either $G_u^c$ or $\alpha_u\in(u_1-u^2,u_1]$. Therefore
\[
\nu_u(E_{10}^c\,|\,Z\neq\varnothing,\ \alpha_u\le u_1)
\le
\nu_u(G_u^c\,|\,Z\neq\varnothing)
+
\nu_u(\alpha_u\in(u_1-u^2,u_1]\,|\,Z\neq\varnothing,\alpha_u\le u_1)
\le
u+\frac{u^2}{u_1}
=
O(u).
\]
Hence
\[
\nu_u(E_{10}\,|\,Z\neq\varnothing)
=
\nu_u(\alpha_u\le u_1\,|\,Z\neq\varnothing)\,(1-O(u))
=
\frac{u_1}{u}+O(u).
\]
Multiplying by $\nu_u(Z\neq\varnothing)=1-e^{-\beta u}=\beta u+O(u^2)$ yields
\begin{equation}\label{eq:nu-E10-asympt-bm}
\nu_u(E_{10})=\beta u_1+O(u^2).
\end{equation}
An entirely symmetric argument gives
\begin{equation}\label{eq:nu-E01-asympt-bm}
\nu_u(E_{01})=\beta u_2+O(u^2).
\end{equation}
Combining \eqref{eq:sigma-asympt-bm} with \eqref{eq:nu-E10-asympt-bm}--\eqref{eq:nu-E01-asympt-bm} gives
\[
|\nu_u(E_{10})-\sigma(E_{10})|=O(u^2),\qquad |\nu_u(E_{01})-\sigma(E_{01})|=O(u^2),
\]
which is exactly \textup{(iv(a))}.
\end{proof}

The following lemma provides the crucial bridge between the short-time restriction law and the reference law at the smaller scale. By conditioning on the anchor and analyzing the exponentially tilted last-zero distribution, it shows that the discrepancy between these laws is of order $O(u)$.

\begin{lemma}
\label{lem:bm-bridge-lastzero}
Fix $s,t>0$ and set $u=\lambda(s+t)$, $u_1=\lambda s$, $u_2=\lambda t$, as in Corollary~\ref{lem:hellinger-small-criterion-cond}.
Write $\nu_r:=\nu_r^{\mathrm{BM},a}$, and $c(x):=\frac{|\ln x|}{x^2}$ for $x\in(0,1)$. Define the $u$--dependent comparison law on the left block
\[
\bar\nu_{u}^{(u_1)}
:=\Law_{\nu_u}\bigl(R_{0,u_1}(Z)\,\big|\,R_{0,u_1}(Z)\neq\varnothing\bigr),
\qquad
\bar\nu_{u_1}:=\nu_{u_1}(\,\cdot\,|\,Z\neq\varnothing).
\]
Then, for all sufficiently small $\lambda>0$, one has
\begin{equation}\label{eq:bridge-TV-bound}
\bigl\|\bar\nu_u^{(u_1)}-\bar\nu_{u_1}\bigr\|_{\mathrm{TV}}
\ \le\
\frac{C}{u_1}\int_0^{u_1}\Bigl(\frac{1}{c(x)\,u}+\frac{1}{1+c(x)(u_1-x)}\Bigr)\,dx,
\end{equation}
 for some constant $C>0$, depending only on the fixed pair $(s,t)$, and not on $\lambda$. 
In particular, the right-hand side is $O(u)$ as $\lambda\downarrow 0$, hence $1-H\!\bigl(\bar\nu_u^{(u_1)},\bar\nu_{u_1}\bigr)=O(u).$
\end{lemma}

\begin{proof}
Throughout the proof we introduce numerical constants $C_1,C_2,\dots$ once and for all and then set
\begin{equation}\label{eq:def-big-C}
C\ :=\ \max\{C_1,C_2,\dots,C_{11}\}.
\end{equation}
With this convention, any estimate proved with a constant $C_i$ automatically holds with $C$ as well.

On $\{Z\neq\varnothing\}$, $R_{0,u_1}(Z)\neq\varnothing$ is equivalent to $\alpha_u(Z)\le u_1$.
Since $\alpha_u\sim \mathrm{Unif}(0,u)$ under $\nu_u(\cdot\mid Z\neq\varnothing)$, conditioning on
$\{\alpha_u\le u_1\}$ yields $\alpha_u\sim\mathrm{Unif}(0,u_1)$ under
$\bar\nu_u^{(u_1)}=\Law_{\nu_u}(R_{0,u_1}(Z)\mid R_{0,u_1}(Z)\neq\varnothing)$.
Likewise $\alpha_{u_1}\sim \mathrm{Unif}(0,u_1)$ under $\bar\nu_{u_1}=\nu_{u_1}(\cdot\mid Z\neq\varnothing)$.

Since $\mathscr C_{u_1}^{\{\ast\}}$ is standard Borel and $\alpha$ is Borel (Lemma~\ref{lem:anchor-borel}),
there exist regular conditional laws (kernels)
\[
\Pi_x^{(u)}:=\Law_{\nu_u}(R_{0,u_1}(Z)\mid \alpha_u=x,\,Z\neq\varnothing),\qquad
\Pi_x^{(u_1)}:=\Law_{\nu_{u_1}}(Z\mid \alpha_{u_1}=x,\,Z\neq\varnothing),
\]
defined for a.e.\ $x\in(0,u_1]$, such that
\[
\bar\nu_u^{(u_1)}=\frac1{u_1}\int_0^{u_1}\Pi_x^{(u)}\,dx,\qquad
\bar\nu_{u_1}=\frac1{u_1}\int_0^{u_1}\Pi_x^{(u_1)}\,dx.
\]
By convexity of total variation under mixtures with the same mixing measure,
\begin{equation}\label{eq:TV-mixture-bridge}
\bigl\|\bar\nu_u^{(u_1)}-\bar\nu_{u_1}\bigr\|_{\mathrm{TV}}
\ \le\ \frac{C_1}{u_1}\int_0^{u_1}\bigl\|\Pi_x^{(u)}-\Pi_x^{(u_1)}\bigr\|_{\mathrm{TV}}\,dx,
\end{equation}
where $C_1:=1.$

Fix $x\in(0,u_1]$ and abbreviate
\begin{equation}\label{eq:Trc-def}
T:=u-x,\qquad r:=u_1-x\in[0,T),\qquad c:=c(x)=\frac{|\ln x|}{x^2}.
\end{equation}
(For $\lambda$ small we have $u_1<1$, hence $x\in(0,1)$ and $c(x)$ is well-defined.)

Under the original Brownian zero--set law $\mu_u^{\mathrm{BM},a}$, conditioning on $\{\alpha_u=x\}$ means that the Brownian motion first hits $0$ at time $x$;
by the strong Markov property, the post-$x$ process is a standard Brownian motion $W$ started at $0$, independent of the past.
Thus the shifted post-anchor zero set is
\[
\mathcal Z_T:=\{s\in[0,T]:W_s=0\}\in\mathscr C_T^{\{\ast\}}.
\]
Moreover, on the fiber $\{\alpha_u=x\}$ the spread satisfies
\[
\mathrm{diam}_u(Z)=\sup Z-\inf Z=\sup(x+\mathcal Z_T)-x=\sup(\mathcal Z_T)=:G_T,
\]
where $G_T:=\sup\{s\le T:W_s=0\}$ is the last zero time before $T$.

By anchor-adapted localization \eqref{eq:loc-tilt-anchor}, on the fiber $\{\alpha=x\}$  the Radon--Nikodym weight is
$\exp(-c\,\mathrm{diam}_u)=\exp(-c\,G_T)$. Palm uniformization multiplies only by a function of $\alpha$ and therefore does not change the conditional law given $\alpha=x$.
Hence, under $\Pi_x^{(u)}$ the law of $G_T$ is the exponential tilt of the classical arcsine law:
\begin{equation}\label{eq:lastzero-tilt-density-T-bridge}
\mu_{T,c}(dg):=\Law_{\Pi_x^{(u)}}(G_T\in dg)
=\frac{e^{-cg}}{Z_T(c)}\,
\frac{1}{\pi\sqrt{g(T-g)}}\,\mathbf 1_{(0,T)}(g)\,dg,
\qquad
Z_T(c):=\int_0^T e^{-cg}\frac{dg}{\pi\sqrt{g(T-g)}}.
\end{equation}
Here we used the classical arcsine density for $G_T$ under Brownian motion (L\'evy's arcsine law),
see e.g.\ \cite{RevuzYor1999}. Similarly, under $\Pi_x^{(u_1)}$ the post-anchor horizon equals $r$, and the corresponding last zero time has law $\mu_{r,c}$
given by \eqref{eq:lastzero-tilt-density-T-bridge} with $T$ replaced by $r$.

Next, consider the restriction of the post--anchor zero set to $[0,r]$:
\[
\mathcal Z_T^{(r)}:=\mathcal Z_T\cap[0,r]\in\mathscr C_r^{\{\ast\}}.
\]
Under the un-tilted Brownian law, conditional on $G_T=g$, the segment $(W_s)_{0\le s\le g}$
is a Brownian bridge of length $g$ from $0$ to $0$, while the segment $(|W_{g+s}|)_{0\le s\le T-g}$ is a
Brownian meander of length $T-g$ and hence has no zeros on $(0,T-g]$
(Denisov/last-exit decomposition; see, e.g., \cite[Ch.\ XII]{RevuzYor1999}).
In particular, on $\{G_T=g\le r\}$ the restricted zero set $\mathcal Z_T^{(r)}$
has zeros only in $[0,g]$ and is empty on $(g,r]$; moreover, its conditional law depends only on $g$
through the Brownian bridge of length $g$, and not on $T$.
Accordingly, for $g\in(0,r]$ let $K_g$ denote the law on $\mathscr C_r^{\{\ast\}}$ of the zero set of a Brownian bridge of length $g$ on $[0,g]$, viewed as a closed subset of $[0,r]$ by adding no points on $(g,r]$.
Then, for every $T\ge r$, the map $g\mapsto K_g$ is a version of the regular conditional law of $\mathcal Z_T^{(r)}$ given $G_T=g$, that is,
\begin{equation}\label{eq:Kg-def}
\Law(\mathcal Z_T^{(r)}\,|\,G_T=g)=K_g
\qquad\text{for a.e.\ }g\in(0,r].
\end{equation}
Since the exponential tilt defining $\mu_{T,c}$ depends only on $G_T$, the same conditional law $K_g$
also serves as a version of the conditional law of $\mathcal Z_T^{(r)}$ given $G_T=g$ under the tilted law. Consequently, under the tilted laws the restricted-set laws admit the mixture representation
\begin{equation}\label{eq:mixture-rep}
\Pi_x^{(u)} \;=\;\int_{(0,r]} K_g\,\mu_{T,c}(dg)\;+\;\mathrm{Err}_{x},
\qquad
\Pi_x^{(u_1)} \;=\;\int_{(0,r]} K_g\,\mu_{r,c}(dg),
\end{equation}
where $\mathrm{Err}_{x}$ is a subprobability measure of total mass
$\mathrm{Err}_x(\mathscr C_r)=\mu_{T,c}(G_T>r)$ (coming from the event $\{G_T>r\}$, on which \eqref{eq:Kg-def} does not apply).

Since $K_g$ is the same in both mixtures and Markov kernels contract total variation, we have
\begin{equation}\label{eq:fiber-TV-reduce-bridge}
\bigl\|\Pi_x^{(u)}-\Pi_x^{(u_1)}\bigr\|_{\mathrm{TV}}
\ \le\
C_2\Bigl(\mu_{T,c}(G_T>r)\;+\;\bigl\|\mu_{T,c}\!\restriction_{(0,r]}-\mu_{r,c}\bigr\|_{\mathrm{TV}}\Bigr),
\end{equation}
where we again fix $C_2:=1.$

Next, we introduce the Gamma$(\frac12,\text{rate }c)$ density
\[
w_c(g):=\sqrt{\frac{c}{\pi}}\,g^{-1/2}e^{-cg},\qquad g>0,
\]
and write, for $S>0$,
\[
b_S(g):=\Bigl(1-\frac{g}{S}\Bigr)^{-1/2},\qquad g\in(0,S).
\]
Using $\frac{1}{\sqrt{S-g}}=\frac{1}{\sqrt S}\,b_S(g)$,
\eqref{eq:lastzero-tilt-density-T-bridge} can be rewritten as
\begin{equation}\label{eq:muSc-as-weighted-wc}
\mu_{S,c}(dg)
=\frac{b_S(g)\,\mathbf 1_{(0,S)}(g)}{\EE_{w_c}\!\bigl[b_S(G)\,\mathbf 1_{\{G<S\}}\bigr]}\,w_c(g)\,dg.
\end{equation}

We establish a quantitative bound for $\|\mu_{S,c}-w_c\|_{\mathrm{TV}}$. For this, 
fix $C_3:=2,$ $C_4:=6,$ $C_5:=7.$
We claim that for all $S,c>0$,
\begin{equation}\label{eq:muSc-vs-wc}
\bigl\|\mu_{S,c}-w_c\bigr\|_{\mathrm{TV}}
\ \le\ \min\Bigl\{1,\ \frac{C_4}{cS}\Bigr\}
\ \le\ \frac{C_5}{1+cS}.
\end{equation}

Indeed, if $cS<2$ then $\|\mu_{S,c}-w_c\|_{\mathrm{TV}}\le 1\le \frac{C_5}{1+cS}$ because $C_5\ge 3$.
Assume now $cS\ge 2$.
Using \eqref{eq:muSc-as-weighted-wc} and the standard identity \eqref{eq:st-ident}, we get
\begin{equation}\label{eq:tv-basic}
\|\mu_{S,c}-w_c\|_{\mathrm{TV}}
\le
\frac{1}{\EE_{w_c}[b_S(G)\mathbf 1_{\{G<S\}}]}\,
\EE_{w_c}\!\Bigl[\bigl|b_S(G)\mathbf 1_{\{G<S\}}-1\bigr|\Bigr].
\end{equation}
Since $b_S\ge 1$, the denominator is at least $\PP_{w_c}(G<S)$.
But $\EE_{w_c}[G]=\frac1{2c}$, so by Markov's inequality
$\PP_{w_c}(G\ge S)\le \frac{\EE[G]}{S}=\frac{1}{2cS}\le \frac14$, hence
\begin{equation}\label{eq:denom-lower}
\EE_{w_c}[b_S(G)\mathbf 1_{\{G<S\}}]\ \ge\ \PP_{w_c}(G<S)\ \ge\ \frac34.
\end{equation}

For the numerator, split according to $\{G\le S/2\}$ and $\{G>S/2\}$.
On $[0,S/2]$, by the mean value theorem and $b_S'(g)=\frac{1}{2S}(1-g/S)^{-3/2}\le \frac{\sqrt2}{S}\le\frac{C_3}{S}$,
\begin{equation}\label{eq:bS-linear}
|b_S(g)-1|\le \frac{C_3}{S}\,g\qquad(0\le g\le S/2).
\end{equation}
Thus
\begin{equation}\label{eq:num1}
\EE_{w_c}\!\Bigl[|b_S(G)-1|\,\mathbf 1_{\{G\le S/2\}}\Bigr]
\le \frac{C_3}{S}\EE_{w_c}[G]
=\frac{C_3}{2cS}.
\end{equation}

On $\{G>S/2\}$ we use the elementary bound $|b_S(G)\mathbf 1_{\{G<S\}}-1|
\le
b_S(G)\mathbf 1_{\{G<S\}}+\mathbf 1_{\{G\ge S\}}.$
It therefore remains to control $\EE_{w_c}\!\bigl[b_S(G)\mathbf 1_{\{S/2<G<S\}}\bigr].$
By the definition of $w_c$ and $b_S$,
\begin{align*}
\EE_{w_c}\!\bigl[b_S(G)\mathbf 1_{\{S/2<G<S\}}\bigr]
&=
\sqrt{\frac{c}{\pi}}
\int_{S/2}^{S}
g^{-1/2}e^{-cg}\Bigl(1-\frac{g}{S}\Bigr)^{-1/2}\,dg.
\end{align*}
For $g\in[S/2,S]$ we have $g^{-1/2}\le (S/2)^{-1/2},$
and $e^{-cg}\le e^{-cS/2},$
hence
\begin{align*}
\EE_{w_c}\!\bigl[b_S(G)\mathbf 1_{\{S/2<G<S\}}\bigr]
&\le
\sqrt{\frac{c}{\pi}}\,(S/2)^{-1/2}e^{-cS/2}
\int_{S/2}^{S}\Bigl(1-\frac{g}{S}\Bigr)^{-1/2}\,dg&=
\frac{2}{\sqrt\pi}\,\sqrt{cS}\,e^{-cS/2}.
\end{align*}
Now put $y:=cS$. Since $y\ge2$, the function $\phi(y):=y^{3/2}e^{-y/2}$
attains its maximum on $[2,\infty)$ at $y=3$, where $\phi(3)=3^{3/2}e^{-3/2}<2.$
Therefore, for all $y\ge2$,
$\sqrt y\,e^{-y/2}\le \frac{2}{y}.$
Applying this with $y=cS$ yields
\begin{equation}\label{eq:num2}
\EE_{w_c}\!\bigl[b_S(G)\mathbf 1_{\{S/2<G<S\}}\bigr]
\le
\frac{2}{\sqrt\pi}\cdot \frac{2}{cS}
=
\frac{4}{\sqrt\pi}\,\frac{1}{cS}.\end{equation}

Also $\PP_{w_c}(G\ge S)\le \frac{1}{2cS}$ as above.
Combining, we get
\begin{equation}\label{eq:numerator-total}
\EE_{w_c}\!\Bigl[\bigl|b_S(G)\mathbf 1_{\{G<S\}}-1\bigr|\Bigr]
\le \Bigl(\frac{C_3}{2}+\frac{4}{\sqrt\pi}+\frac12\Bigr)\frac{1}{cS}
\le \frac{9}{2}\cdot\frac{1}{cS}.
\end{equation}
Insert \eqref{eq:denom-lower} and \eqref{eq:numerator-total} into \eqref{eq:tv-basic} to obtain
$\|\mu_{S,c}-w_c\|_{\mathrm{TV}}\le \frac{4}{3}\cdot\frac{9}{2}\cdot\frac{1}{cS}\le \frac{C_4}{cS}$,
and \eqref{eq:muSc-vs-wc} follows (the second inequality is an elementary check).

By triangle inequality and \eqref{eq:muSc-vs-wc},
\begin{equation}\label{eq:muTc-murc-TV}
\bigl\|\mu_{T,c}-\mu_{r,c}\bigr\|_{\mathrm{TV}}
\le
\bigl\|\mu_{T,c}-w_c\bigr\|_{\mathrm{TV}}+\bigl\|\mu_{r,c}-w_c\bigr\|_{\mathrm{TV}}
\le
\frac{C_5}{cT}+\frac{C_5}{1+cr}.
\end{equation}
Moreover, since $G_T\le T$ always,
\begin{equation}\label{eq:EG-bound}
\EE_{\mu_{T,c}}[G_T]
\le \EE_{w_c}[G] + T\,\|\mu_{T,c}-w_c\|_{\mathrm{TV}}
\le \frac{1}{2c}+\frac{C_5}{c}
\le \frac{C_6}{c},
\qquad C_6:=8.
\end{equation}
Therefore, by Markov's inequality and the elementary bound $\min\{1,a/y\}\le (1+a)/(1+y)$,
\begin{equation}\label{eq:tail-bound}
\mu_{T,c}(G_T>r)\ \le\ \min\Bigl\{1,\frac{C_6}{cr}\Bigr\}
\ \le\ \frac{C_7}{1+cr},
\qquad C_7:=1+C_6=9.
\end{equation}

To complete the proof, we first notice that
\begin{equation}\label{eq:truncate}
\bigl\|\mu_{T,c}\!\restriction_{(0,r]}-\mu_{r,c}\bigr\|_{\mathrm{TV}}
\le
\bigl\|\mu_{T,c}-\mu_{r,c}\bigr\|_{\mathrm{TV}}+\mu_{T,c}(G_T>r),
\end{equation}
since removing the tail $(r,T)$ from $\mu_{T,c}$ changes total variation by exactly its mass.
Combining \eqref{eq:fiber-TV-reduce-bridge}, \eqref{eq:muTc-murc-TV}, \eqref{eq:tail-bound} and \eqref{eq:truncate} gives
\begin{equation}\label{eq:fiber-final-bridge}
\bigl\|\Pi_x^{(u)}-\Pi_x^{(u_1)}\bigr\|_{\mathrm{TV}}
\le
\frac{C_9}{cT}+\frac{C_9}{1+cr},
\qquad C_9:=25.
\end{equation}

Insert \eqref{eq:fiber-final-bridge} into \eqref{eq:TV-mixture-bridge} and use $C_1=1$ to obtain
\begin{equation}\label{eq:pre-final-integral}
\bigl\|\bar\nu_u^{(u_1)}-\bar\nu_{u_1}\bigr\|_{\mathrm{TV}}
\le
\frac{C_1C_9}{u_1}\int_0^{u_1}\Bigl(\frac{1}{c(x)\,(u-x)}+\frac{1}{1+c(x)(u_1-x)}\Bigr)\,dx.
\end{equation}
Since $x\in(0,u_1]$ implies $u-x\ge u-u_1=u_2$, we have
\begin{equation}\label{eq:ux-lower}
\frac{1}{u-x}\le \frac{1}{u_2}=\frac{s+t}{t}\cdot \frac{1}{u}.
\end{equation}
Define $C_{10}:=\frac{s+t}{t},$ $C_{11}:=C_1C_9C_{10}.$
Then \eqref{eq:pre-final-integral} and \eqref{eq:ux-lower} imply the claimed bound \eqref{eq:bridge-TV-bound}
with the constant $C$ from \eqref{eq:def-big-C} (note that $C\ge C_{11}$ by definition).

It remains to establish the Hellinger bound and the $O(u)$ estimate.
For probability measures $\rho,\eta$ one has the standard inequality
$1-H(\rho,\eta)\le \|\rho-\eta\|_{\mathrm{TV}}$. Hence the Hellinger conclusion follows from \eqref{eq:bridge-TV-bound}.

Finally, to see that the right-hand side of \eqref{eq:bridge-TV-bound} is $O(u)$ for $c(x)=|\ln x|/x^2$,
note that $|\ln x|\ge |\ln u_1|$ for $x\in(0,u_1]$, so
\begin{align*}
\frac1{u_1}\int_0^{u_1}\frac{1}{c(x)\,u}\,dx
&=\frac{1}{u_1u}\int_0^{u_1}\frac{x^2}{|\ln x|}\,dx
\le \frac{1}{u_1u}\cdot\frac{1}{|\ln u_1|}\int_0^{u_1}x^2\,dx
=\frac{u_1^2}{3u|\ln u_1|}
=O\!\Bigl(\frac{u}{|\ln u|}\Bigr),
\\
\frac1{u_1}\int_0^{u_1}\frac{1}{1+c(x)(u_1-x)}\,dx
&\le \frac1{u_1}\int_0^{u_1}\frac{1}{1+c(u_1)(u_1-x)}\,dx
=\frac{1}{u_1c(u_1)}\ln\bigl(1+c(u_1)u_1\bigr)
\\
&=\frac{u_1}{|\ln u_1|}\,\ln\Bigl(1+\frac{|\ln u_1|}{u_1}\Bigr)
=O(u_1)=O(u),
\end{align*}
as $\lambda\downarrow0$ (for fixed $s,t$). 
This completes the proof.
\end{proof}

The importance of Lemma~\ref{lem:bm-bridge-lastzero} is that it reduces the
verification of condition~\textup{(iv(b))} in
Corollary~\ref{lem:hellinger-small-criterion-cond} to a stability estimate
for the one-block restriction law. This step is inherently model-dependent,
as it relies on specific structural properties of the Brownian zero set,
notably the Brownian bridge and meander decomposition at the last zero.
Combined with Lemma~\ref{lem:bm-one-block-stability}, it yields the required
$O(u)$ bound on the conditional one-block discrepancy.

\begin{proposition} The Brownian seed $\boldsymbol\nu^{\mathrm{BM},a}$ satisfies condition \textup{(iv(b))} of
Corollary~\ref{lem:hellinger-small-criterion-cond}.
\end{proposition}

\begin{proof}
Fix $s,t>0$. Set $u=\lambda(s+t),$ $u_1=\lambda s,$ $u_2=\lambda t,$ and write $\nu_r:=\nu_r^{\mathrm{BM},a}$ as before. Let $\widehat\nu_u^{10}$ (resp.\ $\widehat\nu_u^{01}$) be the conditional restriction laws in
Corollary~\ref{lem:hellinger-small-criterion-cond}, and write
\[
\bar\nu_{u_1}:=\nu_{u_1}(\,\cdot\,|\,Z\neq\varnothing),\qquad
\bar\nu_{u_2}:=\nu_{u_2}(\,\cdot\,|\,Z\neq\varnothing).
\]
We show that,  as $\lambda\downarrow0$,
\[
1-H\!\bigl(\widehat\nu_u^{10},\bar\nu_{u_1}\bigr)=O(u),
\qquad
1-H\!\bigl(\widehat\nu_u^{01},\bar\nu_{u_2}\bigr)=O(u),
\]
i.e.\ condition \textup{(iv(b))} of Corollary~\ref{lem:hellinger-small-criterion-cond} holds for
$\boldsymbol\nu^{\mathrm{BM},a}$.
Since $1-H(\rho,\eta)\ \le\ \|\rho-\eta\|_{\mathrm{TV}},$ for any two probability measures $\rho,\eta$, it suffices to prove $\|\widehat\nu_u^{10}-\bar\nu_{u_1}\|_{\mathrm{TV}}=O(u)$ and similarly on the right. 

For this, consider the comparison laws (depending on $u$) introduced in Lemma~\ref{lem:bm-bridge-lastzero}:
\[
\bar\nu_{u}^{(u_1)}
:=\Law_{\nu_u}\bigl(R_{0,u_1}(Z)\ \big|\ R_{0,u_1}(Z)\neq\varnothing\bigr),
\qquad
\bar\nu_{u}^{(u_2)}
:=\Law_{\nu_u}\bigl(R_{u_1,u}(Z)\ \big|\ R_{u_1,u}(Z)\neq\varnothing\bigr).
\]
By Lemma~\ref{lem:bm-one-block-stability},
using anchor-uniformity and the $u^2$ diameter tail for $\nu_u$,
there exists $C^{\mathrm{stab}}_{s,t}<\infty$ such that for all sufficiently small $\lambda$,
\begin{equation}\label{eq:stab-left}
\bigl\|\widehat\nu_u^{10}-\bar\nu_{u}^{(u_1)}\bigr\|_{\mathrm{TV}}\le C^{\mathrm{stab}}_{s,t}\,u,
\qquad
\bigl\|\widehat\nu_u^{01}-\bar\nu_{u}^{(u_2)}\bigr\|_{\mathrm{TV}}\le C^{\mathrm{stab}}_{s,t}\,u.
\end{equation}
By Lemma~\ref{lem:bm-bridge-lastzero}, 
\begin{equation}\label{eq:bridge-left-Ou}
\bigl\|\bar\nu_{u}^{(u_1)}-\bar\nu_{u_1}\bigr\|_{\mathrm{TV}}=O(u)\quad\text{and}\quad \bigl\|\bar\nu_{u}^{(u_2)}-\bar\nu_{u_2}\bigr\|_{\mathrm{TV}}=O(u).
\end{equation}
By triangle inequality, combining \eqref{eq:stab-left} with \eqref{eq:bridge-left-Ou},
\[
\bigl\|\widehat\nu_u^{10}-\bar\nu_{u_1}\bigr\|_{\mathrm{TV}}
\ \le\
\bigl\|\widehat\nu_u^{10}-\bar\nu_{u}^{(u_1)}\bigr\|_{\mathrm{TV}}
+
\bigl\|\bar\nu_{u}^{(u_1)}-\bar\nu_{u_1}\bigr\|_{\mathrm{TV}}
=O(u),
\]
and similarly, 
$\|\widehat\nu_u^{01}-\bar\nu_{u_2}\|_{\mathrm{TV}}=O(u)$. The proof is now complete.
\end{proof}

Combining the previous estimates with the abstract small-seed criterion and the infinite-product construction, we obtain the main Brownian example of a type~$\mathrm{III}$ random-set product system.

\begin{theorem}\label{th:brownian-typeIII}
Fix $a>0$. If $\boldsymbol\mu_{\infty}^{\mathrm{BM},a}$ is the measurable factorizing family over $[0,1]\times\N$, constructed from the seed $\boldsymbol\nu^{\mathrm{BM},a}$ as in Proposition \ref{th:meas-seed} (take $a_n=1/n$), then the  random-set system $\mathbb E^{\boldsymbol\mu_{\infty}^{\mathrm{BM},a}}$ is  type $\mathrm{III}$.
\end{theorem}
\begin{observation}
The same construction applies if we start with a Bessel $\mathrm{BES}^{(\delta)}$ process 
$\{R_t\}_{t\ge0}$ (with $\delta\in(0,2)$), started at $R_0=a$, on a filtered space 
$(\Omega,\mathcal F,(\mathcal F_t),\mathbb P_a)$, and consider the random closed set
\begin{equation}\label{eq:bessel}
Z_t^{\mathrm{Bes}}:=\{u\in[0,t]: R_u=0\}\in\mathscr C_t^{\{\ast\}},
\qquad
\mu_t^{\mathrm{Bes}}:=\Law(Z_t^{\mathrm{Bes}}).
\end{equation}
Then the seed 
$\boldsymbol\mu^{\mathrm{Bes}(\delta)}=\{\mu_t^{\mathrm{Bes}}\}_{t>0}$ 
yields, as above, a measurable factorizing family over $[0,1]\times\N$, 
denoted 
$\boldsymbol\mu_{\infty}^{\mathrm{Bes}(\delta)}$, 
whose associated random-set system 
$\mathbb E^{\boldsymbol\mu_{\infty}^{\mathrm{Bes}(\delta)}}$ 
is of type $\mathrm{III}$.

Whether the map 
$\delta\mapsto \mathbb E^{\boldsymbol\mu_{\infty}^{\mathrm{Bes}(\delta)}}$ 
is injective up to isomorphism is a subtle question requiring additional analysis. We defer this issue to future work.
\end{observation}


\end{document}